\documentclass{article}

\usepackage[english]{babel}

\usepackage[a4paper,top=2cm,bottom=2cm,left=3cm,right=3cm,marginparwidth=1.75cm]{geometry}

\usepackage{amsmath,amsfonts,mathtools,amsthm}
\usepackage{mathabx}
\usepackage{graphicx}
\usepackage{appendix}
\usepackage{enumitem}
\newcommand{\commentout}[1]{}
\usepackage[colorlinks=true, allcolors=blue]{hyperref}
\usepackage{amsthm}
\newcommand{\R}{\mathbb R}
\newcommand{\loc}{{\text{loc}}}
\newcommand{\sub}{{\text{sub}}}
\newcommand{\super}{\text{super}}

\newcommand{\calT}{{\mathcal{T}}}
\newcommand{\calU}{{\mathcal{U}}}

\newcommand{\ra}{{\rightarrow}}

\newcommand{\calH}{{\mathcal H}}
\newcommand{\calA}{{\mathcal A}}
\newcommand{\Rm}{{\mathbb R}}

\newcommand{\Pm}{{\mathbb P}}

\newcommand{\N}{{\mathbb N}}
\newcommand{\expE}{{\mathbb E}}
\newcommand{\Leb}{{\text{Leb}}}
\newcommand{\Ind}{{{{1}}}}
\newcommand{\calF}{{\mathcal F}}
\newcommand{\calG}{{\mathcal G}}

\newtheorem{theo}{Theorem}[section]
\newtheorem{lem}[theo]{Lemma}
\newtheorem{defin}[theo]{Definition}

\newtheorem{prop}[theo]{Proposition}
\newtheorem{cor}[theo]{Corollary}
\newtheorem{rmk}[theo]{Remark}

\newtheorem{assum}[theo]{Assumption}

\newtheorem{example}[theo]{Example}

\newcommand{\glSurv}{\ensuremath{\mathrm{glSurv}}}
\newcommand{\locSurv}{\ensuremath{\mathrm{locSurv}}}
\newcommand{\rlocSurv}{\ensuremath{\mathrm{r{-}locSurv}}}

\title{Generalised principal eigenvalues and global survival of branching Markov processes}
\author{Pascal Maillard\thanks{Institut de Mathématiques de Toulouse, CNRS UMR 5219, Université de Toulouse, 118 route de Narbonne, 31062 Toulouse Cedex 9, France and Institut Universitaire de France. E-mail: \texttt{pascal DOT maillard AT math DOT univ-toulouse DOT fr}}, Oliver Tough\thanks{Department of Mathematical Sciences, Durham University, Upper Mountjoy Campus, Stockton Rd, Durham DH1 3LE. E-mail: \texttt{fjwd57@durham.ac.uk}}}

\begin{document}
\maketitle

\begin{abstract}
We study necessary and sufficient criteria for global survival of discrete or continuous-time branching Markov processes. 
We relate these to several definitions of generalised principle eigenvalues for elliptic operators due to Berestycki and Rossi. In doing so, we extend these notions to fairly general semigroups of linear positive operators. We use this relation to prove new results about the generalised principle eigenvalues, as well as about uniqueness and non-uniqueness of stationary solutions of a generalised FKPP equation. The probabilistic approach through branching processes gives rise to relatively simple and transparent proofs under much more general assumptions, as well as constructions of (counter-)examples to certain conjectures.
\end{abstract}

\textbf{This is an unpolished draft. Comments are welcome.}

\section{Introduction}

Branching processes are a staple of probability theory, whose study goes back to the 19th century, when Bienaymé \cite{bienayme1845loi} introduced the model and understood that the a (single-type) branching process survives with positive probability if and only if the mean offspring is greater than 1, see Kendall \cite{Kendall1975} for a historical account. Branching processes have since found applications into a wide array of areas, including, but not limited to, random graphs, quantum gravity, turbulence, front propagation, disordered systems, stochastic geometry, nuclear fission, time series, and of course to population biology, where branching processes originate from.

In this paper, we consider general multitype branching processes, or branching Markov processes, in discrete or continuous time. This includes a variety of processes, such as branching processes with a finite or countable type space, branching random walks on finite or infinite graphs or branching diffusions. We are interested in determining whether the process survives with positive probability, starting from a single particle. This is well-understood in the case of a finite number of types, as well as for branching diffusions on bounded domains. In these cases, classic results provide that survival with positive probability is characterized in terms of the principal eigenvalue of a certain operator: the mean matrix of the offspring distribution in the first case \cite{Harris1963}, and the generator of the first moment semigroup in the second case \cite{Hering1978b}. 

It may seem surprising that in the general case, i.e.~in the case of infinite, possibly uncountable type space, no general criterion for survival seems to be known to date. The literature seems to be entirely focused on criteria that guarantee that the process can be approximated by the process killed outside a finite, or bounded subset of the type space \cite{KestenSupercriticalBranching1989,BraunsteinsDecrouezHautphennePathwiseApproach2019}, or under which laws of large numbers hold \cite{AH1983,Englander2010,WangStrongLaw2015,Marguet2019}. This is related to the question of \emph{local survival}. In a locally compact state space, we say that the process survives locally, if every sufficiently large compact set is visited infinitely often with positive probability. A necessary and sufficient criterion for local survival is known in two cases: branching processes in discrete time with countable state space and branching diffusions. In these cases, local survival is characterized through a certain eigenvalue: the \emph{spectral radius} of the (infinite) mean matrix in the case of countable state space \cite{MullerCriterionTransience2008}, and the so-called \emph{generalised principal eigenvalue} in the branching diffusion case \cite{Englander2004}. 

In this paper, we address for the first time the question of \emph{global survival} with positive probability of a general branching Markov process, starting from a single particle. Note that, in a certain sense, this notion is more fundamental than the notion of local survival, as it does not rely on a notion of locality, which does not make sense in arbitrary measurable spaces. We draw inspiration from the PDE literature, and notably from a seminal paper by Berestycki and Rossi~\cite{BerestyckiRossi}. In that paper, the authors define two new notions of generalised principal eigenvalues for elliptic second-order differential operators, called $\lambda_1'$ and $\lambda_1''$. In the current paper, we show that one of these ($\lambda_1'$) two generalised principal eigenvalues precisely governs global survival for branching Markov processes, whereas the other one ($\lambda_1''$) governs a weaker notion of global survival: global survival \emph{in expectation}.

\paragraph{Main contributions.}
The main contributions of the current paper are as follows:

\begin{itemize}
    \item 
We generalize the generalised principal eigenvalues $\lambda_1'$ and $\lambda_1''$ from Berestycki and Rossi~\cite{BerestyckiRossi} to quite general semigroups of positive operators, in discrete or continuous-time. We call these two quantities $\lambda_c'$ and $\lambda_c''$. We show that they indeed generalize $\lambda_1'$ and $\lambda_1''$ (or rather, $-\lambda_1'$ and $-\lambda_1''$).

\item
We show that, under suitable assumptions, $\lambda_c'>0$ implies global survival of the branching process with positive probability, whereas $\lambda_c'<0$ implies global extinction. In continuous time, we provide an equivalent formulation in terms of an associated reaction-diffusion type equation. 

\item 
We provide an expression of $\lambda_c''$ in terms of the asymptotic growth of the expected number of particles. This expression seems to be new even in the PDE literature. From this expression, we deduce that $\lambda_c''>0$ implies global survival in expectation, whereas $\lambda_c''<0$ implies global extinction in expectation.

\item We study the relation between $\lambda_c'$ and $\lambda_c''$. Namely, we show that $\lambda_c' \le \lambda_c''$ and provide examples for $\lambda_c' < \lambda_c''$. In particular, we refute a conjecture by Berestycki and Rossi (\cite[Conjecture 1.8]{BerestyckiRossi}) which states that $\lambda_1' = \lambda_1''$ under mild assumptions. Whereas we show that the conjecture is not true in general, we prove it holds for self-adjoint $L$. This provides a sharp characterisation of the validity of the maximum principle for self-adjoint elliptic operators on unbounded domains.

\item In the case of discrete time and countable state space or of branching diffusions, we show that $\lambda_c \le \lambda_c'$ and provide a relatively general sufficient criterion for $\lambda_c = \lambda_c' = \lambda_c''$, which significantly generalizes a criterion by Berestycki and Rossi.

\item We connect the existence and uniqueness of stationary solutions of the generalised FKPP equation with the global survival and local survival of the corresponding branching process. In particular, we show that if the branching process undergoes local survival with positive probability, then the corresponding stationary solution of the FKPP equation is unique if and only if it is impossible for the branching process to survive by ``drifting off to infinity''. We use this to provide a short probabilistic proof of uniqueness of stationary solutions of the FKPP equation on a bounded domain without any assumptions on the boundary. This proves a conjecture of Berestycki and Graham \cite[Conjecture 4.2]{Berestycki2025}, who conjectured it for Lipschitz boundary, and extends a classical result of Rabinowitz who proved it for smooth boundary \cite{Rabinowitz1971}. We then demonstrate the applicability of this criterion on a number of examples where the domain is unbounded.
  
\item
We study the critical cases $\lambda_c'=0$ and $\lambda_c'' = 0$, providing examples that show that all scenarios are possible.


\end{itemize}

We believe that our results have far-reaching consequences both in probability and PDE theory. One the one hand, we provide for a new and general criterion for global survival for branching Markov processes, which has potential to be applied in many specific instances. On the other hand, we provide a new interpretation of the Berestycki-Rossi generalised principal eigenvalues $\lambda_1'$ and $\lambda_1''$ and prove new results about them, greatly advancing their understanding. We also point out that we work in a general setting encompassing both discrete and continuous space or time, benefitting from the powerful general theory of Markov processes.

In future work, we aim to study the question of \emph{local survival} in general, extending the known results valid for countable type spaces and for branching diffusions to more general spaces and processes.

\paragraph{Overview of the paper.}
In Section~\ref{sec:results_discrete_time}, we present our main results in the discrete time setting. The discreteness of time removes some technicalities from the continuous-time setting and provides for relatively simple proofs. Section~\ref{sec:results_continuous_time} establishes the analogous results in the continuous-time setting. Section~\ref{sec:results_branching_diffusions} addresses specifically branching diffusions, establishing the relation with the generalised principal eigenvalues of Berestycki and Rossi. The last Section~\ref{sec:examples} is devoted to various examples.

\section{Discrete-time results}
\label{sec:results_discrete_time}

We consider a discrete-time branching Markov process as in Jagers~\cite{Jagers1989} or Biggins and Kyprianou~\cite{BK2004}. That is, let $E$ be a measurable space, $\partial$ a cemetery symbol, and set $\hat E = E\cup\{\partial\}$. We define the spaces of functions
\begin{align*}
  pB(E) &= \{f:E\to [0,\infty]\text{ measurable}\}\\
  pB_b(E) &= \{f\in pB(E):\text{ $f$ bounded}\}\\
  pB_1(E) &= \{f\in pB(E): f\le 1\}.
\end{align*}
Let $\Pi$ be a probability kernel from $E$ to $E^\N$. We extend $\Pi$ to $\hat E$ by setting $\Pi(\partial,\cdot) = \delta_{\partial^{\N}}$.  Let 
\[
  U = \bigcup_{n=0}^\infty U_n,\quad U_n = \N^n,
\]
with $\N^0 = \{\emptyset\}$. Thus, $U$ is the set of words of finite length over the alphabet $\N$. If $u\in U_n$, we set $|u| = n$. We write $u\le v$ if there exists $w\in U$, such that $uw = v$. A (discrete-time) branching Markov chain taking values in $E$, with offspring distribution $\Pi$, is a random field $(X_u)_{u\in U}$ taking values in $\hat E$, such that for all $n\in\N_0$, the collection of random vectors $((X_{ui})_{i\in\N})$, for $u\in U_n$ is iid according to the law $\Pi$.
The number of particles at time $n$ is denoted by
$$
N_n = \sum_{u\in U_n}\Ind_{X_u \ne\partial}.
$$

In what follows, when summing over $i\in\N$ or $u\in U$, we also assume that we only sum over such $i$ or $u$ such that $X_i \ne \partial$ (or $X_u\ne \partial$). The same holds for taking the product. We define accordingly,
\begin{align}
  mf(x) &= \expE_x\left[\sum_{i\in\N} f(X_i)\right],\quad f\in pB(E)\\
  \pi f(x) &= \expE_x\left[\prod_{i\in\N} f(X_i)\right],\quad f\in pB_1(E)
\end{align}
\begin{rmk}
  Whereas $\pi$ characterizes the branching kernel $\Pi$ (modulo reordering of the indices), this is not the case for $m$. For example, for a given $m$, one can define a branching kernel $\Pi$ which corresponds to the law of the atoms of a Poisson point process with intensity measure $m(x,\cdot)$, in some arbitrary order, corresponding to total independence between the children. On the other hand, if $m(x,\cdot)$ is finite, one could also take a random number of children, with expectation $\|m(x,\cdot)\|$, all at the same position distributed according to the law $m(x,\cdot)/\|m(x,\cdot)\|$. If the object of interest is the kernel $m$, rather than the branching Markov chain, then some choices could be better than others.
\end{rmk}

We then define for every $x\in E$,
\begin{align}
  \rho'_{c,x}(m) &= \sup\{\rho\ge 0: \exists u\in pB_b(E): mu \ge \rho u,\ u(x) > 0\},\label{eq:discrete time gen evalue 1}\\ 
  \rho''_{c,x}(m) &= \inf\{\rho\ge 0: \exists u\in pB(E): mu \le \rho u,\ u\ge 1,\ u(x)<\infty\}.
\end{align}

We define the event of \emph{(global) survival} by 
\[
  \glSurv \coloneqq \{N_n \ge 1\,\forall n\ge 0\}.
\]

Our first theorem expresses $\rho_{c,x}'$ and $\rho_{c,x}''$ in terms of the growth of the expected number of particles. In particular, it provides for an explicit expression for $\rho_{c,x}''$.

    \begin{theo}
      \label{th:rho'_rho''}
      For every $x\in E$, we have
      \[
        \rho'_{c,x} \le \liminf_{n\to\infty} \left(\expE_x[N_n]\right)^{1/n} \le \limsup_{n\to\infty} \left(\expE_x[N_n]\right)^{1/n} = \rho_{c,x}''.
      \]
    \end{theo}

    \begin{rmk}
      Both inequalities in Theorem~\ref{th:rho'_rho''} may be strict as the example in Section~\ref{sec:example_mutations} shows.
    \end{rmk}

    The expected number of particles may overestimate the typical number of particles due to rare events accounting for a large number of particles. Our next theorem expresses the typical behavior of the branching process in terms of $\rho'_{c,x}$ and $\rho''_{c,x}$. In particular, it provides for necessary and sufficient conditions for $\Pm_x(\glSurv) > 0$, in terms of $\rho'_{c,x}$. In the case where the offspring distribution is bounded by a constant, it shows that $\rho'_{c,x} > 1$ implies $\Pm_x(\glSurv) > 0$ and $\Pm_x(\glSurv) > 0$ implies $\rho'_{c,x} \ge 1$. In the general case, something slightly strong than $\rho'_{c,x} > 1$ is needed to guarantee $\Pm_x(\glSurv) > 0$, see \eqref{eq:assumption_u_discrete}.

\begin{theo}
  \label{th:rho'_survival}
  The following holds.
  \begin{enumerate}
    \item Let $x\in E$. If $\Pm_x(\glSurv) > 0$, then $\rho'_{c,x}(m) \ge 1$.
    \item Suppose there exists $u\in pB_b(E)$, $\rho\ge 1$ and $\delta>0$, such that 
      \begin{equation}
        \label{eq:assumption_u_discrete}
        \forall x\in E: \expE_x\left[\sum_{i\in\N} u(X_i) \exp\left(-\delta\sum_{i\in\N} u(X_i)\right)\right] \ge \rho u(x).
      \end{equation}
      Then for all $x\in E$, we have that $u(x)>0$ implies
      \[
      \Pm_x(\liminf_{n\to\infty} \rho^{-n}N_n > 0) > 0,
      \]
      and, in particular, $\Pm_x(\glSurv) > 0$.

    \item 
      Suppose the following holds:
      \begin{itemize}
        \item[(B)]
      there exists a constant $K<\infty$ such that $N_1 \le K$, $\Pm_x$-a.s. for all $x\in E$. 
  \end{itemize}
  Then for every $\rho < \rho'_{c,x}(m)$, there exists $u\in pB_b(E)$ and $\delta>0$, such that \eqref{eq:assumption_u_discrete} is verified with this value of $\rho$ and $u(x) > 0$. In particular,
  \[
    \text{$\rho'_{c,x}(m) > 1$ implies $\Pm_x(\glSurv)>0$ and }\forall \rho < \rho'_{c,x}:\Pm_x\left(\liminf_{n\to\infty} N_n^{1/n} \ge \rho\right) > 0.
  \]

    \item We have for every $x\in E$,
      \[
        \Pm_x\left(\limsup_{n\to\infty} N_n^{1/n} \le \rho''_{c,x}\right) = 1.
      \]
  \end{enumerate}
\end{theo}

    %
    
\subsubsection*{Dependence of global survival upon the initial condition}

We now investigate the relation between the values of $\rho'_{c,x}$ and $\rho''_{c,x}$ for different $x\in E$. Define
    \begin{align*}
  \rho'_c(m) &= \sup\{\rho\ge 0: \exists u\in pB_b(E): mu \ge \rho u,\ u(x) > 0\,\forall x\in E\},\\ 
  \rho''_{c}(m) &= \inf\{\rho\ge 0: \exists u\in pB(E): mu \le \rho u,\ 1\le u<\infty\}.
    \end{align*}
    
\begin{prop}
  \label{prop:rho dependence on x}
  We have $\rho'_c(m) = \inf_{x\in E}\rho'_{c,x}(m)$ and $\rho''_c(m) = \sup_{x\in E}\rho''_{c,x}(m)$.
\end{prop}

\begin{cor}
\label{cor:rho}
\begin{enumerate}
    \item We have,
      \[
        \rho'_c \le \inf_{x\in E}\liminf_{n\to\infty} \left(\expE_x[N_n]\right)^{1/n} \le \sup_{x\in E}\limsup_{n\to\infty} \left(\expE_x[N_n]\right)^{1/n} = \rho_c''.
      \]
    \item Suppose there exists $u\in pB(E)$, $0<u\le 1$, $\rho\ge 1$ and $\delta>0$, such that \eqref{eq:assumption_u_discrete} holds.
      Then, for all $x\in E$,
      \[
      \Pm_x(\liminf_{n\to\infty} \rho^{-n}N_n > 0) > 0,
      \]
      and, in particular, $\Pm_x(\glSurv) > 0$.

    \item 
      Suppose the following holds:
      \begin{itemize}
        \item[(B)]
      there exists a constant $K<\infty$ such that $N_1 \le K$, $\Pm_x$-a.s. for all $x\in E$. 
  \end{itemize}
  Then for every $\rho < \rho'_c(m)$, there exists $u\in pB(E)$, $0<u\le 1$ and $\delta>0$, such that \eqref{eq:assumption_u_discrete} is verified with this value of $\rho$. In particular,
  \[
    \text{$\rho'_c(m) > 1$ implies $\Pm_x(\glSurv)>0\,\forall x\in E$ and }\forall \rho < \rho'_c\,\forall x\in E:\Pm_x\left(\liminf_{n\to\infty} N_n^{1/n} \ge \rho\right) > 0.
  \]

    \item We have for every $x\in E$,
      \[
        \Pm_x\left(\limsup_{n\to\infty} N_n^{1/n} \le \rho''_c\right) = 1.
      \]
\end{enumerate}
\end{cor}
    
We now introduce the following Harris-type irreducibility assumption.
\begin{assum}\label{assum:Harris discrete}
There exists a measure $\nu$ on $E$ such that for every measurable set $A$ with $\nu(A)>0$ and every $x\in E$ there exists $n\ge 0$ (possibly dependent upon $x$ and $A$) such that we have $m^n(x,A) > 0$.
\end{assum}
\begin{theo}\label{th:discrete_harris}
Suppose Assumption \ref{assum:Harris discrete} holds, for some given measure $\nu$. Then we have either $\Pm_x(\glSurv)>0$ for all $x\in E$, or otherwise $\Pm_x(\glSurv)=0$ for $\nu$-almost every $x \in E$.
\end{theo}

    \subsubsection*{Additional results for countable state spaces}
    
    In the remainder of the section, we assume that $E$ is finite or countable. The kernel $m$ is then of the form $mf(x) = \sum_y m(x,y)f(y)$, where $m$ is a (possibly infinite) matrix with non-negative entries. We assume the following
    \begin{itemize}
    \item[(I)]
        \label{ass:irreducible}
        The matrix $m$ is irreducible, i.e., for every $x,y\in E$ there exists $n\in\N$, such that $m^n(x,y) >0$.
            \item[(F)]
        \label{ass:finite}
        We have $m^n(x,y) <\infty$ for all $x,y\in E$ and all $n\in \N$.
    \end{itemize}
    These assumptions imply Assumption~\ref{assum:Harris discrete}. In particular, we have $\rho'_{c,x} = \rho'_c$ and $\rho''_{c,x} = \rho''_c$.
    
    We now compare $\rho_c'$ and $\rho_c''$ to the classical notion of the \emph{spectral radius} $\rho_c = \rho_c(m)$. For fixed $x,y\in E$, set
    \begin{equation}
        \label{def:spectral_radius}
        \rho_c = \inf\{\rho\ge 0: \exists u\in pB(E): mu \le \rho u,\ 0< u<\infty\}.
    \end{equation}
    \begin{theo}[see for example Seneta~\cite{SenetaNonNegativeMatrices2006}]
    \label{th:spectral_radius}
    We have the following equivalent definitions of $\rho_c$:
    \begin{enumerate}
        \item For every $x,y\in E$, \[
        \rho_c = \limsup_{n\to\infty} \left(m^n(x,y)\right)^{1/n}.
        \]
        \item For a subset $F\subset E$, define
        \(
        m_F(x,y) = m(x,y)\Ind_{x\in F,\ y\in F}
        \). Then 
        \[
        \rho_c = \sup_{F\subset E,\ F\text{ finite},\ m_F\text{ irreducible}} \rho_c(m_F).
        \]
    \end{enumerate}
    \end{theo}
    \begin{proof}
      The first part is Theorem~6.3 in Seneta~\cite{SenetaNonNegativeMatrices2006}. The second part is essentially Theorem~6.8 in Seneta~\cite{SenetaNonNegativeMatrices2006}, although Seneta works under the assumption that $m_F$ is irreducible for all but a finite number of $F$. However, the stronger statement as written here still holds, see e.g.~the proof of Theorem 2.4 in Müller~\cite{MullerCriterionTransience2008}.
    \end{proof}
  The spectral radius has an interpretation in terms of \emph{local} survival or extinction of the branching Markov chain. Define the events
    \[
      \locSurv = \bigcap_{x\in E} \locSurv_x,\quad \locSurv_x \coloneqq \left\{\sum_{|u|=n} \Ind_{X_u=x} > 0\text{ infinitely often}\right\}
    \]
    We call $\locSurv_x$ the event of \emph{local survival at $x$} and $\locSurv$ the event of local survival. Note that for all $x,x',y,y'\in E$, we have $\Pm_x(\locSurv_{x'}) > 0$ if and only if $\Pm_y(\locSurv_{y'}) > 0$, by irreducibility, and therefore $\Pm_x(\locSurv) = \Pm_x(\locSurv_y)$ for every $x,y\in E$. We say that \emph{the BMC survives locally}, if $\Pm_x(\locSurv) > 0$ for some (or all) $x\in E$ and \emph{undergoes local extinction} otherwise. Müller has shown the following result:
    \begin{theo}[Müller~\cite{MullerCriterionTransience2008}]\label{th:muller}
    The BMC survives locally if and only if $\rho_c(m) > 1$.
    \end{theo}
    
    We are now ready to compare $\rho_c$ to $\rho_c'$ and $\rho_c''$.
    It follows directly from the definitions that $\rho_c\le \rho_c''$; this can also be seen from the first part of Theorem~\ref{th:spectral_radius} and the expression of $\rho''_{c,x}$ from Theorem~\ref{th:rho'_rho''}. The next theorem shows that, in fact, $\rho_c\le \rho_c'$.
    \begin{theo}\label{th:rho_c_rho_c'}
      We have $\rho_c\le \rho_c'$.
    \end{theo}
    We will give two proofs of Theorem~\ref{th:rho_c_rho_c'}, one using the branching Markov chain, using the intuitively obvious fact that local survival implies global survival, and a second, more analytic one.
    
    %

    Our final result in this section gives a sufficient condition for $\rho_c = \rho_c' = \rho_c''$. This theorem covers for example branching random walks on bounded degree graphs with subexponential volume growth (such as lattices) and with position-dependent local branching. Note that the case of a finite state space $E$ is covered by classical theory which in particular implies that $\rho_c = \rho_c' = \rho_c''$.
    \begin{theo}
      \label{th:reversible}
      Assume there exists $x\in E$ and a sequence $c_n$ of positive numbers, such that $\limsup_{n\to\infty} c_n^{1/n} \le 0$ and such that for every $y\in E$,
          \[
            m^n(x,y) \le c_n m^n(y,x),
          \]
          and such that $|A_n| \le c_n$, where
          $$
          A_n = \{y\in E: m^n(x,y) > 0\}.
          $$
      Then we have that $\rho_c = \rho_c' = \rho_c''$.
    \end{theo}

    \subsection{Proofs}
    \begin{proof}[Proof of Theorem~\ref{th:rho'_rho''}]
      We start with the first inequality. 
    Let $\rho < \rho_{c,x}'$. Then there exists $u$ taking values in $[0,1]$, such that $mu \ge \rho u$ and $u(x) > 0$. We have
    \[
    \expE_x[N_n]  \ge m^n u(x) \ge \rho^n u(x),
    \]
    and so
    \[
    \liminf_{n\to\infty} (\expE_x[N_n])^{1/n} \ge \rho.
    \]
    This proves the first inequality. The second inequality is trivial.

      We now prove the equality, i.e.~that $\rho''_{c,x} = \limsup_{n\to\infty}\expE_x\left[N_n\right]^{1/n}$.
    Fix $x\in E$.
    ``$\ge$'': Let $\rho > \rho'_{c,x}$. By definition of $\rho'_{c,x}$,  there exists $v\ge 1$ measurable, with $v(x) < \infty$, such that $mv \le \rho v$, hence $m^n v\le \rho^n v$. It follows that
    \[
    \expE_x[N_n] = m^n 1(x) \le m^n v(x) \le \rho^n v(x),
    \]
    and therefore, since $v(x) < \infty$,
    \[
    \limsup_{n\to\infty} \expE_x[N_n]^{1/n} \le \rho.
    \]
    Since $\rho > \rho'_{c,x}$ was arbitrary, this shows the ``$\ge$'' part in the statement of the theorem.
    
    ``$\le$'': It suffices to consider the case $\rho \coloneqq \limsup_{n\to\infty} \left(\expE[N_n]\right)^{1/n} < \infty$. Fix $\eta > \rho$ and define
    \[
    \bar N_\eta = \sum_{n=0}^\infty \eta^{-n}N_n.
    \]
    Define 
    \[
    v(z) \coloneqq \expE_z[\bar N_\eta],
    \]
    and note that $v \ge 1$ and $v(x) < \infty$ by the definition of $\eta$ and $\rho$. Decomposing according to the particles at the first generation, we get
    \[
    v(z) = 1+ \eta^{-1} m v(z),\quad z\in E,
    \]
    and therefore
    \[
    mv(z) = \eta(v(z)-1) \le \eta v(z).
    \]
    It follows that $\rho_{c,x}'' \le \eta$. Since $\eta>\rho$ was arbitrary, we get $\rho_{c,x}'' \le \rho$, which finishes the proof.
    \end{proof}

    \begin{proof}[Proof of Theorem~\ref{th:rho'_survival}]
\emph{First part.} Define $u\in pB_1(E)$ by  $u(x) = \Pm_x(\glSurv)$, $x\in E$. Fix $x\in E$ and suppose $u(x) > 0$. By the branching property, we have
  \[
    1-u = \pi(1-u) \ge 1-mu,
  \]
  and therefore,
  \[
    mu \ge u.
  \]
  It follows by definition that $\rho'_{c,x} \ge 1$.

  \emph{Second part.}  Let $u$, $\rho$ and $\delta$ be as in the statement of the theorem. For $s\le \delta$, set $f_s = e^{-su}$ and let for every $x\in E$,
  \begin{align*}
    F_s(x)
    \coloneqq \pi f_s(x)
    = \expE_x\left[ \exp\left(-s\sum_{i\in\N} u(X_i)\right)\right].
  \end{align*}
  Differentiating under the integral sign yields
  \begin{align*}
    \frac{\partial}{\partial s} F_s(x)
     &= -\expE_x\left[ \sum_{i\in\N} u(X_i)\exp\left(-s\sum_{i\in\N} u(X_i)\right)\right].
    \end{align*}
    By hypothesis \eqref{eq:assumption_u_discrete}, we have
    \(
    \frac{\partial}{\partial s} F_s(x)  \le -\rho u(x),
    \)
    and therefore, for every $s\le \delta$,
    \begin{equation}
      \label{eq:F_t_upperbound}
      \pi f_s(x)= F_s(x) = F_0(x) + \int_0^s \frac{\partial}{\partial u} F_u(x)\,du \le 1 - \rho s u(x) \le e^{-\rho s u(x)} = f_{\rho s}(x).
    \end{equation}


    Now fix $x\in E$. Suppose that $u(x)>0$, so that $f_\delta(x) < 1$. It follows from \eqref{eq:F_t_upperbound} that $M_n = \prod_{|u|=n} f_{\rho^{-n}\delta}(X_u)$ is a (positive) supermartingale. In particular, it converges almost surely to a non-negative limit $M_\infty$, which satisfies $\expE_x[M_\infty] \le M_0 = f_\delta(x) < 1$. If $L = \sup_{y\in E} u(y)<\infty$, we observe that
    \[
      \exp(-\delta L \rho^{-n} N_n) \le M_n,
    \]
    so that
    \[
      \liminf_{n\to\infty} \rho^{-n}N_n \ge -\log(M_\infty)/(\delta L).
    \]
    Now, since $M_\infty \ge 0$ and $\expE_x[M_\infty] < 1$, we have that $\Pm_x(M_\infty < 1) > 0$. It follows that 
    \[
      \Pm_x(\liminf_{n\to\infty} \rho^{-n}N_n > 0) > 0.
    \]
    This proves the statement.

    \emph{Third part.} Assume that $K<\infty$ is a constant such that $\Pm_x(N_1 \le K) = 1$ for every $x\in E$. Assume $\rho'_{c,x} > 1$ and fix $\rho<\rho'_{c,x}$. Let $\tilde \rho\in(\rho,\rho'_{c,x})$ and let $u\in pB_b(E)$ be such that $mu \ge \tilde\rho u$. We have that for every $\delta > 0$,
    \[
      \exp\left(-\delta \sum_{i\in\N} u(X_i)\right) \ge e^{-\delta K L},
    \]
    with $L = \sup_{y\in E} u(y)$ as above. Hence, letting $\delta$ sufficiently small so that $e^{-\delta K L} \ge \rho/\tilde \rho$, we have that
    \[
      \expE_x\left[\sum_{i\in \N} u(X_i) \exp\left(-\delta \sum_{i\in \N} u(X_i)\right) \right] \ge (\rho/\tilde\rho)\expE_x\left[\sum_{i\in \N} u(X_i)\right] = (\rho/\tilde\rho) mu(x) \ge \rho u(x).
    \]
    In other words, Assumption~\eqref{eq:assumption_u_discrete} is verified. The statement follows from the second part of the Theorem.

    \emph{Fourth part.} By Theorem~\ref{th:rho'_rho''}, we have $\rho''_{c,x} = \limsup_{n\to\infty} \expE_x\left[N_n\right]^{1/n}$. In particular, we have for every $\varepsilon > 0$,
    \[
      \expE_x\left[N_n\right] \le (\rho''_{c,x}+\varepsilon)^n,\quad \text{for $n$ sufficiently large.}
    \]
    It follows by Markov's inequality that
    \[
      \Pm_x\left(N_n \ge (\rho''_{c,x}+2\varepsilon)^n\right) \le e^{-\varepsilon n},
    \]
    for $n$ sufficiently large. By the Borel-Cantelli lemma, we have
    \[
      \Pm_x\left(\limsup_{n\to\infty} N_n^{1/n} \le \rho''_{c,x}+2\varepsilon\right) = 1.
    \]
    Since $\varepsilon>0$ was arbitrary, the statement follows.
    \end{proof}
    
    \begin{proof}[Proof of Proposition~\ref{prop:rho dependence on x}]
    We first prove the first identity.
    We trivially have $\rho'_c \le \rho'_{c,x}$ for every $x\in E$, since the supremum is taken over a smaller set in the definition of $\rho'_c$ compared to $\rho'_{c,x}$. It is therefore enough to show that $\rho'_c \ge \inf_{x\in E} \rho'_{c,x}.$
    Let $\rho < \inf_{x\in E} \rho'_{c,x}$. Then for each $x\in E$, there exists a function $u_x\in pB(E)$ such that $u_x\le 1$, $mu_x \ge \rho u_x$ and $u_x(x) > 0$. Set $u \coloneqq \sup_{x\in E} u_x$. Then we have that $0<u\le 1$ and $u\in pB(E)$. Furthermore, $$mu \ge \sup_{x\in E} mu_x \ge \rho \sup_{x\in E} u_x = \rho u.$$ It follows that $\rho'_c \ge \rho$. Since $\rho < \inf_{x\in E} \rho'_{c,x}$ was arbitrary, it follows that $\rho'_c \ge \inf_{x\in E}\rho'_{c,x}$, which proves the equality. The proof for the second identity is similar.
    \end{proof}

\begin{proof}[Proof of Theorem \ref{th:discrete_harris}]
We define
\[
A:=\{x\in E:\Pm_x(\glSurv)>0\}.
\]
Then if $\nu(A)=0$ we have $\Pm_x(\glSurv)=0$ for $\nu$-almost every $x\in E$. 

We now assume the converse, so that $\nu(A)>0$. Then by Assumption \ref{assum:Harris discrete} we can, for every $x\in E$, find $n\ge 0$ such that $m^n(x,A) > 0$. This implies that there exists $u\in U_n$, such that
\[
\Pm_x(X_u \in A) > 0.
\]
It follows that
\[
\Pm_x(\glSurv)\geq \expE_x[\Pm_{X_u}(\glSurv)\,|\,X_u\in A]\Pm_x(X_u\in A) > 0.
\]
This concludes the proof of Theorem \ref{th:discrete_harris}.
\end{proof}
    
    \begin{proof}[Proof of Theorem~\ref{th:rho_c_rho_c'}]
    \emph{First proof.} Upon multiplying $m$ by a constant, it is enough to show that $\rho_c > 1$ implies $\rho'_c \ge 1$. Now, by Theorem~\ref{th:muller}, the BMC survives locally if $\rho_c > 1$, i.e.~$\Pm_x(\locSurv) > 0$ for all $x\in E$. Since $\locSurv \subset \glSurv$, it follows that $\Pm_x(\glSurv) > 0$ for all $x\in E$. Hence, by Theorem~\ref{th:rho'_survival}, we have $\rho'_c\ge 1$, which was to be proven.
    
    \emph{Second proof.} By the second part of Theorem~\ref{th:spectral_radius}, it is enough to show that $\rho'_c \ge \rho_c(m_F)$ for every finite $F\subset E$ such that $m_F$ is irreducible. Let $F\ne \emptyset$ be such a set. Note that $\rho_c(m_F)$ is the Perron-Frobenius eigenvalue of the finite matrix $m_F$. Let $u_F$ be the corresponding right eigenvector. We identify it with a function on $E$, by setting $u_F(x) = 0$ for $x\in E\backslash F$. We then have $m u_F \ge m_Fu_F = \rho_c(m_F) u_F$. Furthermore, $u_F$ is bounded since $F$ is finite and positive on $F$, since $m_F$ is irreducible. It follows that $u = u_F/\|u_F\|_\infty$ satisfies the conditions in the definition of $\rho'_{c,x}$, for some $x\in F$. Hence, $\rho'_{c,x} \ge \rho_c(m_F)$, and thus $\rho'_c = \rho'_{c,x} \ge \rho_c(m_F)$. This finishes the proof.
    \end{proof}

    \begin{proof}[Proof of Theorem~\ref{th:reversible}]
      Let $x\in E$ as in the statement of the theorem. By Theorem~\ref{th:rho'_rho''},  and Theorem~\ref{th:rho_c_rho_c'}, it is enough to show that $\rho_c \ge \rho_c''$.
      Observe that for every $n\in\N$, we have
      \begin{align*}
        m^{2n}(x,x) 
        &= \sum_{y\in A_n} m^n(x,y)m^n(y,x)\\
        &\ge c_n^{-1} \sum_{y\in A_n} m^n(x,y)m^n(x,y) && \text{by hypothesis}\\
        &\ge c_n^{-1} |A_n|^{-1}\left(\sum_{y\in A_n} m^n(x,y)\right)^2 && \text{by Cauchy-Schwarz}\\
        &\ge c_n^{-2} \left(\expE_x[N_n]\right)^2 && \text{by hypothesis and definition of $N_n$}
      \end{align*}
    Using the hypothesis on $c_n$, it follows that 
    \[
      \limsup_{n\to\infty} (m^{2n}(x,x))^{1/2n} \ge \limsup_{n\to\infty} \left(\expE_x[N_n]\right)^{1/n} = \rho_{c,x}'' = \rho''_c.
      \]
    Theorem~\ref{th:spectral_radius} implies that $\rho_c \ge \rho_c''$, which concludes the proof.
    \end{proof}

\section{Continuous-time results}
\label{sec:results_continuous_time}

This section is devoted to branching Markov processes (BMP0) in continuous time. We first describe the setting. Let $E$ be a Lusin space, i.e.~a topological space homeomorphic to a Borel subset of a Polish space. The space $E$ is endowed with the Borel $\sigma$-algebra. We work with the same spaces of measurable functions on $E$ as in Section~\ref{sec:results_discrete_time}. 

The BMP we consider will be defined through the quadruple $(X^0,b,(p_n)_{n\ge 2},(K_n)_{n\ge 2})$, where
\begin{enumerate}
  \item $X^0 = (X^0)_{0\le t<\tau}$ is a  right  process\footnote{We refer to Appendix A in Li~\cite{LiMeasureValuedBranching2011} for a reference on right processes. Note that these include for example Feller processes, diffusions in arbitrary domains in Euclidean space or manifolds with various boundary conditions, piecewise deterministic Markov processes, continuous-time Markov chains in countable state space, etc. } on $E$, with possibly finite lifetime $\tau$.
	\item $r$ is a bounded non-negative measurable function on $E$, called the \emph{branching rate}.
  \item $(p_n)_{n\ge 2}$ are non-negative measurable functions on $E$ with $\sum_{n=2}^\infty p_n(x) \equiv 1$, called the \emph{offspring distribution}. We assume that this family is uniformly integrable, i.e. for every $\varepsilon>0$, there exists $n\in \N$, such that $\sum_{m\ge n}mp_m(x) < \varepsilon$ for all $x\in E$.
  \item $(K_n)_{n\ge 2}$ is a family of probability kernels from $E$ to $E^n$, called the \emph{branching kernel}. 
\end{enumerate}
The BMP is heuristically defined as follows: we start with a single particle which moves according to the right process $X^0$ until it either dies (at its lifetime), or branches. Branching happens at space-dependent rate $r(x)$. At a branching event, the number of children is drawn according to the offspring distribution $(p_n(x))_{n\ge 2}$, then these children are randomly displaced according to the branching kernel $K_n$. The resulting particles then continue this process independently. We say that the BMP has \emph{purely local branching}, if the particles do not displace at branching events, i.e. if 
  $K_n(x,\cdot) = \delta_{(x,\ldots,x)},\quad x\in E.$ Otherwise, we say that the BMP has \emph{non-local branching}.

Formally, a BMP is a Markov process $(\vec X_t)_{t\ge0}$ on the space of finite counting measures on $E$, i.e., measures of the form $\sum_{i=1}^n \delta_{x_n}$, for $n\in\N$ and $x_1,\ldots,x_n\in E$, the space being endowed with the weak topology. It is shown by Beznea and Lupaşcu \cite{BezneaLupascuMeasurevaluedDiscrete2015} that a process evolving as above exists\footnote{Beznea and Lupascu~\cite{BezneaLupascuMeasurevaluedDiscrete2015} assume the process $X^0$ to be conservative, i.e.~having infinite lifetime. One can reduce to this assumption in the usual way by adding an isolated cemetery point to the state space $E$.} and it in fact a right process. Classical treatments include Ikeda, Nagasawa and Watanabe \cite{Ikeda1968} and Asmussen and Hering \cite[Chapter~5]{AH1983}, who deal with compact and locally compact $E$, respectively. We denote the law of the process, started from a single particle at $x\in E$, by $\Pm_x$. The process $\vec X$ satisfies the \emph{branching property}, stating that the process started from particles at positions $x_1,\ldots,x_n\in E$ is equal in law to the sum of $n$ independent processes distributed according to $\Pm_{x_1},\ldots,\Pm_{x_n}$, repectively.

For ease of notation, we denote the positions of the particles at time $t$ by $X_t^1,\ldots,X_t^{N_t}$, where $N_t$ denotes the number of particles at time $t$ and where we order the particles arbitrarily: we will never consider quantities that depend on a specific order of the particles.

\begin{rmk}
  Our setting is quite general, but not the most general possible. Ikeda, Nagasawa and Watanabe \cite{Ikeda1968} allow for branching at random times which may not be generated by a rate function. More importantly, we assume our branching rate to be bounded, which is certainly not necessary for the existence of the process, but many of our results rely on this assumption. Note however, that we allow for the most general form of killing, such as ``hard'' killing at certain stopping times, or ``soft'' killing with a position-dependent and possibly unbounded rate. In fact, all of these can be integrated into the lifetime $\tau$ of the right process $X^0$.
\end{rmk}

We define a family of non-negative kernels $(P_t)_{t\geq 0}$ on $E$, called the \emph{expectation semigroup} by
\begin{equation}
    P_t(x,\cdot):=\expE_{\bar X_0=(x)}\left[\sum_{i=1}^{N_t}\delta_{X^i_t}(\cdot)\right].
\end{equation}
Then we define
\[
P_tf(x):=\expE_{\bar X_0=(x)}\left[\sum_{i=1}^{N_t}f(X^i_t)\right]
\]
for $f$ belonging to the following classes of measurable functions:
\begin{enumerate}
    \item For $f \in B_b(E)$. In this case $P_tf\in B_b(E)$ necessarily since the jump rate is bounded and the offspring distribution has bounded expectation.
    \item For $f\in pB(E)$. In this case $P_tf\in pB(E)$.
    \item For $f\in B(E)$ such that $\sup_{x\in E}f(x)<\infty$ (but allowing for $\inf_{x\in E}f(x)=-\infty$). In this case $f\in B(E)$ and $\sup_{x\in E}f(x)<\infty$, again since the jump rate is bounded and the offspring distribution has bounded expectation.
\end{enumerate}
In all of the above cases $P_tf$ is a measurable function. Note that if $f$ is unbounded from above and below, $P_tf$ is \textit{not well-defined}, hence we shall not consider $P_tf$ for such $f$. 

The expectation semigroup satisfies an integral equation. To state it, define the semigroup corresponding to the process $X^0$ by
\[
P^0_tf(x):=\expE_x[f(X^0_t)\Ind(\tau>t)],\quad f\in B_b(E).
\]
The expectation semigroup $P_t$ then satisfies,
\begin{align}
  \label{eq:linear_semigroup}
  P_t f = P^0_t f + \int_0^t P^0_{t-s}b (M-1) P_s f\,ds,
\end{align}
where in the above formula, $1$ is the identity operator and $M$ is defined by 
\[
Mf(x) \coloneqq \sum_{n=2}^\infty p_n(x)\int \sum_{i=1}^n f(x_i) K_n(x,dx_1\cdots dx_n),
\]
We will not use \eqref{eq:linear_semigroup} directly. We will only use an explicit representation of $P_t$ in the case of purely local branching. In that case, $M$ is the multiplication operator
\[
  Mf(x) = m(x) f(x),\quad m(x) \coloneqq \sum_{n\ge 2}np_n(x)
\]
and $P_t$ is a Feynman-Kac semigroup:
\begin{equation}
  \label{eq:Pt_FK}
  P_t f(x) = \expE_x\left[f(X^0_t)\exp\left(\int_0^t b(X^0_t)(m(X^0_t)-1)\,ds\right)1\{t<\tau\}\right].
\end{equation}
\begin{rmk}
  \label{rmk:expectation_semigroup}
  Note that, as in the discrete-time setting, several BMP may give rise to the same expectation semigroup $(P_t)_{t\ge0}$. For this article, BMP with purely local branching will play a special role as we will be able to prove more things for them. If one is given a linear semigroup of non-negative operators that admits a Feynman-Kac representation as in \eqref{eq:Pt_FK}, then under mild assumptions one can define a BMP with purely local branching that admits $(P_t)_{t\ge0}$ as the expectation semigroup.
\end{rmk}

\subsection{Results}

As in Section~\ref{sec:results_discrete_time}, we define the event of global survival by
\[
  \glSurv \coloneqq \{N_t \ge 1\,\forall t\ge 0\}.
\]
If $\Pm_x(\glSurv) > 0$, we say that the BMP \emph{survives globally} when starting from $x$. In the contrary event, we say that is \emph{undergoes extinction}.

In order to study global survival vs. extinction, we propose to consider the following generalised principal eigenvalues:
\begin{align*}
  \lambda_{c,x}' &= \lambda_{c,x}'((P_t)_{t\geq 0}) := \sup\{\lambda\in\R: \exists u\in pB(E): P_tu\geq e^{\lambda t}u\,\forall t\ge 0,\,u\le 1,\,u(x) > 0\}.\\
  \lambda_{c,x}'' &= \lambda_{c,x}''((P_t)_{t\geq 0}) := \inf\{\lambda\in\R: \exists u\in pB(E): P_tu\leq e^{\lambda t}u\,\forall t\ge 0,\,u\ge 1,\,u(x) < \infty\}.
\end{align*}
We also define
\begin{align*}
  \lambda_c' &= \lambda_c'((P_t)_{t\geq 0}) := \sup\{\lambda\in\R: \exists u\in pB(E): P_tu\geq e^{\lambda t}u\,\forall t\ge 0,\,0<u \le 1\}.\\
  \lambda_c'' &= \lambda_c''((P_t)_{t\geq 0}) := \inf\{\lambda\in\R: \exists u\in pB(E): P_tu\leq e^{\lambda t}u\,\forall t\ge 0,\,1\le u<\infty\}.
\end{align*}
These are the analogues of $\lambda'_1$ and $\lambda_1''$ in \cite{BerestyckiRossi}, in a sense which we shall make precise in Section \ref{sec:results_branching_diffusions}.

The following two theorems are the analogues of the corresponding discrete-time results.
    \begin{theo}
      \label{th:lambda'_lambda''}
      For every $x\in E$, we have
      \[
        \lambda'_{c,x} \le \liminf_{t\to\infty} \frac 1 t \log \expE_x[N_t] \le \limsup_{t\to\infty} \frac 1 t \log \expE_x[N_t] = \lambda_{c,x}''.
      \]
    \end{theo}

    \begin{rmk}
      As mentioned above, both inequalities in Theorem~\ref{th:lambda'_lambda''} may be strict as the example in Section~\ref{sec:example_mutations} shows.
    \end{rmk}

\begin{theo}
  \label{th:lambda'_survival}
  The following holds.
  \begin{enumerate}
    \item Let $x\in E$. If $\Pm_x(\glSurv) > 0$, then $\lambda'_{c,x}(m) \ge 0$.

    \item 
      Suppose the branching is purely local. Then,
      \[
    \text{$\lambda'_{c,x} > 0$ implies $\Pm_x(\glSurv)>0$ and }\forall \lambda < \lambda'_{c,x}:\Pm_x\left(\liminf_{t\to\infty} \frac 1 t \log N_t \ge \lambda\right) > 0.
  \]
  Moreover, we have $\lambda'_{c,x} = \frac 1 t \log \rho'_{c,x}(P_t)$ for every $t>0$.
    \item We have for every $x\in E$,
      \[
        \Pm_x\left(\limsup_{t\to\infty} \frac 1 t \log N_t \le \lambda''_{c,x}\right) = 1.
      \]
  \end{enumerate}
\end{theo}

\begin{rmk}
It is not clear to us whether the local branching assumption in Part 2 of Theorem \ref{th:lambda'_survival} is necessary.
\end{rmk}

\subsubsection*{Dependence of global survival upon the initial condition}
We now consider the dependence of global survival and the generalised principal eigenvalues $\lambda'_{c,x}$ and $\lambda''_{c,x}$ upon $x$. We start with a basic result.
\begin{prop}
  \label{prop:lambda dependence on x}
  We have $\lambda'_c(m) = \inf_{x\in E}\lambda'_{c,x}(m)$ and $\lambda''_c(m) = \sup_{x\in E}\lambda''_{c,x}(m)$.
\end{prop}
Theorem~\ref{th:lambda'_survival} and Proposition~\ref{prop:lambda dependence on x} immediately imply the following:
\begin{cor}
  \label{cor:lambda'_survival}
  The following holds.
  \begin{enumerate}
  \item  We have
      \[
        \lambda'_c \le \inf_{x\in E}\liminf_{t\to\infty} \frac 1 t \log \expE_x[N_t] \le \sup_{x\in E}\limsup_{t\to\infty} \frac 1 t \log \expE_x[N_t] = \lambda_c''.
      \]
    \item If $\Pm_x(\glSurv) > 0$ for all $x\in E$, then $\lambda'_c(m) \ge 0$.

    \item 
      Suppose the branching is purely local. Then,
      \[
    \text{$\lambda'_c > 1$ implies for all $x\in E: \Pm_x(\glSurv)>0$ and }\forall \lambda<\lambda'_c:\Pm_x\left(\liminf_{t\to\infty} \frac 1 t \log N_t \ge \lambda\right) > 0.
  \]
  Moreover, we have $\lambda'_c = \frac 1 t \log \rho'_{c,x}(P_t)$ for every $t>0$ and every $x\in E$.
    \item We have for every $x\in E$,
      \[
        \Pm_x\left(\limsup_{t\to\infty} \frac 1 t \log N_t \le \lambda''_c\right) = 1.
      \]
  \end{enumerate}
\end{cor}

We now establish conditions ensuring global survival doesn't depend upon the initial condition. 
The first of these will be the following Harris-type irreducibility assumption.
\begin{assum}\label{assum:Harris type irreducibility}
There exists a measure $\nu$ such that for every measurable set $A$ with $\nu(A)>0$ and every $x\in E$ there exists $t>0$ (possibly dependent upon $x$ and $A$) such that we have $\Pm_x(X^i_t\in A\text{ for some $1\leq i\leq N_t$})>0$.
\end{assum}
\begin{theo}\label{theo:dichotomy Harris assumptions}
We assume Assumption \ref{assum:Harris type irreducibility}, for some given measure $\nu$. Then we have either $\Pm_x(\glSurv)>0$ for all $x\in E$, or otherwise $\Pm_x(\glSurv)=0$ for $\nu$-almost every $x \in E$.
\end{theo}

It follows that under the assumptions of Theorem \ref{theo:dichotomy Harris assumptions} we have that:
\begin{enumerate}
    \item if $\lambda_c'<0$ and branching is purely local, then we have $\Pm_x(\text{global extinction})=1$ for $\nu$-almost every $x\in E$;
    \item if $\lambda_c'>0$ then we have $\Pm_x(\glSurv)>0$ for all $x\in E$.
\end{enumerate}
In particular, under the assumptions of Theorem \ref{theo:dichotomy Harris assumptions} we have that
\begin{equation}
    \lambda_c'=\lambda_{c,x}'\quad\text{for all}\quad x\in E.
\end{equation}

The next assumption is stronger and relies on the strong Feller property and topological irreducibility assumptions. We define $P^0_t$ to be the transition density for the killed right process $(X_t)_{0\leq t<\tau_E}$, i.e.
\[
(P^0_tf)(x):=\expE_x[f(X_t)\Ind(\tau>t)],\quad f\in B_b(E).
\]
\begin{assum}\label{assum:Strong Feller top irr}
We assume that the following holds:
\begin{enumerate}
    \item The \textit{strong Feller assumption}, meaning that
\[
P^0_t(B_b(E))\subseteq C_b(E).
\]
for all $t>0$.
\item \textit{Topological irreducibility}, meaning that
\[
\Pm_x(X^i_t\in U\text{ for some $1\leq i\leq N_t$ and some $t>0$})>0
\]
for every $x\in E$ and non-empty, open subset $U\subseteq E$.
\end{enumerate}
\end{assum}

We then have the following dichotomy. 
\begin{theo}\label{theo:global extinction dichotomy}
We assume that Assumption \ref{assum:Strong Feller top irr} holds. Then either $\Pm_x(\glSurv)>0$ for all $x\in E$, or otherwise $\Pm_x(\glSurv)=0$ for all $x \in E$.
\end{theo}

It follows that under the conditions of Theorem \ref{theo:global extinction dichotomy} we have the following dichotomy:
\begin{enumerate}
    \item if $\lambda_c'<0$ and branching is purely local then we have $\Pm_x(\glSurv)=0$ for all $x\in E$;
    \item if $\lambda_c'>0$ then we have $\Pm_x(\glSurv)>0$ for all $x\in E$.
\end{enumerate}

Note that piecewise-deterministic Markov processes would typically satisfy Assumption \ref{assum:Harris type irreducibility} with $\nu$ being Lebesgue measure, but not be strong Feller so not satisfy Assumption \ref{assum:Strong Feller top irr}.

These provide conditions ensuring that $\lambda_{c,x}'$ doesn't depend upon $x$. One might similarly ask whether one can obtain similar conditions ensuring that $\lambda_{c,x}''=\limsup_{t\ra\infty}\frac{1}{t}\ln \expE_x[N_t]$ also doesn't depend upon $x$. We will establish that $\lambda_c''=\limsup_{t\ra\infty}\frac{1}{t}\ln \expE_x[N_t]$ for all $x\in E$, i.e. $\lambda_c''=\lambda_{c,x}''$ for all $x\in E$, in the case of uniformly elliptic branching diffusions in Section \ref{sec:results_branching_diffusions}. However this relies on the Harnack inequality, so it's not clear to us whether it can be made more general.

The rest of this section is devoted to the proofs of the above results.

\subsection{Proofs}

\subsubsection{Proof of Theorem~\ref{th:lambda'_lambda''}}

\subsubsection*{Proof of the first inequality in Theorem~\ref{th:lambda'_lambda''}}

The proof is exactly the same as the proof of the corresponding discrete-time result in Theorem~\ref{th:rho'_rho''}.
    Let $\lambda < \lambda_{c,x}'$. Then there exists $u$ taking values in $[0,1]$, such that $P_t u \ge e^{\lambda t} u$ for all $t\ge 0$ and $u(x) > 0$. We have
    \[
    \expE_x[N_t]  \ge P_t u(x) \ge e^{\lambda t}u(x),
    \]
    and so
    \[
    \liminf_{t\to\infty} \frac 1 t \log \expE_x[N_t]\ge \lambda.
    \]
    This proves the first inequality.

\subsubsection*{Proof of the equality in Theorem~\ref{th:lambda'_lambda''}}
Fix $x_\ast$. We first prove that $\lambda''_{c,x_\ast} \ge \limsup_{t\to\infty} \frac 1 t \log \expE_{x_\ast}[N_t]$. This works as in the discrete-time setting.
Let $\lambda > \lambda_{x_{\ast}}''$. Then we may take $u\in pB(E)$ and $\epsilon>0$ such that: (1) $u\ge 1$, (2) $u(x_{\ast})<\infty$, and (3) $u$ satisfies $P_tu\leq e^{\lambda t} u$. Then we have for all $x\in E$,
\[
\expE_{x}[N_t]=(P_t1)(x)\leq (P_tu)(x)\leq e^{\lambda t} u(x).
\]
In particular, this shows that 
\[
\limsup_{t\to\infty} \frac 1 t \expE_{x_{\ast}}[N_t] \le \lambda.
\]
Since $\lambda > \lambda_{x_{\ast}}''$ was arbitrary, this establishes the statement.

 We now prove $\lambda''_{c,x_\ast} \le \limsup_{t\ra\infty}\frac{1}{t}\ln \expE_{x_{\ast}}[N_t]$.
Compared to the corresponding discrete-time proof, there is an additional technical difficulty, as the function we construct is not necessarily uniformly bounded from below.
Since we can always add a constant local binary branch rate or add an additional constant killing rate, which has an obvious effect on the expectation semigroup and generalised principal eigenvalues, it suffices to prove the following: if $\limsup_{t\ra\infty}\frac{1}{t}\ln \expE_{x_{\ast}}[N_t]<0$ then $\lambda_{x_{\ast}}''\leq 0$.
We suppose that $\limsup_{t\ra\infty}\frac{1}{t}\ln \expE_{x_{\ast}}[N_t]<0$. We now define
\begin{equation}\label{eq:total mass}
\calT \coloneqq \int_0^\infty N_t\,dt.
\end{equation}
We then define
\[
v(x):=\expE_x[\calT].
\]
We observe that $v\geq 0$ everywhere and 
\begin{equation}\label{eq:lambda''<= pf v finite at xast}
v(x_{\ast})=\int_0^{\infty}\expE_{x_{\ast}}[N_t]<\infty
\end{equation}
since $\limsup_{t\ra\infty}\frac{1}{t}\ln \expE_{x_{\ast}}[N_t]<0$.

Given $t>0$ and $1\leq i\leq N_t$ we consider the same quantity defined by the children of $X^i_t$ from time $t$, i.e.
\[
\calT^i_t:=\int_t^{\infty}\sum_{j=1}^{N_s}\Ind(X^j_s\text{ is a descendent of $X^i_t$})\,ds.
\]
Then for any $t>0$ we have that
\[
\begin{split}
\calT=\int_0^{t}\sum_{j=1}^{N_s}\Ind(X^j_s\text{ is a descendent of $X^1_0$})\,ds+\sum_{i=1}^{N_t}\int_t^{\infty}\sum_{j=1}^{N_s}\Ind(X^j_s\text{ is a descendent of $X^i_t$})\,ds\\
= \int_0^{t}\sum_{j=1}^{N_s}\Ind(X^j_s\text{ is a descendent of $X^1_0$})\,ds+\sum_{i=1}^{N_t}\calT^i_t.
\end{split}
\]
It follows that
\[
v(x)\geq \int_0^t(P_s1)(x)\,ds+(P_tv)(x).
\]
We now define
\[
u(x)=v(x)+\epsilon
\]
for $\epsilon>0$ to be determined. In what follows, we drop the argument $x$ from the notation. We observe that
\begin{equation}\label{eq:pf of lambda''<=0 Ptu inequality}
\begin{split}
P_tu= P_tv+P_t\epsilon\leq  [v-\int_0^t P_s1\,ds]+\epsilon P_t1\\
\leq  u+[\epsilon (P_t1- 1)- \int_0^t P_s1\,ds]\\
\leq  u+[\epsilon (P_t1-1)-\int_0^t P_s1\,ds].
\end{split}
\end{equation}

Our goal now is to show that $\epsilon>0$ can be chosen so that
\begin{equation}\label{eq:epsilon Pt1-1-int ineq}
\epsilon (P_t1-1)-\int_0^t P_s1\,ds\leq 0
\end{equation}
for all $t>0$ sufficiently small.

We have that for all $0\le s\le t$,
\[
P_t1=P_sP_{t-s}1\leq e^{C(t-s)}P_s1,
\]
where
\[
C:=1+\lvert\lvert r\rvert\rvert\sup_{x}\sum_nnp_n(x).
\]
Therefore, for
\[
P_s1\geq e^{-C(t-s)}P_t1.
\]
Thus
\[
\epsilon (P_t1-1)-\int_0^t(P_s1)\,ds\leq \epsilon (P_t1-1)-\left(\int_0^te^{-C(t-s)}\,ds \right) P_t1=\epsilon (P_t1-1)-\frac{1-e^{-Ct}}{C}P_t1.
\]

In what follows, we suppose that $t>0$ is sufficiently small, depending on $C$ and $\epsilon.$
We observe that $\frac{1-e^{-Ct}}{C}\geq \frac{t}{2}$. Therefore
\[
\epsilon (P_t1-1)-\int_0^t(P_s1)\,ds\leq \epsilon P_t1-\epsilon -\frac{t}{2}P_t1.
\]
Since $P_t1\leq 1+2Ct$, we have for $t < 2\epsilon$,
\[
\epsilon P_t1-\epsilon-\frac{t}{2}P_t1\leq (1+2Ct)(\epsilon-\frac{t}{2})-\epsilon \leq 2Ct\epsilon-\frac{t}{2}\leq 0,
\]
if $4C\epsilon\leq 1$. Therefore by choosing 
\[
\epsilon=\frac{1}{4C},
\]
and thereby defining $u:=v+\epsilon$, we have established \eqref{eq:epsilon Pt1-1-int ineq} for all $t>0$ sufficiently small. Feeding this into \eqref{eq:pf of lambda''<=0 Ptu inequality}, we obtain that there exists $T>0$ such that
\[
P_tu\leq u
\]
for all $0\leq t\leq T$. It then follows by induction that $P_tu\leq  u$ for all $t\in \Rm_{>0}$ by induction. Moreover $u\geq \epsilon>0$ everywhere and $u(x_{\ast})<\infty$ (since $v(x_{\ast})<\infty$). Therefore $\lambda''_x\leq 0$. This finishes the proof.
\qed

\subsubsection{Proof of Part 1 of Theorem \ref{th:lambda'_survival}}
Suppose that we have global survival from $x_{\ast}$. Consider the function
\[
u(x):=\Pm_x(\glSurv).
\]
We observe that $u\in pB_1(E)$ and $u(x_{\ast})>0$.

We say that $X^i_t$ survives for some $1\leq i\leq N_t$ if the branching process made up of $X^i$ and all its descendents from time $t$ onwards survives. 

We observe that
\[
\sum_{i=1}^{N_t}\Ind(X^i_t\text{ survives})\geq \Ind(X^i_t\text{ survives for some $1\leq i\leq N_t$})=\Ind(X^1_0\text{ survives}).
\]
Then
\[
\begin{split}
(P_tu)(x)=\expE_x\Big[\sum_{i=1}^{N_t}\Ind(X^i_t\text{ survives})\Big]
\geq \expE_x\Big[ \Ind(X^i_t\text{ survives for some $1\leq i\leq N_t$})\Big]=u(x).
\end{split}
\]
Therefore
\[
(P_tu)(x)\geq u(x)
\]
for all $x\in E$, implying that $\lambda_{x_{\ast}}'\geq 0$.

\subsubsection{Proof of Part 2 of Theorem \ref{th:lambda'_survival}}

We wish to reduce to the corresponding discrete-time result (Theorem~\ref{th:rho'_survival}), by considering a discrete-time skeleton of the process. The key to this is the following proposition, which allows to bound the process from below by a discrete-time process with bounded number of offspring.
\begin{prop}\label{prop:large minorising bounded offspring dist branching processes}
We assume that the branching is purely local. Then for all $\epsilon>0$ there exists a discrete-time branching process ${\bar Y}_n=(Y^1_n,\ldots,Y^{M_n}_n)$ ($n\in \mathbb{N}$) with bounded offspring distribution such that:
\begin{enumerate}
    \item The discrete-time skeleton of $\bar X$, $(\bar X_n)_{n\in \mathbb{N}}$, dominates $(\bar Y_n)_{n\in \mathbb{N}}$ in the sense that, for every starting point $y\in E$, we can couple them so that 
    \[
    \Pm_y\left(\forall n\ge 0: \sum_{i=1}^{N_n}\delta_{X^i_n}\geq \sum_{j=1}^{M_n}\delta_{Y^i_n}\right) = 1,
    \]
    where for two measures $\mu$ and $\nu$ we write $\mu\ge \nu$ if $\mu(A) \ge \nu(A)$ for every measurable $A$.
    \item The associated expectation kernel 
    \begin{equation}\label{eq:mean kernel Q}
    Q(y,\cdot):=\expE_x\Big[\sum_{j=1}^{M_1}\delta_{Y^i_1}\Big]
    \end{equation}
    satisfies $Q(y,\cdot)\geq (1-\epsilon)P_1(y,\cdot)$ for all $y\in E$.
\end{enumerate}
\end{prop}
\begin{proof}[Proof of Proposition~\ref{prop:large minorising bounded offspring dist branching processes}]
We firstly reduce to the case where the offspring distribution (of the continuous-time process) is bounded. We let $X_t$ be a single particle evolving according to the dynamics of our fixed right process. Then we have that
\[
P_t(x,\cdot)=\expE_x[e^{\int_0^tr(X_s)\sum_{m}mp_m(X_s)ds}\delta_{X_t}(\cdot)].
\]
By the boundedness of $r$ and the uniform integrability of the offspring distribution, for any $\epsilon>0$ we can choose $n_0<\infty$ such that $r(x)\sum_{m}mp_m(x)<\epsilon$ for all $x\in E$. If we now let $\tilde{P}_t$ be the transition semigroup for the process with the same driving motion but with offspring distribution
\[
\tilde{p}_n=\begin{cases}
p_n,\quad n<n_0,\\
\sum_{m\geq n}p_m,\quad n\geq n_0,
\end{cases}
\]
we see that
\[
\tilde{P}_t(x,\cdot)=\expE_x[e^{\int_0^tr(X_s)\sum_{m}(m\wedge n_0)p_m(X_s)ds}\delta_{X_t}(\cdot)]\geq e^{-\epsilon t}P_t(x,\cdot).
\]
Obviously, we can couple the two processes in such a way that $\bar X$ dominates the process with the truncated offspring distribution.
Therefore it suffices to prove Proposition \ref{prop:large minorising bounded offspring dist branching processes} in the case of bounded offspring distribution.

We henceforth assume that the offspring distribution is bounded. We let $n_0<\infty$ be an upper bound on the offspring distribution, so that $p_n(x)=0$ for all $x\in E$ and $n\geq n_0$. It suffices to construct the coupling over a single time-step.

We firstly take an $n_0$-ary branching tree $\mathcal{T}$ (without any spatial motion), branching at rate $\lvert\lvert r\rvert\rvert_{\infty}$ in continuous time over time $1$. We let $\mathbb V$ be the set of leaves of the tree. For $m<\infty$ to be determined we write $A_m$ for the event that $\lvert \mathbb V\rvert\leq m$. We now embed the process $\bar X$ into this tree as follows. We first determine the spatial motion along each branch of the tree according to our given right-process (for a given fixed initial position $x\in E$). Then, at every branch point, if the particle has position $x$, we remove branches (and the whole subtrees) randomly according to the following rule:
\begin{itemize}
    \item With probability $r(x)/\|r\|_\infty$, we sample a r.v. $N$ according to the offspring distribution and keep only $N$ of the $n_0$ branches, chosen randomly
    \item Otherwise (with probability $1-r(x)/\|r\|_\infty$), we keep only one branch chosen randomly and remove the others.
\end{itemize}
We let $\mathbb{V}_a$ be the set of living leaves leftover after pruning the tree, and write $Y_v$ for the position of the particle corresponding to each $v\in \mathbb{V}_a$. The result is a copy of our continuous-time branching process over a time step $1$, i.e.
\[
\sum_{v\in \mathbb{V}_a}\delta_{Y^v}\overset{d}{=}\sum_{i=1}^{N_1}\delta_{X^i_1},
\]
given the same initial distribution. Our goal is to show that $m<\infty$ may be chosen so that
\[
\expE_y[\Ind_{A_m}\sum_{v\in \mathbb{V}_a}\delta_{Y^v}(\cdot)]\geq (1-\epsilon)\expE_y[\sum_{i=1}^{N_1}\delta_{X^i_1}],
\]
for all possible initial positions $y\in E$.

We note that since we determine the spatial motions before pruning the tree, we in fact have determined a spatial position $Y_v$ for each $v\in \mathbb V$. Then we have that
\[
\begin{split}
\expE_y[\Ind(A_m)\sum_{v\in \mathbb{V}_a}\delta_{Y^v}(\cdot)]\geq \expE_y[\sum_{v\in \mathbb{V}_a}\delta_{Y^v}(\cdot)]-\expE_y[\Ind(A_m^c)\sum_{v\in \mathbb{V}}\delta_{Y^v}(\cdot)]\\
= \expE_y[\sum_{i=1}^{N_1}\delta_{X^i_1}(\cdot)]-\expE_y[\Ind(A_m^c)\sum_{v\in \mathbb{V}}\delta_{Y^v}(\cdot)].
\end{split}
\]
It therefore suffices to prove that $m<\infty$ may be chosen so that
\begin{equation}\label{eq:final sufficient criterion minorising discrete prop}
\expE_y[\Ind(A_m^c)\sum_{v\in \mathbb{V}}\delta_{Y^v}(\cdot)]\leq \epsilon \expE_y[\sum_{i=1}^{N_1}\delta_{X^i_1}]
\end{equation}
for all $y\in E$. We let $(X^0_t:0\leq t<\tau)$ be a copy of our right-process (without branching). Then we have that
\[
\expE_y[\Ind(A_m^c)\sum_{v\in \mathbb{V}}\delta_{Y^v}(\cdot)]=\expE_y[\expE_y[\Ind(A_m^c)\sum_{v\in \mathbb{V}}\delta_{Y^v}(\cdot)\mid \mathcal{T}]]=\expE[\lvert\mathbb V\rvert\times \Ind(A_m^c)]\times \expE_y[\delta_{X^0_1}(\cdot)\Ind(\tau>1)].
\]
Since there's no extraneous soft-killing rate (i.e. $p_0(x)=0$),
\[
\expE_y[\delta_{X^0_1}(\cdot)\Ind(\tau>1)]\leq \expE_y[\sum_{i=1}^{N_1}\delta_{X^i_1}]
\]
for all $y\in E$. Moreover $\expE[\lvert\mathbb V\rvert\times \Ind(A_m^c)]\ra 0$ as $m\ra\infty$. Therefore $m<\infty$ can be chosen so that we have \eqref{eq:final sufficient criterion minorising discrete prop}, and hence Proposition \ref{prop:large minorising bounded offspring dist branching processes} is established.
\end{proof}

We now use Proposition \ref{prop:large minorising bounded offspring dist branching processes} to conclude the proof of Part 2 of Theorem \ref{th:lambda'_survival}. We suppose that $\lambda'_{c,x}>0$. Then for all $\lambda<\lambda'_{c,x}$ there exists $u\in pB_1(E)$ such that $P_tu\geq e^{\lambda t}u$ for all $t\geq 0$, $u(x)>0$. We take arbitrary $\epsilon>0$, and use Proposition \ref{prop:large minorising bounded offspring dist branching processes} to take a minorising discrete-time branching process $\bar Y_n$ with mean kernel $Q$ (defined as in \eqref{eq:mean kernel Q}) satisfying $Q\geq (1-\epsilon)P_1$, so that
\[
Qu\geq (1-\epsilon)e^{\lambda}u.
\]
This implies that 
\[
\rho_{c,x}'(Q)\geq (1-\epsilon)e^{\lambda}, 
\]
where $\rho_{c,x}'$ is the discrete-time generalised principal eigenvalue defined in \eqref{eq:discrete time gen evalue 1}. Since $\lambda_{c,x}'>0$, $\lambda$ and $\epsilon$ may be chosen so that $(1-\epsilon)e^{\lambda}>0$, and in fact this can be made arbitrarily close to $e^{\lambda_{c,x}'}$. We now apply the discrete-time result Theorem \ref{th:rho'_survival} to the discrete-time process $\bar Y_n$, and use that it minorises $\bar X_n$, to see that
\[
\Pm_x(\glSurv)>0\text{ and }\Pm_x\left(\liminf_{n\to\infty} \frac 1 n \log N_n \ge \lambda\right) > 0.
\]
for all $\lambda <\lambda'_{c,x}$. 

Finally we wish to extend the above limit infimum along integer times to be along all times in $\Rm_{\geq 0}$. To do so it suffices to show that
\[
\Pm_x(\frac{1}{n}\log N_{n+1}\geq \lambda,\; \inf_{0\leq h\leq 1}\frac{1}{n}\log N_{n+h}\leq \lambda-\epsilon\quad\text{infinitely often})=0
\]
for any $\lambda<\lambda_{c,x}'$ and $\epsilon>0$. We fix $n$ and let $\tau:=\inf\{t>n:\frac{1}{n}\log N_{t}\leq \lambda-\epsilon\}\wedge n+1$. Since the offsping distribution and branching rate are bounded, we have
\[
\begin{split}
\Pm_x(\frac{1}{n}\log N_{n+1}\geq \lambda,\; \inf_{0\leq h\leq 1}\frac{1}{n}\log N_{n+h}\leq \lambda-\epsilon)=\Pm_x(N_{n+1}\geq e^{nh}N_{\tau})\\=\Pm_x(\frac{N_{n+1}}{N_{\tau}}\geq e^{nh})\leq \frac{\expE_x[\frac{N_{n+1}}{N_{\tau}}]}{e^{nh}}\leq \frac{C}{e^{nh}},
\end{split}
\]
for some $C<\infty$ not dependent upon $n$. We then conclude by the Borel-Cantelli lemma.

The above proof finally also implies that $\lambda'_{c,x} = \log \rho'_{c,x}(P_1)$, proving the last statement for $t=1$. The extension to general $t$ follows by a scaling argument or by adapting the above proof.
\qed

\subsubsection{Proof of Part 3 of Theorem~\ref{th:lambda'_survival}}

We could hope to use a martingale argument, using a $\lambda$-superharmonic function $u$. However, we don't assume regularity for $u$ which would guarantee that the martingale is cadlag. Hence, we revert to a more straightforward argument.

To prove the statement, we use Markov's inequality and the expression for $\lambda''_{c,x}$ in Theorem~\ref{th:lambda'_lambda''} to see that for all $\epsilon>0$ we have
\[
\Pm_x(\frac{1}{t}\ln N_t\geq \lambda_{c,x}''+\epsilon)=\Pm_x( N_t\geq e^{(\lambda_{c,x}''+\epsilon)t})\leq \frac{\expE_x[N_t]}{e^{(\lambda_{c,x}''+\epsilon)t}}\leq Ce^{-\frac{\epsilon}{2}t},
\]
for some $C<\infty$ which depends upon $\epsilon$ but not on $t<\infty$. Then by applying the Borel-Cantelli lemma we see that
\[
\Pm_x(\frac{1}{t}\ln N_t\geq \lambda_{c,x}''+\epsilon\quad\text{for infinitely many $t\in \mathbb{N}$})=0.
\]
By a similar argument, the boundedness of the branching rate allows us to show that 
\[
\Pm_x(\sup_{t\leq s\leq t+1}N_t\geq e^{\epsilon t}N_t\quad\text{for infinitely many $t\in \mathbb{N}$})=0
\]
for all $\epsilon>0$. This concludes the proof of the statement.
%
\qed

\subsection{Proof of Corollary \ref{cor:lambda'_survival}}
This proof is exactly the same as for the corresponding discrete-time result (Proposition~\ref{prop:rho dependence on x}), and we leave the details to the reader.

\subsection{Proof of Theorem \ref{theo:dichotomy Harris assumptions}}
The proof is similar to the corresponding discrete-time result.
We define
\[
A:=\{x\in E:\Pm_x(\glSurv)>0\}.
\]
Then if $\nu(A)=0$ we have $\Pm_x(\glSurv)=0$ for $\nu$-almost every $x\in E$. 

We now assume the converse, so that $\nu(A)>0$. Then by Assumption \ref{assum:Harris type irreducibility} we can, for every $x\in E$, find $t>0$ so that $\Pm_x(G) > 0$, where
\[
G = \{X^i_t\in A\text{ for some $1\leq i\leq N_t$}\} .
\]
Note that on the event $G$, we have
\(
\max_{i=1,\ldots,N_t}\Pm_{X^i_t}(\glSurv) > 0.
\)
This implies that 
\[
\Pm_x(\glSurv)\geq \expE_x[\max_{i=1,\ldots,N_t}\Pm_{X^i_t}(\glSurv)\mid G]\times \Pm_x[G] > 0.
\]
This concludes the proof of Theorem \ref{theo:dichotomy Harris assumptions}.\qed

\subsection{Proof of theorem \ref{theo:global extinction dichotomy}}

We let $u(x):=\Pm(\glSurv)$ and let $t>0$. We recall that our fixed right process is $X^0$ with possibly finite lifetime $\tau$. Then we have that
\[
\lvert u(x)-\expE_x[u(X^0_{t})\Ind(\tau>t)]\leq \Pm_x(\text{no branching in time $t$}).
\]
Since the branching rate is bounded, the right-hand side converges to $0$ uniformly in $x$ as $t\ra 0$. On the other hand we can write
\[
\expE_x[u(X^0_{t})\Ind(\tau>t)]=(P^0_tu)(x),
\]
which is continuous for all $t>0$ by the strong Feller property. Therefore there are two possibilities:
\begin{enumerate}
    \item $u\equiv 0$. We are then done.
    \item $u>0$ on some non-empty open subset $U\subset E$. In this case we then have $u>0$ everywhere by the topological irreducibility assumption. 
\end{enumerate}
We are done.\qed

\section{Branching diffusions and uniformly elliptic PDEs}
\label{sec:results_branching_diffusions}

In this section we consider the application of our results to branching diffusions and PDEs. Here we will be able to say somewhat more than is possible in the case of general continuous-time branching processes, as a result of the Harnack inequality in particular.

Let us begin with definitions and assumptions. We impose throughout this section the following standing assumption
\begin{assum}\label{assum:uniformly elliptic standing assumptions}
We assume that $E$ is a non-empty, open, connected subdomain of $\Rm^d$ with $C^{1,1}$ boundary, and denote its closure in $\Rm^d$ by $\bar E$. We consider general elliptic operators in nondivergence form:
\[
L u = \frac{1}{2} a_{ij}(x)\partial_{ij}u+b_i(x)\partial_iu+c(x)u,
\]
where we employ the usual convention of summation from $1$ to $d$ of repeated indices. We define
\[
\underline{a}(x):=\min_{\substack{\xi\in \Rm^d\\ \lvert \xi\rvert =1}}a_{ij}(x)\xi_i\xi_j,\quad x\in \bar E.
\]
We will assume that
\[
a_{ij}(x),\quad b_i(x),\quad c(x)
\]
are all globally bounded and locally H\"older continuous (with exponent $\alpha\in (0,1)$) on $E$, and don't depend upon time $t$. We further assume that $a$ is symmetric for all $x$ and uniformly positive definite, meaning that $\inf_x\underline{a}(x)>0$.
\end{assum}
We note that this is a narrower class of elliptic operators than was considered by \cite{BerestyckiRossi}. In particular we assume the coefficients are H\"older continuous, which is necessary to obtain Theorem \ref{theo:parabolic Feynman-Kac}. We believe that Theorem \ref{theo:parabolic Feynman-Kac} is true (with $u$ belonging to $W^{2,d}_{\loc}$, which suffices for our purpuses) if one instead assumes that $b$ is bounded and measurable and $a$ continuous, as is assumed in \cite{BerestyckiRossi}, but we couldn't find the requisite PDE result (an analogue of  \cite[Theorem 9, p.69]{Friedman1964}) to obtain a proof of this. Nevertheless our goal is to connect the Berestycki-Rossi eigenvalues with the global survival or extinction properties of branching processes in as straightforward a manner as possible, rather than trying to do so in as general a setting as possible.


In the following:
\begin{enumerate}
    \item $W^{2,d}_{\text{loc}}(E)$ is the Sobolev space of twice weakly differentiable functions $E\ra \Rm$ with $L^d_{\loc}$ derivatives;
    \item a function $u$ defined on an open subset $F$ of $\Rm^d\times \Rm$ is said to belong to $C^{2,1,\alpha}_{\loc}(F)$ if $u(x,t)$ is locally twice differentiable in space, once in time, and $u$, $D_xu$, $D^2_xu$ and $D_tu$ are all locally H\"older continuous (exponent $\alpha\in (0,1)$). Where there is no ambiguity as to $F$ we shall simply write $u\in C^{2,1,\alpha}_{\loc}$. The meaning of $u\in C^{2,\alpha}_{\loc}(E)$ is similar for functions $u:E\ra \Rm$.
\end{enumerate}

We recall that Berestycki and Rossi \cite{BerestyckiRossi} considered different notions of generalised principal eigenvalue for a uniformly elliptic second-order differential operator $L$ on a smooth and \textit{unbounded} domain $E\subseteq \Rm^d$, and established their relationship with each other. In particular they introduced the quantities
\begin{equation}\label{eq:Berestycki-Rossi eigenvalues elliptic section}
\begin{split}
\lambda_1'=\lambda_1'(L):=\sup\{\lambda\in\Rm:\exists u\in W^{2,d}_{\text{loc}}(E)\cap L^{\infty}(E)\text{ such that $(L-\lambda)u\geq 0,$}\\\text{$u(x)>0$ for all $x\in E$, $u(x)\ra 0$ as $x\ra \xi$ for all $\xi\in \partial E$}\},\\
\lambda_1''=\lambda_1''(L):=\inf\{\lambda\in\Rm:\exists u\in W^{2,d}_{\text{loc}}(E)\text{ such that $(L-\lambda)u\leq 0$, $\inf_{x\in E}u(x)>0$}\}.
\end{split}
\end{equation}
These extend the more classical notion of generalised principal eigenvalue, namely
\begin{equation}
\lambda_1=\lambda_1(L):=\inf\{\lambda\in\Rm:\exists u\in W^{2,d}_{\text{loc}}(E)\text{ such that $(L-\lambda)u\leq 0$, $u(x)>0$ for all $x\in E$}\}.
\end{equation}
We emphasise that the above notions of generalised principal eigenvalue differ from those found in Berestycki-Rossi \cite{BerestyckiRossi} by a factor of $-1$. 

The main goal of this section is to (1) show that these notions of generalised eigenvalue correspond to the notions we have defined in a more general context, and (2) show that the positivity or negativity of these eigenvalues correspond to different notions of global extinction or survival. The latter will extend the work of \cite{Englander2004}, who showed that $\lambda_1>0$ if and only if the corresponding branching diffusion enjoys local survival. Later we will connect these notions to the existence and uniqueness of stationary solutions of the corresponding FKPP equation.


We associate to this a branching process $\bar X_t=(X^1_t,\ldots,X^{N_t}_t)$ as follows. We take $\sigma$ such that $\sigma(x)\sigma^{\dagger}(x)=a(x)$. Each particle $X^i_t$ evolves in between branching and killing events independently as a weak solution of the SDE
\[
dX_t=b(x_t)dt+\sigma(X_t)dW_t,
\]
solutions of which are unique in law (by \cite[Chapter 7]{Stroock1997}). We then take non-negative, bounded and Borel measurable
\[
r(x),\; p_n(x), \quad n\geq 0,
\]
such that 
\[
\sum_np_n(x)\equiv 1\quad\text{and}\quad r(x)\sum_{n\geq 0}(n-1)p_n(x)\equiv c(x).
\]
Each particle in our branching process branches locally (i.e. the children are born at the parents' location) at rate $r(x)$, so that $n$ offspring produced independently with probability $p_n(x)$ and the parent particle is killed. If $n=0$ offspring are produced, this corresponds to \textit{soft killing}. We impose the additional assumption that the offspring distribution $(p_n(x))_{n\geq 0}$ is uniformly integrable, meaning that
\[
\sup_x\sum_{n\geq m}np_n(x)\ra 0
\]
as $m\ra\infty$. We kill particles instantaneously upon contact with $E^c$. Note that the boundedness of the coefficients precludes particles from escaping to infinity. We have now defined the branching process $(\bar X_t)_{t\geq 0}$, corresponding to which is the expectation semigroup
\[
P_t(x,\cdot):=\expE_x\Big[\sum_{i=1}^{N_t}\delta_{X_t}(\cdot)\Big].
\]

Theorem \ref{theo:dichotomy Harris assumptions} ensures that either 
\begin{equation}\label{eq:PDE global survival everywhere or nowhere}
    \text{$\Pm_x(\glSurv)>0$ for all $x\in E$ or $\Pm_x(\text{global extinction})=0$ for all $x\in E$.}
\end{equation} 
We recall that this requires establishing a strong Feller and topological irreducibility assumptions (Assumption \ref{assum:Strong Feller top irr}). The former is satisfied e.g. as an immediate consequence of the very general result \cite{Krylov2021}. Note that whereas \cite{Krylov2021} provides the strong Feller property for diffusions without killing, the presence of a killing boundary doesn't affect the strong Feller property by the same argument as in \cite[Proof of Lemma 6.4]{Fan2023}. Topological irreducibility is satisfied, for instance, by the Stroock-Varadhan support Theorem \cite[Theorem 3.1]{Stroock1972} (using here the connectedness of $E$).

Moreover we have a similar fact for $\limsup_{t\ra\infty}\frac{1}{t}\ln \expE_x[N_t]$.
\begin{theo}\label{theo:limsup mass doesn't depend on x}
We assume Assumption \ref{assum:uniformly elliptic standing assumptions}. Then $\limsup_{t\ra\infty}\frac{1}{t}\ln \expE_x[N_t]$ doesn't depend upon $x\in E$.
\end{theo}

For these reasons and in light of the results of Section \ref{sec:results_continuous_time}, we don't consider in this section eigenvalues which depend upon $x\in E$.

We now connect the Berestycki-Rossi eigenvalues given in \eqref{eq:Berestycki-Rossi eigenvalues elliptic section} with the eigenvalues we have defined in Section~\ref{sec:results_continuous_time}.
\begin{theo}\label{theo:PDE equality of eigenvalues different notion}
We assume Assumption \ref{assum:uniformly elliptic standing assumptions}. Then $\lambda_1'(L)=\lambda_c'((P_t)_{t\geq 0})$ and $\lambda_1''(L)=\lambda_c''((P_t)_{t\geq 0})$.
\end{theo}

\subsection*{Global survival and extinction of branching diffusions}

The following are immediate corollaries of Theorem \ref{theo:PDE equality of eigenvalues different notion} and the results of section \ref{sec:results_continuous_time}.

\begin{theo}\label{theo:elliptic regularity limsup of mass}
We assume Assumption \ref{assum:uniformly elliptic standing assumptions}. Then
\begin{equation}\label{eq:elliptic lim sup expectation}
\lambda'_1(L)\leq \liminf_{t\ra\infty}\frac{1}{t}\ln\expE_x[N_t]\leq  \limsup_{t\ra\infty}\frac{1}{t}\ln\expE_x[N_t]= \lambda''_1(L).
\end{equation}
for all $x\in E$. 
\end{theo}
Therefore $\expE_x[N_t]\ra 0$ if $\lambda_1''(L)<0$, and only if $\lambda_1''(L)\leq 0$. 


This (1) provides an alternative probabilistic proof of the fact, proven in \cite{BerestyckiRossi} that $\lambda'(L)\leq \lambda''(L)$, and (2) establishes that if $\lambda'_1(L)=\lambda''_1(L)$ then $\lim_{t\ra\infty}\frac{1}{t}\ln \expE_x[N_t]$ exists, doesn't depend upon $x\in E$, and is equal to $\lambda_1'(L)=\lambda_1''(L)$. This result could readily be interpreted as a statement about the mass of solutions of the corresponding Fokker-Planck equation, in the obvious manner.

We now connect the positivity or negativity of $\lambda'(L)$ with global survival or extinction of the associated branching process. 
\begin{theo}\label{theo:PDE global extinction e-value relationship}
We assume Assumption \ref{assum:uniformly elliptic standing assumptions}. Then we have the following:
\begin{enumerate}
    \item if also $\lambda'_1(L)>0$, then $\Pm_x(\glSurv)>0$ for all $x\in E$, and
    \begin{equation}\label{eq:liminf of ln of mass}
    \Pm_x(\liminf_{t\ra\infty}\frac{1}{t}\ln N_t\geq \lambda)>0
    \end{equation}
    for all $x\in E$ and $0<\lambda<\lambda'_1(L)$;
    \item conversely if we also have $\lambda'_1(L)<0$, then $\Pm_x(\text{global extinction})=1$ for all $x\in E$;
\item we have that
\begin{equation}\label{eq:elliptic almost sure lim sup}
    \Pm_x(\limsup_{t\ra\infty}\frac{1}{t}\ln  N_t\leq \lambda_1''(L))=1.
\end{equation}
for all $x\in E$.
\end{enumerate}
\end{theo}

The example given in Section \ref{section: example both global extinction and survival are possible at criticality} shows that if $\lambda_1'(L)=0$, then both global extinction and global survival are possible. 

\subsection*{Local survival and extinction}

Whereas this paper has principally converned itself with global survival, there is another notion of survival of a branching process - \textit{local survival} - which allows us to say more than \eqref{eq:liminf of ln of mass}. This notion of local survival will be important when we consider stationary solutions of the FKPP equation later in this section. Precisely, local survival is the following event.

\begin{defin}[Local survival]\label{defin:local suvival}
Given a non-empty, open, pre-compact subset $U\subset\subset E$, we say that the branching process $\bar X_t$ undergoes \textit{$U$-local survival}, denoted by $\locSurv_U$, if 
\[
\sup\{t:X^i_t\in U\text{ for some $1\leq i\leq N_t$}\}=+\infty.
\]
We say that $\bar X_t$ undergoas \textit{local survival}, denoted by $\locSurv$, if it undergoes $U$-local survival for every non-empty, open, pre-compact $U$, i.e. we define the event $\locSurv:=\bigcap_U\locSurv_U$. 
\end{defin}
It is not immediately clear that $\locSurv$ is an event, since the collection of all non-empty, open, pre-compact subsets $U\subset\subset E$ is uncountable. By taking a countable subset $\{x_1,x_2,\ldots\}$ and considering all balls $\bar B(x_i,r)\subset \subset E$ for $r\in \mathbb{Q}_{>0}$, we see that
\[
\locSurv=\bigcap_{B(x_i,r)\subset\subset E}\locSurv_{B(x_i,r)},
\]
hence $\locSurv$ is an event. Moreover under Assumption \ref{assum:uniformly elliptic standing assumptions}, the events $\locSurv_{U_1}$ and $\locSurv_{U_2}$ are equal up to $\Pm_x$-null sets for any two non-empty, open, pre-compact subsets $U_1,U_2\subset\subset E$, by a simple Borel-Cantelli argument. Therefore $\locSurv_U=\locSurv$ up to $\Pm_x$-null sets for any $x\in E$ and non-empty, open, pre-compact $U$. 

Furthermore we note that 
\begin{equation}
    \Pm_x(\locSurv)>0\quad\text{if and only if}\quad \Pm_{y}(\locSurv)>0
\end{equation}
for any $x,y\in E$, by exactly the same proof as the proof of the corresponding fact for global survival in \eqref{eq:PDE global survival everywhere or nowhere}. Kyprianou and Engl\"ander showed that $\Pm_x(\locSurv)>0$ if and only if $\lambda(L)>0$ \cite{Englander2004}, under additional assumptions - in particular they consider only binary branching with branching rate that is strictly positive somewhere, and that $L$ can be put into the divergence form $\frac{1}{2}\nabla \cdot(a\nabla)+b\nabla$. Whilst we believe the results of Kyprianou and Engl\"ander should be extendable to the setting considered here, we will not attempt to do that here. However we point out that, for $L$ such that both our results and those of Kyprianou-Englander can be applied, we have the following probabilistic proof that $\lambda(L)\leq \lambda'(L)$ (proven by analytic methods in \cite{BerestyckiRossi}). Suppose that $\lambda(L)>0$, so that $\Pm_x(\locSurv)>0$ for the corresponding binary branching process. Since local survival implies global survival, $\Pm_x(\glSurv)>0$ and hence $\lambda'(L)\geq 0$. Since we can always add or subtract a constant from $L$ resulting in adding or subtracting the same value from the eigenvalues, this suffices to prove 
\[
\lambda(L)\leq \lambda'(L).
\]
Nevertheless the proof that $\lambda(L)<0$ implies local extinction is very straightforward to prove in our setting. We will provide a proof of this, whish will be useful in our proofs (see Proposition \ref{prop:nve lambda implies local survival}).

The results of Kyprianou and Engl\"ander in fact allow us to establish an equivalence between local survival and the concept of ``persistence'' of FKPP equations, under the assumptions they impose. We will elaborate on this when we consider FKPP equations later on (see Theorem \ref{theo:persistence local survival equivalence}).

The following provides a convenient characterisation of local survival, which shall be useful later when we come to study stationary solutions of the FKPP equation. We say that the branching process undergoes $r$-local survival, denoted by $\rlocSurv$, if it survives globally and 
\begin{equation}
\liminf_{t\ra\infty}\inf_{1\leq i\leq N_t}\lvert X^i_t\rvert<\infty.
\end{equation}
Clearly local survival implies $r$-local survival, but the converse is not immediate.

In order for the branching process to undergo both $r$-local survival and local extinction, it would need to return to $B(x_0,r)\cap E$ infinitely often for some $x_0\in E$, $r>0$, but do so by getting closer and closer to the boundary $\partial E$ so as not to visit any given  compact subset infinitely often. Whilst it seems intuitively obvious that this should be impossible (up to $\Pm_x$-null sets), this is surprisingly tricky to prove. In the end we have only been able to prove this under the additional assumption that the drift is identically equal to $0$. This assumption is necessary as our proof makes use of the parabolic boundary Harnack inequality for non-divergence form operators \cite{Torres-Latorre2024}, and this has only been proven for identically $0$ drift. Moreover we expect this equivalence to be valid under little to no assumptions on the boundary regularity, but this is not the case for boundary Harnack inequalities \cite{Bass1991,Torres-Latorre2024}.
\begin{theo}\label{theo:local survival equivalence}
We assume Assumption \ref{assum:uniformly elliptic standing assumptions}, and that $b\equiv 0$. Then the events $\locSurv$ and $\rlocSurv$ are equal up to $\Pm_x$-null sets, for any $x\in E$.
\end{theo}

The notion of local survival allows us to extend \eqref{eq:liminf of ln of mass} as follows.
\begin{theo}\label{theo:local survival gives liminf of number of particles}
We assume Assumption \ref{assum:uniformly elliptic standing assumptions}, that $\lambda_1'(L)>0$, and $\Pm_x(\locSurv)>0$ for some given $x\in E$. Then 
\begin{equation}
    \Pm_x(\liminf_{t\ra\infty}\frac{1}{t}\ln N_t\geq \lambda_1'(L) \mid \locSurv)=1.
\end{equation}
\end{theo}

\subsection*{A relationship between $\lambda_1(L)$ and $\lambda_1'(L)$}

Given a (non-empty) subdomain $E'\subseteq E$, we write $\lambda'(L;E')$ and $\lambda''(L;E')$ for the corresponding eigenvalues when $L$ is defined only on the subdomain $E'$. Berestycki and Rossi established the following relationship between $\lambda(L)$ and $\lambda''_1(L)$ \cite[Theorem 7.6]{BerestyckiRossi} (for $E$ unbounded and smooth),
\[
\lambda''(L)=\max(\lambda(L),\lim_{r\ra\infty}\lambda''(L;E\setminus \bar B(0,r))).
\]

We now provide an analous result for $\lambda'(L)$, using the aforedescribed results on local survival.
\begin{prop}\label{prop:lambda' lambda'' outside balls relationship}
We assume Assumption \ref{assum:uniformly elliptic standing assumptions}, and that $b\equiv 0$. Then \[
\lambda'(L)=\max(\lambda(L),\lim_{r\ra\infty}\lambda'(L;E\setminus \bar B(0,r))).
\]
\end{prop}
\begin{rmk}
As stated before, we believe that Theorem \ref{theo:local survival equivalence} remains true when $b\nequiv 0$, but can't prove it because we don't have available the necessary boundary Harnack inequality. This is the only reason we need to stipulate $b\equiv 0$ in Proposition \ref{prop:lambda' lambda'' outside balls relationship}, so we believe it remains true without this stipulation.
\end{rmk}
\begin{proof}[Proof of Proposition \ref{prop:lambda' lambda'' outside balls relationship}]
Suppose that $\lambda'(L)>0$. Then $\Pm_x(\glSurv)>0$ by Theorem \ref{theo:PDE global extinction e-value relationship}. There are then two possibilities: \begin{enumerate}
    \item $\Pm_x(\locSurv)>0$, or else 
    \item $\Pm_x(\glSurv\setminus \locSurv)>0$.
\end{enumerate}  

Englander and Kyprianou established that $\Pm_x(\locSurv)> 0$ implies that $\lambda(L)\geq 0$ under somewhat different conditions than we consider here. We will check the same proof works under Assumption \ref{assum:uniformly elliptic standing assumptions} in Proposition \ref{prop:nve lambda implies local survival}. Therefore $\lambda(L)\geq 0$ on the first possiblility.

On the other hand if $\Pm_x(\glSurv\setminus \locSurv)>0$, then Theorem \ref{theo:local survival equivalence} tells us that 
\[
\Pm_x(\glSurv\setminus \rlocSurv)>0.
\]
This implies that we have a positive probability of global survival on $E\setminus \bar B(0,r)$ when we start at $x\in E\setminus \bar B(0,r)$, for any $r<\infty$. It then follows from Theorem \ref{theo:PDE global extinction e-value relationship} that $\lim_{r\ra\infty}\lambda'(L,E\setminus \bar B(0,r))\geq 0$. 

Therefore $\lambda'(L)>0$ implies that
\[
\max(\lambda(L),\lim_{r\ra\infty}\lambda'(L;E\setminus \bar B(0,r)))\geq 0.
\]
By adding or subtracting a constant from $L$, we obtain that
\[
\lambda'(L)\leq \max(\lambda(L),\lim_{r\ra\infty}\lambda'(L;E\setminus \bar B(0,r))).
\]
On the other hand, we know that $\lambda'(L;E\setminus \bar B(0,r))\leq \lambda'(L)$ for all $r<\infty$, and $\lambda(L)\leq \lambda'(L)$ (by Berestycki-Rossi \cite[Theorem 1.7 (ii)]{BerestyckiRossi}). The claim is therefore proven.
\end{proof}



In particular, if we have Assumption \ref{assum:uniformly elliptic standing assumptions}, $\lambda_1''(L)=\lambda_1'(L)>0$ and $\Pm_x(\locSurv)>0$ for some given $x\in E$, then
\begin{equation}
    \Pm_x(\lim_{t\ra\infty}\frac{1}{t}\ln N_t= \lambda'(L)\lvert \locSurv)=1.
\end{equation}

\subsection*{Conditions for $\lambda_1'(L)=\lambda_1''(L)$}

Berestycki and Rossi conjectured in \cite[Conjecture 1.8]{BerestyckiRossi} that $\lambda_1'(L) = \lambda_1''(L)$ under mild assumptions. We refute this in the example given in Section \ref{section:lambda' lambda'' can be different}. They were, however, able to prove this when $L$ is self-adjoint and either $d=1$ or $L$ is radially symmetric and $E=\Rm^d$ \cite[Theorem 1.9]{BerestyckiRossi}. We will now generalise this to general self-adjoint $L$. We are aware that H. Berestycki, J. Berestycki and Graham are currently and independently working on this problem (of showing that $\lambda'_1(L)=\lambda''_1(L)$ for self-adjoint $L$) via PDE methods.

At the same time we will prove that $\lambda_1(L)=\lambda'_1(L)$ for free. Whilst this latter fact was proven by Berestycki-Rossi \cite{BerestyckiRossi} using PDE methods, we wish to exhibit a probabilistic proof.

\begin{theo}\label{theo:lambda'=lambda'' condition 1}
We suppose that Assumptions \ref{assum:uniformly elliptic standing assumptions} is satisfied, that $a$ is $C^1$, and that $L$ can be put into the divergence form $Lu=\frac{1}{2}\partial_i(a_{ij}(x)\partial_ju)+c(x)u$. Then $\lambda(L)=\lambda_1'(L)=\lambda_1''(L)$.
\end{theo}

Here $L$ is symmetric with respect to Lebesgue measure. In fact, we can generalise this as follows. Following \cite[Chapter 4]{Baudoin2014}, we will assume that $L$ is symmetric with respect to a smooth measure $\mu$, which we further assume grows subexponentially. Precisely, we consider the following assumption.
\begin{assum}\label{assum:uniformly elliptic symmetric}
We suppose that $\mu$ is a given Borel measure $\mu$ which has a $C^{\infty}(E)$ and everywhere strictly positive (but not necessarily uniformly positive) density with respect to Lebesgue measure on $E$. We assume that $L$ has $C^{\infty}(E)$ coefficients and is symmetric with respect to $\mu$ in the sense that
\[
\int_{E}gLfd\mu=\int_EfLgd\mu
\]
for all $f,g\in C_c^{\infty}(E)$. Finally, we assume that
\[
\limsup_{R\ra\infty}\frac{1}{R}\ln\mu(B(0,R)\cap E)\leq 0,
\]
where $x_0$ is a fixed point in $E$.
\end{assum}
We note that generalising the $\mu$ with respect to which $L$ is symmetric has come at the cost of having to impose additional regularity on the coefficients.\commentout{
\begin{example}
If $L$ has $C^{\infty}(E)$ coefficients and can be put into the divergence form,
\[
Lu=\nabla\cdot(a \nabla u) +cu,
\]
then $L$ is symmetric with respect to Lebesgue measure, which satisfies the subexponential growth condition. 
\end{example}}
\begin{example}
Suppose that $E=\Rm_{>0}$, $a\equiv 1$, $b,c\in C^{\infty}(E)\cap B_b(E)$, and $b$ satisfies
\[
\limsup_{x\ra\infty}\frac{1}{x}\int_0^xb(y)dy\leq 0.
\]
Then $L$ is symmetric with respect to $\mu$ defined by
\[
\frac{d\mu(x)}{dx}=\exp\Big(2\int_0^xb(y)dy\Big),
\]
which satisfies the subexponential growth property.
\end{example}
\begin{theo}\label{theo:lambda'=lambda'' condition}
We suppose that Assumptions \ref{assum:uniformly elliptic standing assumptions} and \ref{assum:uniformly elliptic symmetric} are satisfied. Then $\lambda_1'(L)=\lambda_1''(L)$.
\end{theo}

\subsection*{Maximum principle}
Berestycki and Rossi \cite{BerestyckiRossi} defined the maximum principle to hold if
\begin{equation}\label{eq:maximum principle PDE}
Lu\geq 0,\quad \limsup_{x\rightarrow \xi}u(x)\leq 0\text{ for all $\xi\in \partial \Omega$}, \quad \sup u<\infty
\end{equation}
implies that $u\leq 0$.

They showed that
\begin{equation}
\lambda_1'' <0\Rightarrow \text{maximum principle}\Rightarrow \lambda_1'\leq 0,
\end{equation}
the former under the assumption of a growth condition. Note that the latter is immediate from the definition of $\lambda'_1$, since if $\lambda_1'>0$ we can take $u$ violating the maximum principle.

When we have $\lambda_1'(L)=\lambda_1''(L)$, we then obtain a sharp characterisation of the validity of the maximum principle:
\[
\text{the maximum principle holds if $\lambda_1'(L)<0$ but not if $\lambda_1'(L)>0$. }
\]

We now provide a probabilistic proof of the fact that $\lambda_1''<0$ implies the maximum principle.\commentout{

As previously stated, the implication
\[
\text{maximum principle}\Rightarrow \lambda_1'\leq 0
\]
is immediate, since if $\lambda_1'>0$ then we can take $u$ violating the maximum principle.

We now seek to prove that if $\lambda_1''<0$ then the maximum principle holds.} We suppose that \eqref{eq:maximum principle PDE} holds for some $u\in W^{2,d}_{\loc}$. We then observe that $\sum_{i=1}^{N_t}u(X^i_t)$
is a local submartingale by Ito's lemma \cite[Krylov, Controlled diffusions, p.122 and 47]{Krylov1980} and the fact $\limsup_{x\ra\xi}u(x)\leq 0$ for $\xi\in \partial E$. Therefore (using that $\sup u<\infty$ and the offspring distribution is bounded)
\[
u(x)\leq \expE_{x}[\sum_{i=1}^{N_t}u(X^i_t)\Ind(u(X^i_t)\geq 0)]\leq (\sup_{x\in E}u(x)\vee 0)\expE_x[N_t].
\]
Since $\lambda_{1}''<0$, Theorem \ref{theo:elliptic regularity limsup of mass} implies that $\expE_x[N_t]\ra 0$ as $t\ra\infty$. Therefore $u(x)\leq 0$.\qed


\subsection*{Stationary solutions of the FKPP equation}

We associate to the branch rate and offspring distribution the nonlinear term
\begin{equation}\label{eq:nonlinearity f}
f(u,x):=r(x)\sum_{n=0}^{\infty}p_n(x)[(1-u)-(1-u)^n].
\end{equation}
For instance associated to binary branching at rate $1$ we have the classical FKPP nonlinearity $u-u^2$.

It is well-known since the work of Skorokhod \cite{Skorokhod1964} (see also the work of McKean \cite{McKean1975}) that branching Brownian motion is dual to the classical FKPP equation
\[
\partial_tu=\frac{1}{2}\Delta u+u(1-u),
\]
a connection which is known to be valid in much greater generality. An exposition of this can be found in \cite[Section 1.2.2]{Ryzhik2023}). This connection is made in very great generality in \cite{Ikeda1966}, where the FKPP is formulated in an abstract manner. Nevertheless we couldn't find a reference where the connection is stated and proved in precisely the setting we wish to employ here, so have provided a statement and proof in the appendix (see Theorem \ref{theo:McKean rep}).

We will now connect notions of global survival to the existence of $(0,1]$-valued solutions of the semilinear elliptic equation
\begin{equation}\label{eq:semilinear elliptic equation}
L_0u+f(u,x)=0,\quad x\in E,
\end{equation}
whereby 
\[
L_0u:=\frac{1}{2}a_{ij}(x)\partial_{ij}u+b_i(x)\partial_iu
\]
Note that for the classical FKPP nonlinearity $f(u)=u(1-u)$, the equation \eqref{eq:semilinear elliptic equation} can be written as
\[
(L_0-1)(u-1)=(u-1)^2.
\]
Since $\lambda''_1(L_0-1)=\lambda''_1(L_0)-1\leq -1$ and $(u-1)^2\geq 0$, we can apply the maximum principle \eqref{eq:maximum principle PDE} to conclude there are no bounded positive solutions with $u\geq 1$ anywhere, hence $(0,1]$-valued solutions and positive, bounded solutions are equivalent. On the other hand, if we instead consider $E=\Rm^d$, $L_0=\frac{1}{2}\Delta$, and the nonlinearity $f(u)=(1-u)-(1-u)^3$, we observe that $u\equiv 2$ is a positive, bounded solution of \eqref{eq:semilinear elliptic equation}. Therefore the specification of $(0,1]$-valued solution, rather than positive and bounded solution, will be necessary. 

The existence and uniqueness of positive, bounded solutions to the semilinear elliptic equations for different classes of nonlinear term $f$, and in particular for FKPP-type nonlinearities, has received a great deal of attention in the PDE literature both contemporaneously and classically \cite{Rabinowitz1971,Berestycki1981,Berestycki1997,Esteban1982,Berestycki2023,Berestycki2025}. However, to our knowledge, these questions for FKPP-type nonlinearities have previously been addressed exclusively from a PDE perspective. The following provides a probabilistic perspective to these questions. We note a very different probabilistic connection has been established in \cite{Dynkin1991,LeGall1999} for nonlinearites of the different form $f(u)=u^p$.

\begin{theo}\label{theo:PDE semilinear elliptic equation global survival}
We assume Assumption \ref{assum:uniformly elliptic standing assumptions}. Then we have $\Pm_x(\glSurv)>0$ for some (hence all) $x\in E$ if and only if there exists a $(0,1]$-valued solution of \eqref{eq:semilinear elliptic equation}, where $f$ is given by \eqref{eq:nonlinearity f}. In the case that this exists then
\[
v(x):=\Pm_x(\glSurv)
\]
is the maximal $(0,1]$-valued solution of \eqref{eq:semilinear elliptic equation}, in the sense that $\tilde{u}(x)\leq v(x)$ for any other $(0,1]$-valued solution $\tilde{u}$ of \eqref{eq:semilinear elliptic equation}.
\end{theo}
The example given in Section \ref{section: example both global extinction and survival are possible at criticality} demonstrates that both $\Pm_x(\glSurv)=0$ and $\Pm_x(\glSurv)>0$ are possible when $\lambda'(L)=0$. Therefore Theorem \ref{theo:PDE semilinear elliptic equation global survival} implies that both existence and non-existence of $(0,1]$-valued solutions to \eqref{eq:semilinear elliptic equation} are possible when $\lambda'(L)=0$. It remains open to determine whether $\Pm_x(\glSurv)>0$ is possible when $L$ is self-adjoint and $p_0\equiv 0$ (note that the simple example of binary branching Brownian motion on a ball of critical radius demonstrates that global $\Pm_x(\glSurv)=0$ is possible).

By combining Theorems \ref{theo:PDE semilinear elliptic equation global survival} and \ref{theo:PDE global extinction e-value relationship}, we obtain the following corollary.
\begin{cor}
There exists a $(0,1]$-valued solution of \eqref{eq:semilinear elliptic equation} if $\lambda_1'(L)>0$, but not if $\lambda_1'(L)<0$.
\end{cor}

We now connect local survival to the uniqueness of solutions to \eqref{eq:semilinear elliptic equation}.
\begin{theo}\label{theo:local survival minimal solution}
We suppose that $\Pm_x(\locSurv)>0$ for some (hence all) $x\in E$, and that $r(x)p_n(x)>0$ for some $x\in E$ and some $n\geq 2$ (i.e. that branching is possible somewhere). Then
\[
w(x):=\Pm_x(\locSurv)
\]
is a $(0,1]$-valued solution of \eqref{eq:semilinear elliptic equation}, which is minimal in the sense that
\[
w(x)\leq \tilde{u}(x)
\]
for any other solution $\tilde{u}$ of \eqref{eq:semilinear elliptic equation}.

In particular, the following are equivalent:
\begin{enumerate}
    \item there is a unique $(0,1]$-valued solution $u$ of \eqref{eq:semilinear elliptic equation};
    \item $\Pm_x(\locSurv \lvert \glSurv)=1$.
\end{enumerate}
\end{theo}

We summarise the above as follows:
\begin{enumerate}
    \item If $\Pm_x(\glSurv)=0$ then there are no $(0,1]$-valued solutions of \eqref{eq:semilinear elliptic equation}.
    \item If $\Pm_x(\glSurv)>0$ but $\Pm_x(\glSurv\setminus \locSurv)=0$ then there is a unique $(0,1]$-valued solution of \eqref{eq:semilinear elliptic equation} given by $x\mapsto \Pm_x(\glSurv)$.
    \item If $\Pm_x(\locSurv),\Pm_x(\glSurv\setminus \locSurv)>0$, then there are multiple $(0,1]$-valued solutions of \eqref{eq:semilinear elliptic equation}, including both a minimal and maximal one. The minimal one is given by $x\mapsto \Pm_x(\locSurv)$ and the maximal one by $x\mapsto \Pm_x(\glSurv)$.
\end{enumerate}

The above is restricted to nonlinear terms of the form \eqref{eq:nonlinearity f}. In \cite{Etheridge2017}, Etheridge, Freeman and Penington extended McKean's representation for BBM to provide a stochastic representation of the Allen-Cahn equation using ternary branching Brownian motion and implementing a voting procedure on it. This was generalised by An, Henderson and Ryzhik in \cite{An2023} to provide a stochastic representation for seimilinear parabolic PDEs with quite general nonlinearities. This therefore begs the question as to whether we can combine these ideas with those described above to provide a probabilistic interpretation of semilinear elliptic PDEs with nonlinear terms which aren't of the form \eqref{eq:nonlinearity f}.

\subsection*{Persistence, local survival and global survival}

We let $u(x,t)$ be the solution to the FKPP equation \eqref{eq:semilinear elliptic equation}, and $\bar X_t=(X^1_t,\ldots,X^{N_t}_t)$ the associated branching process. Then it follows straightforwardly from the stochastic representation of \eqref{eq:semilinear elliptic equation} (see Theorem \ref{theo:McKean rep}) that $\Pm_x(N_t>0)=u(x,t)$ where $u(x,0)\equiv 1$. Therefore $\Pm_x(\glSurv)=\lim_{t\ra\infty}u(x,t)$ with $u(x,0)\equiv 1$. Theorem \ref{theo:PDE semilinear elliptic equation global survival} therefore has the PDE interpretation of saying that \eqref{eq:semilinear elliptic equation} has a $(0,1]$-valued stationary solution if and only if it has a long-time limit when we take $u(x,0)\equiv 1$ as the initial condition, in which case this long-time limit is the maximal stationary solution.

The PDE interpretation of \eqref{theo:local survival minimal solution} is not so immediate, however. Persistence is the property that $\liminf_{t\ra\infty}u(x,t)>0$ whenever $u(x,0)\nequiv 0$ and $u(x,0)$ is continuous with compact support. We would like to interpret $\Pm_x(\locSurv)>0$ as being equivalent to persistence, analogously to our previous interpretation of global survival. In order to make this connection, we will need to employ the results of Kyprianou and Engl\"ander \cite{Englander2004}. We prove the following.
\begin{theo}\label{theo:persistence local survival equivalence}
In addition to Assumption \ref{assum:uniformly elliptic standing assumptions}, we further assume that:
\begin{enumerate}
    \item branching is purely binary, so that $p_2(x)\equiv 1$, and the branch rate is positive somewhere ($r\nequiv 0$);
    \item $L_0$ can be put into the divergence form $L_0=\nabla\cdot (\tilde{a}\nabla)+\tilde{b}\nabla$ for some $\tilde{a}$ and $\tilde{b}$ belonging to $C^{1,\alpha}$ for some $\alpha\in (0,1]$ (note that $\tilde{a}=a$ necessarily).
\end{enumerate}
Then $\Pm_x(\locSurv)>0$ if and only if we have persistence for the FKPP equation
\[
\partial_tu=L_0u+u(1-u).
\]
\end{theo}\commentout{
\begin{example}
Consider the FKPP equation 
\begin{equation}\label{eq:FKPP drift 1 branching 1}
    \partial_tu=\frac{1}{2}\Delta u+2\partial_xu+u(1-u),
\end{equation}
corresponding to which is branching Brownian motion with drift $2$. Since we have $\Pm_x(\glSurv)=1>0$, Theorem \ref{theo:PDE semilinear elliptic equation global survival} ensures the existence of a stationary solution to \eqref{eq:FKPP drift 1 branching 1}. On the other hand, $\Pm_x(\locSurv)=0$ (since the drift is larger than the asymptotic spreading speed $\sqrt{2}$ of branching Brownian motion), hence we don't have persistence by 

\end{example}}
\begin{proof}[Proof of Theorem \ref{theo:persistence local survival equivalence}]
We assume that $u(x,0)$ is continuous with compact support. We recall from Theorem \ref{theo:McKean rep} that
\[
u(x,t)=\expE_x[1-\prod_{i=1}^{N_t}(1-u(x,0))].
\]
We take $U$ to be a non-empty, open, pre-compact subset compactly contained within the interior of the support of $u(x,0)$, and let $K:=\text{spt}(u(x,0))$. We define $\underline{u}_0:=\inf\{u(x):x\in U\}>0$, and write $N_t(U)$ and $N_t(K)$ for the number of particles in $U$ and $K$ respectively at time $K$. Then we observe that:
\begin{enumerate}
    \item if $N_t(K)=0$ then $1-\prod_{i=1}^{N_t}(1-u(x,0))=0$;
    \item $1-\prod_{i=1}^{N_t}(1-u(x,0))\geq 1-(1-\underline{u}_0)^{N_t(U)}$.
\end{enumerate}
Therefore
\[
\Ind(N_t(K)>0)\geq 1-\prod_{i=1}^{N_t}(1-u(x,0))\geq 1-(1-\underline{u}_0)^{N_t(U)},
\]
so that
\[
\limsup_{t\ra\infty}\Pm(N_t(K)>0)\geq \limsup_{t\ra\infty}u(x,t)\geq \liminf_{t\ra\infty}u(x,t)\geq \Pm(N_t(U)\ra \infty\text{ as $t\ra\infty$}).
\]

When $\lambda(L)\leq 0$, \cite[Theorem 3]{Englander2004} ensures that $\Pm_x(\locSurv)=0$, and hence 
\[
\limsup_{t\ra\infty}\Pm(N_t(K)>0)=0
\]
Therefore $u(x,t)\ra 0$ as $t\ra\infty$.

If $\lambda(L)>0$, \cite[Theorem 3]{Englander2004} ensures that $\Pm_x(\locSurv)>0$. However local survival is the event that the branching process visits any given compact set infinitely often, but it's not clear that it can't do so by visiting the set at rare times. We therefore need more than $\Pm_x(\locSurv)>0$. The results of Kyrpianou and Engl\"ander \cite[Theorem 3]{Englander2004} ensure that 
\[
\Pm_x(N_t(U)\ra \infty\text{ as $t\ra\infty$})>0
\]
when $\lambda(L)>0$. Therefore if $\lambda(L)>0$ we have $\Pm_x(\locSurv)>0$ and $\liminf_{t\ra\infty}u(x,t)>0$. 

In both the cases $\lambda(L)>0$ and $\lambda(L)\leq 0$ we have Theorem \ref{theo:persistence local survival equivalence}, so we're done.
\end{proof}

\subsection*{Stationary solutions of the FKPP equation on bounded domains}

In contrast to the rest of this section, here we shall consider bounded $E$, and consider whether the boundary regularity affects the existence and uniqueness of solutions to \eqref{eq:semilinear elliptic equation}. For that reason, the assumption that $\partial E$ is $C^{1,1}$ will be dropped for the time being. 

Rabinowitz \cite{Rabinowitz1971} established the uniqueness of solutions to the semilinear elliptic equation
\[
Lu+f(u,x)=0,\quad x\in E,\quad u\equiv 0\text{ on $\partial E$,}
\]
for self-adjoint $L$ and $f$ including in particular FKPP-type nonlinearities, under the assumption that $\partial E$ is smooth. This was extended to non self-adjoint $L$ by Berestycki \cite{Berestycki1981}. However, these proofs use an interior ball property in an essential way (so that they can employ the Hopf lemma). Nevertheless Berestycki and Graham conjectured in \cite[Conjecture 4.2]{Berestycki2025} that solutions of
\[
\frac{1}{2}\Delta u+f(u)=0,\quad x\in E,\quad u\equiv 0\text{ on $\partial E$,}
\]
are unique if $\partial E$ is only assumed to be Lipschitz, and $f$ is of strong-FKPP type (see \cite[p.2]{Berestycki2025} for a definition of this).

We will now confirm this for $f(u)=u(1-u)$, without any assumptions on the boundary regularity. It will be readily apparent that the argument generalises to general non-divergence form $L_0$ and $f$ of the form \eqref{eq:nonlinearity f}.

We must firstly say what we mean by Dirichlet boundary conditions on a bounded domain $E$ with no boundary regularity. In \cite[Section 3]{Berestycki1994}, Berestycki, Nirenberg and Varadhan constructed a solution, $u_0$, of
\[
\frac{1}{2}\Delta u_0=-1,\quad x\in E,
\]
and which satisfies $u_0=0$ on $\partial E$ in some sense. They use this as a barrier function, saying that $u\overset{u_0}{=}0$ on $\partial E$ if $u_0(x_n)\ra 0$ implies $u(x)\ra 0$ as $n\ra\infty$. 

In fact, $u_0$ has the following interpretation. Let $(B_t:0\leq t<\tau_E)$ be a Brownian motion killed at the time $\tau_E:=\inf\{t>0:B_t\notin E\}$, with $\Pm_x(B_0=x)=1$. Then 
\begin{equation}\label{eq:u0 is exit time}
    u_0(x)=\expE_x[\tau_E].
\end{equation}
To see this, we recall that $u_0$ is constructed by taking an ascending sequence $E_j$ of smooth sub-domains of $E$ with union $E$, and defining $u_j$ to be the unique solution of $\frac{1}{2}\Delta u=-1$ on $E_j$ with Dirichlet boundary conditions defined in the usual way. They define $u_0$ to be the pointwise non-decreasing limit of the $u_js$. Note however that $u_j(x)$ has the interpretation of being the expectation of the exit time started from $x$ of $E_j$, so \eqref{eq:u0 is exit time} follows from the monotone convergence theorem.

We consider the existence and uniqueness of solutions of the following stationary FKPP equation,
\begin{equation}\label{eq:stationary FKPP equation general bounded domain}
\frac{1}{2}\Delta u+u(1-u)=0,\quad x\in E,\quad u\overset{u_0}{=} 0.    
\end{equation}
We characterise the existence of solutions to \eqref{eq:stationary FKPP equation general bounded domain} using $\lambda_1'$ in the same manner as before, where $\lambda'_1$ is defined using the barrier function $u_0$ as follows,
\begin{equation}
\lambda_1' :=\sup\{\lambda\in\Rm:\exists u\in W^{2,d}_{\text{loc}}(E)\cap L^{\infty}(E)\text{ such that $(\frac{1}{2}\Delta-\lambda)u\geq 0,$}\\\text{$u(x)>0$ for all $x\in E$, $u\overset{u_0}{=}0$}\}.
\end{equation}

We prove the following.
\begin{theo}\label{eq:FKPP general bounded domain exis uniq}
\begin{enumerate}
    \item There exists a solution to \eqref{eq:stationary FKPP equation general bounded domain} if $\lambda'_1>0$, but not if $\lambda'_1<0$.
    \item There is at most one solution to \eqref{eq:stationary FKPP equation general bounded domain}.
\end{enumerate}
\end{theo}

\begin{proof}[Proof of Theorem \ref{eq:FKPP general bounded domain exis uniq}]
We note that we didn't use boundary regularity to check Assumption \ref{assum:Strong Feller top irr}, so \eqref{eq:PDE global survival everywhere or nowhere} remains valid.

We will need the following lemma.
\begin{lem}\label{lem:hitting of u0 of Brownian motion}
For all $x\in E$ we have $\lim_{t\uparrow \tau_E}u_0(B_{t})=0$ $\Pm_x$-almost surely.
\end{lem}
We defer the proof of Lemma \ref{lem:hitting of u0 of Brownian motion} for later.

Using Lemma \ref{lem:hitting of u0 of Brownian motion}, we can then repeat the proof that $\lambda'_1(\frac{1}{2}\Delta+1)=\lambda'_1((P_t)_{t\geq 0})$ (as in Theorem \ref{theo:PDE equality of eigenvalues different notion}), as well as the proofs of Theorems \ref{theo:PDE semilinear elliptic equation global survival} and \ref{theo:local survival minimal solution}, exactly as before.

The existence part of Theorem \ref{eq:FKPP general bounded domain exis uniq} then follows from Theorem \ref{theo:global extinction dichotomy}.

We now turn to the proof of uniqueness. It suffices to check that $\Pm_x(\glSurv\setminus\locSurv)=0$ for all $x\in E$. Given $E'\subseteq E$ we consider BBM in $E'$ killed upon exiting $E'$. We write $\glSurv_{E'}$ for the event of global survival and $N^{E'}_t$ for the number of particles at time $t$. 

We assume for contradiction that $\Pm_x(\glSurv\setminus\locSurv)>0$, which implies that $\Pm_x(\glSurv_{E\setminus K})>0$ for all compact $K\subset \subset E$ and $x\in E\setminus K$. Since the heat kernel has bounded density with respect to Lebesgue measure at time $1$, there exists $C<\infty$ not dependent upon $K$ such that \[
\expE_x[N^{E\setminus K}_1]\leq Ce^t\Leb(E\setminus K)
\]
for all $x\in E\setminus K$. By taking an ascending sequence of compact $K\subset\subset E$, we see (from the dominated convergence theorem) that we can take compact $K\subset\subset E$ such that $\expE_x[N^{E\setminus K}_1]\leq \frac{1}{2}$ for all $x\in E\setminus K$. Using Markov's inequality and the Borel-Cantelli lemma, it follows that $\Pm_x(\glSurv_{E\setminus K})= 0$. This is a contradiction, hence we have established the uniqueness of solutions to \eqref{eq:stationary FKPP equation general bounded domain}.

We finally turn to the proof of Lemma \ref{lem:hitting of u0 of Brownian motion}. We define
\[
M_t:=t\wedge \tau_E+u_0(B_{t})\Ind(\tau_E>t)
\]
We firstly establish that $M_t$ is a martingale. To see this, observe that
\[
\begin{split}
\expE[M_t\lvert \mathcal{F}_s]=\expE[t\wedge \tau_E+\expE[(\tau_E-t)\vee 0\lvert \mathcal{F}_t]\lvert \mathcal{F}_s]\\
=\expE[t\wedge \tau_E+(\tau_E-t)\vee 0\lvert \mathcal{F}_s]=\expE[s\wedge \tau_E+(\tau_E-s)\vee 0\lvert \mathcal{F}_s]=M_s
\end{split}
\]
for $0\leq s<t$.

We now write $\tau_n:=\inf\{t>0:d(B_t,E^c)\leq \frac{1}{n}\}$. We see that $(M_{\tau_n})_{n=0}^{\infty}$ is a discrete-time $(\mathcal{F}_{\tau_n})_{n\geq 0}$-martingale, hence 
\[
\expE[\liminf_{n\ra\infty}\tau_n+u_0(B_{\tau_n})]=\expE_x[\liminf_{n\ra\infty}M_{\tau_n}]\leq \liminf_{n\ra\infty}\expE_x[M_{\tau_n}]=\expE[\tau_E].
\]
Since $\tau_n\uparrow\tau_E$ almost surely, it follows that $\expE_x[\liminf_{n\ra\infty}u_0(B_{\tau_n})]=0$ (using the monotone convergence theorem), hence $\liminf_{n\ra\infty}u_0(B_{\tau_n})=0$ almost surely (since it's non-negative). Then by applying Doob's upcrossing lemma (over the set of times $\mathbb{Q}_{\geq 0}$), we obtain Lemma \ref{lem:hitting of u0 of Brownian motion}.

This concludes the proof of Theorem \ref{eq:FKPP general bounded domain exis uniq}.
\end{proof}

\subsection*{Some consequences on unbounded domains}

We now study the properties of solutions to the stationary FKPP equation
\begin{equation}\label{eq:stationary FKPP}
\frac{1}{2}\Delta u+c(x)u(1-u)=0,\quad x\in E,\quad u\equiv 0\text{ on $\partial E$.}
\end{equation}
on unbounded domains, again and for the remainder of this section assuming that $\partial E$ is $C^{1,1}$.

Note that our results would remain true with $\frac{1}{2}\Delta$ replaced by more general uniformly elliptic non-divergence form operators and the non-linearity $u(1-u)$ replaced by that corresponding to a more general offspring distribution, but we restrict our attention for the sake of simplicity. The caveats to this are that we require at times Theorem \ref{theo:local survival equivalence} and the results of \cite{Englander2004}, which have been proven under different restrictions on the differential operator.

Whilst we believe the above described connection could prove a fruitful way of understanding the properties of stationary solutions of the FKPP equation, it is not our intention to exhaustively pursue this here, so we only illustrate some straightforward consequences.

We assume throughout that $c$ is locally $\alpha$-Holder continuous ($\alpha\in (0,1)$) and uniformly bounded. We will further assume that $E$ is a non-empty, open, connected subdomain of $\Rm^d$ with $C^{1,1}$ boundary. We will sometimes consider different domains $E$ and $E'$, in which case both should be understood to have $C^{1,1}$ boundary. Where we need to clarify the domain, we shall write $\lambda_1(\frac{1}{2}\Delta+c,E)$ for instance (and use similar notation for the other generalised principal eigenvalues).

We will connect the properties of this equation to those of branching Brownian motion with instantaneous killing along $\partial E$ and branching at rate $c(x)$. We denote this by $\bar B_t$. The events $\locSurv$, $\glSurv$ and $\rlocSurv$ are then defined for this branching diffusion. Where we wish to clarify the domain we are referring to, we denote this by $\locSurv_E$, ${\glSurv}_E$ and $\rlocSurv$ (if the domain is $E$).

\commentout{
We assume Assumption \ref{assum:uniformly elliptic standing assumptions} throughout, which we recall means that:
\begin{enumerate}
    \item $E$ is a non-empty, open, connected subdomain of $\Rm^d$ with $C^{1,1}$ boundary;
    \item $a_{ij}(x)$, $b_i(x)$ and $c(x)$ are all locally H\"older continuous (with exponent $\alpha\in (0,1)$) and globally bounded;
    \item $a(x)$ is symmetric for all $x$ and uniformly elliptic.
\end{enumerate} 
}

\begin{prop}\label{prop:limsup nve implies uniqueness}
Suppose that $\lambda_1(\frac{1}{2}\Delta+c,E)>0$ but $\limsup_{r\ra\infty}\lambda_1'(\frac{1}{2}\Delta+c,E\setminus \bar B(0,r))<0$. Then there exists a unique $(0,1]$-valued solution of \eqref{eq:stationary FKPP}.
\end{prop}
\begin{proof}[Proof of Proposition \ref{prop:limsup nve implies uniqueness}]
We firstly note that we have $\Pm_x(\locSurv)>0$ by \cite{Englander2004}. The existence part then follows from Theorem \ref{theo:PDE semilinear elliptic equation global survival}. In order to obtain uniqueness, it suffices to check whether $\glSurv_E\setminus\locSurv_E$ is possible with positive probability, by Theorem \ref{theo:local survival minimal solution}. 

We suppose for contradiction that this is the case. On the event $\glSurv_{E}\setminus \locSurv_{E}$, the particles survive but undergo $r$-local extinction (i.e. they drift off to infinity) by Theorem \ref{theo:local survival equivalence}. Therefore we have must have $\Pm_x(\glSurv_{E\setminus \bar B(0,R)}\setminus \locSurv_{E\setminus \bar B(0,R)})$ for $x\in E\setminus \bar B(0,R)$, for all $R<\infty$. However $\lambda'(E\setminus \bar B(0,R))<0$ for all $R<\infty$ sufficiently large, meaning that 
\[
\Pm_x(\glSurv_{E\setminus \bar B(0,R)}\setminus \locSurv_{E\setminus \bar B(0,R)})\leq \Pm_x(\glSurv_{E\setminus \bar B(0,R)})=0
\]
for $R<\infty$ sufficiently large and $x\in E\setminus \bar B(0,R)$, by Theorem \ref{theo:PDE global extinction e-value relationship}.
\end{proof}

\begin{prop}\label{prop:can extend uniqueness bounded extensions}
We suppose that we have two domains $E'\supseteq E$, and $E'\setminus E$ is bounded. We further assume that $\lambda_1'(\frac{1}{2}\Delta+c,E)>0$, and that \eqref{eq:stationary FKPP} has at most one solution on $E$. Then \eqref{eq:stationary FKPP} also has at most one solution on $E'$.
\end{prop}
\begin{proof}[Proof of Proposition \ref{prop:can extend uniqueness bounded extensions}]
Since $\lambda_1'(\frac{1}{2}\Delta+c,E)>0$, $\Pm_x(\locSurv_{E'})\geq \Pm_x(\locSurv_E)>0$ by \cite{Englander2004}. We now suppose for the sake of contradiction that \eqref{eq:stationary FKPP} has more than one solution on $E'$. Then $\Pm_x(\glSurv_{E'}\setminus\locSurv_{E'})>0$ for $x\in E$ by Theorem \ref{theo:local survival minimal solution}. On the event $\glSurv_{E'}\setminus \locSurv_{E'}$, the particles survive but undergo $r$-local extinction on $E'$ (i.e. they drift off to infinity) by Theorem \ref{theo:local survival equivalence}. Since $E'\setminus E$ is bounded, this then implies that $\Pm_x(\glSurv_E\setminus \locSurv_E)>0$ by Theorem \ref{theo:local survival minimal solution}. This contradicts the uniqueness of solutons to \eqref{eq:stationary FKPP} on $E$ by Theorem \ref{theo:local survival minimal solution}. Therefore solutions of \eqref{eq:stationary FKPP} on $E'$ are unique.
\end{proof}

\begin{prop}\label{prop:eulidean compact c}
We assume that $E=\Rm^d$ and $c$ is continuous and compactly supported, with $c\nequiv 0$. We further assume that $\lambda_1(\frac{1}{2}\Delta+c)>0$, so that there exists a solution to \eqref{eq:stationary FKPP}. Then this solution is unique for $d=1,2$, but non-unique for $d\geq 3$.
\end{prop}
\begin{proof}[Proof of Proposition \ref{prop:eulidean compact c}]
Since $\lambda_1(\frac{1}{2}\Delta+c)>0$, $\Pm_x(\locSurv)>0$ by \cite{Englander2004}. By Theorem \ref{theo:local survival minimal solution} it suffices to check whether the event $\glSurv\setminus \locSurv$ is possible with positive probability. We note that $\bar B_t$ doesn't have any killing. 

In the case of $d=1,2$, a given particle must return infinitely often to $B(0,1)$ by the recurrence of Brownian motion in dimensions $1$ and $2$, hence local extinction is not possible. In dimension $3$ or higher on the other hand, there is a positive probability that a given particle will leave the support of $c$ without branching, and then never return, by the transcience of Brownian motion in dimension $3$ or higher. Therefore $\glSurv\setminus \locSurv$ is possible with positive probability in dimension $3$ or higher, but not in dimensions $1$ or $2$. Using Theorem \ref{theo:local survival minimal solution} we are done. 
\end{proof}

In \cite{Berestycki1997}, Berestycki, Cafarelli and Nirenberg considered positive and bounded solutions of the semilinear equation 
\begin{equation}\label{eq:semilinear elliptic equation BerCafNir}
\frac{1}{2}\Delta u+f(u)=0\quad \text{in}\quad E,\quad u\equiv 0\quad\text{on}\quad \partial E,
\end{equation}
for a class of non-linearities that include $f(u)=u(1-u)$, on domains of the form
\begin{equation}
    E=\{x\in \Rm^d:x_d>\varphi(x_1,\ldots,x_{d-1})\}
\end{equation}
for some Lipschitz function $\varphi$. They established the existence of a unique classical solution for \eqref{eq:semilinear elliptic equation BerCafNir}, which they show is monotone in the $x_d$ direction (in fact in a cone of directions). We now provide a probabilistic proof of this in the special FKPP case $f(u)=u(1-u)$, under the additional assumption that $\varphi$ is $C^{1,1}$ (so we can apply our results).
\begin{prop}[Theorems 1.1 and 1.2, \cite{Berestycki1997}]\label{prop:monotone BerNirCaf}
There exists a unique $(0,1]$-valued solution of \eqref{eq:semilinear elliptic equation BerCafNir} with $f(u)=u(1-u)$, which is monotone increasing in the $x_d$ direction, and in fact in a cone of directions.
\end{prop}

\begin{proof}[Proof of Proposition \ref{prop:monotone BerNirCaf}]
We take a ball $B(0,r)$ such that $\lambda_1(\frac{1}{2}\Delta+1,B(0,r))>0$, meaning that we have a positive probability of local survival on $B(0,r)$ by \cite{Englander2004}. Since $B(0,r)$ is a subset of $E$ up to translation, $\Pm_x(\locSurv_E)>0$ for $x\in E$, and hence we have established the existence of solutions to \eqref{eq:semilinear elliptic equation BerCafNir} by Theorem \ref{theo:PDE semilinear elliptic equation global survival}.

The corresponding branching process is simply standard branching Brownian motion in $E$. Since $\varphi$ is Lipschitz, there is an aperture $\theta$ such that for every $h<\infty$ and $x\in \partial E$, we can place an $n$-cone with apex on $x$ and with the central line through the cone pointing in the $(0,\ldots,0,1)$ direction, such that the cone lies entirely within $E$ (except for its apex). We write $C$ for a copy of this open cone with apex at $0$, so the aforedescribed cone is $x+C$ for $x\in \partial E$. 

We take $\epsilon>0$ to be sufficiently small so that branching Brownian motion cannot survive within $\partial E+B(0,2\epsilon)$ (which is possible by the Lipschitz assumption). We define $x_0:=(0,\ldots,0,\epsilon)$. We then set $10\epsilon<h<\infty$ to be sufficiently large so that $\lambda_1(\frac{1}{2}\Delta+1,C)>0$, implying that the probability of survival for branching Brownian motion started at $x_0$ and killed upon exiting $C$ is positive by \cite{Englander2004}. We let this probability be $p>0$.

We now check whether global survival without local survival is possible. On the event of global survival there must be an infinite sequence of times at which a particle is outside of $B(E,2\epsilon)$. But if a particle $X^i_t$ is outside of $B(E,2\epsilon)$ then we can place a cone $C+x$ such that the particle is at $X^i_t=x+x_0$. In this case the probability that its descendents go on to survive locally is at least $p$. 

Therefore on the event of global survival we can construct a sequence of times as follows. We firstly consider the first time that a particle is outside of $B(E,2\epsilon)$, which may be time $0$. We pick such a particle. We then place a cone around it as described above, and follow its descendents until they either all leave the cone at some point, or survive. If they survive the process terminates, and we have local survival. If they terminate, we then take the next time one of the particles leaves $B(E,2\epsilon)$, and repeat the process. 

On the event of global survival the above procedure can only terminate with local survival. Since at each time step the probability of this happening is $p>0$, it must terminate eventually. Therefore $\Pm_x(\locSurv\lvert \glSurv)=1$ for all $x\in E$, and hence we have uniqueness of stationary solutions of \eqref{eq:semilinear elliptic equation BerCafNir}.

To establish monotonicity we employ the representation given by Theorem \ref{theo:PDE semilinear elliptic equation global survival}. Precisely, Theorem \ref{theo:PDE semilinear elliptic equation global survival} tells us that the unique positive, bounded solution is precisely
\[
u(x):=\Pm_x(\glSurv_E).
\]
However for $y\in (0,\ldots,0)\times \Rm_{>0}$, we have
\[
u(x+y):=\Pm_{x+y}(\glSurv_E)=\Pm_{x}(\glSurv_{E-y})\geq \Pm_{x}(\glSurv_{E}),
\]
with the latter inequality being satisfied since $E-y\supseteq E$ and increasing the size of the domain can only increase survival probability. It is clear that the above inequality is in fact strict, and the same argument works for a cone of $y$s (since $\varphi$ is Lipschitz).
\end{proof}

\subsubsection*{Structure of the proofs}
The remainder of this section is devoted to the proofs of the above results. We will prove the theorems in the order in which they appear in the proofs, except that:
\begin{enumerate}
    \item Before proving any of our theorems we will prove Proposition \ref{prop:nve lambda implies local survival}, which is a partial version of Part 1 of Theorem 3\cite{Englander2004} under Assumption \ref{assum:uniformly elliptic standing assumptions}. Specifically this gives that $\lambda(L)<0$ implies $\Pm_x(\locSurv)=0$.
    \item The proof of Theorem \ref{theo:local survival gives liminf of number of particles} will be deferred until after that of Theorems \ref{theo:PDE semilinear elliptic equation global survival} and \ref{theo:local survival minimal solution}.
\end{enumerate}

\subsection{Partial version of Part 1 of Theorem 3 \cite{Englander2004} under Assumption \ref{assum:uniformly elliptic standing assumptions}}

We prove the following.
\begin{prop}\label{prop:nve lambda implies local survival}
We assume Assumption \ref{assum:uniformly elliptic standing assumptions}, and that $\lambda(L)<0$. Then $\Pm_x(\locSurv)=0$ for all $x\in E$.
\end{prop}
\begin{proof}
Since $\lambda(L)<0$, we can take $\lambda>0$ and $u\in W^{2,d}_{\loc}(E)$ such that $u(x)>0$ everywhere and $Lu\leq -\lambda u$. We take non-empty, open, pre-compact $U\subset\subset E$. Then
\[
\expE_x[\#\{i:X^i_t\in U\}]\leq \frac{1}{\inf\{u(x'):x'\in U\}}\expE_x[\sum_{i=1}^{N_t}u(X^i_t)]\leq \frac{1}{\inf\{u(x'):x'\in U\}}e^{-\lambda t}u(x).
\]
We conclude by Markov's inequality and the Borel-Cantelli lemma.
\end{proof}

\subsection{Proof of Theorem \ref{theo:limsup mass doesn't depend on x}}

Our goal is to show that
\begin{equation}\label{eq:limsup of mass doesn't depend on x}
    \limsup_{T\ra\infty}\frac{1}{T}\ln\expE_x[N_T]
\end{equation}
doesn't depend upon $x\in E$.

Since we can always add a constant local binary branching rate $r$ or a constant local killing rate $\kappa$, which affects $\expE_x[N_t]$ by a factor of $e^{rt}$ or $e^{-\kappa t}$ (respectively), it suffices to prove the following proposition.
\begin{prop}\label{prop:control on limsup of mass}
For all $x_0,x\in E$ there exists $C<\infty$ such that 
\[
    \limsup_{T\ra\infty}\expE_{x}[N_T]\leq C\limsup_{T\ra\infty}\expE_{x_0}[N_T].
\]
\end{prop}

We fix $2<T<\infty$ for the time being. Then 
\[
u(x,t):=\expE_{x,t}[N_T]
\]
is a classical solution of Kolmogorov's backward equation,
\[
 \partial_tu=a_{ij}(x)\partial_{ij}u+b_i(x)\partial_iu+c(x)u,
 \]
 by Theorem \ref{theo:parabolic Feynman-Kac}. We now take $\bar c:=\sup_{x}\lvert c(x)\rvert$. Therefore 
\[
v(x,t):=e^{-\bar c t}u(x,t)
\]
is a classical solution of 
\[
 \partial_tv=a_{ij}(x)\partial_{ij}v+b_i(x)\partial_iv+[c(x)-\bar c]v.
\]

We fix $1>\delta>0$ and let $U_{\delta}:=\{x\in E:d(x,\partial E)>10\delta\}$. Then by \cite[Theorem 1.1]{Krylov1981} there exists a constant $C$ which doesn't depend upon $\delta$ nor $T$ such that
\[
v(4\delta^2,x_0)\leq Cv(8\delta^2,x)
\]
for $x_0\in U_{\delta}$ and $x\in B(x_0,\delta)$. Therefore
\[
u(4\delta^2,x_0)= e^{ 4\bar c\delta^2}v(4\delta^2,x_0)\leq Ce^{ 4\bar c\delta^2}v(8\delta^2,x)=Ce^{- 4\bar c\delta^2}u(8\delta^2,x).
\]
Therefore
\[
u(4\delta^2,x_0)\leq Cu(8\delta^2,x).
\]
Since the constant $C$ doesn't depend upon $T$, it follows that
\[
\begin{split}
\limsup_{T\ra\infty}\expE_{x_0}[N_T]=\limsup_{T\ra\infty}\expE_{x_0}[N_{T-4\delta^2}]=u(4\delta^2,x_0)\\
\leq Cu(8\delta^2,x)=C\limsup_{T\ra\infty}\expE_{x}[N_{T-8\delta^2}]=C\limsup_{T\ra\infty}\expE_{x}[N_{T-8\delta^2}]
\end{split}
\]
for all $x_0\in U_{\delta}$ and $x\in B(x_0,\delta)$. We therefore conclude the proof of Proposition \ref{prop:control on limsup of mass}, and hence of \eqref{eq:limsup of mass doesn't depend on x}.\qed

\subsection{Proof of Theorem \ref{theo:PDE equality of eigenvalues different notion}}


Our goal is to show that $\lambda_1'(L)>0$ implies $\lambda_c'((P_t)_{t\geq 0})\geq 0$, and $\lambda_c'((P_t)_{t\geq 0})> 0$ implies $\lambda_1'(L)\geq 0$. Since we can always add a constant killing rate or local binary branching rate, this is sufficient.

We suppose that $\lambda'(L)>0$, so we can take $u\in W^{2,d}_{\text{loc}}(E)\cap L^{\infty}(E)$ such that $Lu\geq 0$, $u(x)>0$ for all $x\in E$, and $u(x)\ra 0$ as $x\ra \xi$ for all $\xi\in \partial E$. Then  by Ito's lemma \cite[Krylov, Controlled diffusions, p.122 and 47]{Krylov1980},
\[
\sum_{i=1}^{N_t}u(X^i_t)
\]
is a local submartingale, hence  a submartingale (since the branching rate is bounded). Therefore $\lambda'_c((P_t)_{t\geq 0})\geq 0$. 

We now assume that $\lambda'_c((P_t)_{t\geq 0})>0$. Then we can take $u\in B_b(E;\Rm_{\geq 0})$ such that $(P_tu)(x)\geq u(x)$ for all $t\geq 0$, $x\in E$. We set $v=P_1u\in C^{2,\alpha}_{\loc}(E)\cap B_b(E;\Rm_{\geq 0})\subset  W^{2,d}_{\loc}(E)\cap B_b(E;\Rm_{\geq 0})$ (by Theorem \ref{theo:parabolic Feynman-Kac}). We observe that $(P_tv)(x)\geq v(x)$ for all $x\in E$ and $t\geq 0$, hence $Lv\geq 0$. Moreover, since $v(x)\leq \expE_x[N_t]\lvert\lvert u\rvert\rvert_{\infty}$, it follows that $v(x)\ra 0$ as $x\ra \xi$ for all $\xi\in \partial E$. Therefore $\lambda'_1(L)\geq 0$.

The proof that $\lambda''_1(L)<0$ implies $\lambda_c''((P_t)_{t\geq 0})\leq 0$ is identical to the proof that $\lambda'_1(L)>0$ implies $\lambda_c'((P_t)_{t\geq 0})\geq 0$. We now assume that $\lambda_c''((P_t)_{t\geq 0})< 0$, so that $\limsup_{t\ra\infty}\frac{1}{t}\ln\expE[N_t]<0$ by Theorem~\ref{th:lambda'_lambda''}. We recall the function $v$ defined in the proof of the equality in Theorem~\ref{th:lambda'_lambda''} was defined by
\[
v(x):=\int_0^{\infty} (P_t1)(x)dt.
\]
It follows from Theorem \ref{theo:elliptic Feynman-Kac} that $v\in C^{2,\alpha}_{\loc}(E)\subset W^{2,d}_{\loc}(E)$. In the proof of the equality in Theorem \ref{th:lambda'_lambda''}, we then defined $u(x)=v(x)+\epsilon$ for some appropriately chosen $\epsilon>0$ sufficiently small. We see that $u=v+\epsilon\in W^{2,d}_{\loc}(E)$, hence $\lambda_c''(L)\geq 0$. \qed

\subsection{Proof of Theorem \ref{theo:local survival equivalence}}

The inclusion $\{\locSurv\}\subseteq \{\text{$r$-local survival}\}$ is obvious. Our goal is to show that 
\[
\{\text{$r$-local survival}\}\cap \{\text{local extinction}\}
\]
is a $\Pm_x$-null set. To do so we shall establish that there exists arbitrarily large $r_0<\infty$ such that, for some open, non-empty, pre-compact subset $U\subset \subset E$ we have
\begin{equation}\label{eq:sufficient condition for local survival theorem}
\glSurv \cap \{\liminf_{t\ra \infty}\inf_{1\leq i\leq N_t}\lvert X^i_t\rvert\leq r_0\} \cap \{\text{$U$-local extinction}\}\quad \text{is a $\Pm_x$ null set for all $x\in E$.}
\end{equation}

We fix aribtrarily large $r_2<\infty$. We take $\delta>0$ to be determined, and set
\[
r_3=r_2+\delta,\quad r_1=r_2-\delta,\quad r_0=r_1-\delta.
\]

We now define the notion of ``return time'', as follows. Given $0\leq s<t<\infty$, $1\leq j\leq N_s$ and $1\leq i\leq N_t$ we write $X^j_s\preccurlyeq X^i_t$ if $X^j_s$ is the time $s$ ancestor of $X^i_t$. Given a particle $X^i_t$ we write $j_s(i,t)$ for the (unique) index of the time $s$ ancestor of $X^i_t$, so that $X^{j(i,t)}_s\preccurlyeq X^i_t$. The last $r_3$-time of $X^i_t$, denoted by $t_-(i,t)$, is defined by
\[
t_-(i,t):=\sup\{s<t:\lvert X^{j_s(i,t)}_s\rvert \geq r_3\},
\]
the above being defined to be $0$ if the above set of times is empty.

We say that $t$ is a $r_2$-return time of $X^i$ if $\lvert X^i_t\rvert =r_2$ and $\lvert X^{j(i,t)}_t\lvert >r_2$ for all $t_-(i,t)\leq s<t$. Essentially, these represent excursions which start at $\partial B(0,r_2)$, touch $\partial B(0,r_3)$, then return to $\partial B(0,r_2)$. If $t$ is a return time of $X^i$ then we say that $(i,t)$ is a return pair. These return pairs then define a branching structure. Note that we define $(1,0)$ ($1$ being the index of the single particle at time $0$) to be a return pair, this is the unique return pair such that the corresponding particle may not be in $\partial B(0,r_2)$. 

Given that $(i_1,t_1)$ and $(i_2,t_2)$ are both return pairs and $t_1<t_2$, we say that $(i_2,t_2)$ is a next return pair of $(i_1,t_1)$ if $X^{i_1}_{t_1}\preccurlyeq X^{i_2}_{t_2}$ and there are no return pairs along the ancestral path between $i_1$ and $i_2$, i.e. $(j,s)$ is not a return pair for all $s\in (t_1,t_2)$ and $j=j_s(i_2,t_2)$. A given return pair may have many next return pairs, since it may branch before touching $r_3$ then touching $r_2$. This therefore defines a branching structure, which we label with Ulam-Harris notation
\[
\mathbb{V}:=\cup_{n\geq 0}\mathbb{N}^n.
\]

We now inductively define a family of sigma-algebras $(\calG_v:v\in \mathbb{V})$ inductively as follows. The usual filtration with respect to which the branching process is adapted is $(\mathcal{F}_t)_{t\geq 0}$. We take $\calG_{\emptyset}:=\calF_0$. We now inductively assume that we have defined $\calG_v$ for some $v\in \mathbb{V}$, that there is a corresponding $\calG_v$-measurable return time $(i_v,t_v)$, and that $X^{i_v}_{t_v}$ is $\calG_v$-measurable. We now consider the branching process $(\bar Y^v_s:s\geq t_v)$ given by following the descendents of $X^{i_v}_{t_v}$ and stopping each descendent once it undergoes a return time. We let $\tau^v_1,\tau^v_2,\ldots$ be the return times of $(\bar Y^v_s:s\geq t_v)$ in chronological order, and define $\calG_{vk}:=\sigma(\bar Y^v_s:t_v\leq s\leq \tau^v_k)$ for each $k<\infty$. There may be finitely or infinitely many subsequent return times $\tau^v_k$. If there are only $k_{\max}<\infty$ of them, we set $\calG_{vk}:=\calG_{vk_{\max}}$ for $k>k_{\max}$ and set $(i_{vk},t_{vk}):=\emptyset$. Whenever $(i_v,t_v):=\emptyset$ we define $(i_{vk},t_{vk}):=\emptyset$ and $\calG_{vk}:=\calG_v$ for all $k\in \mathbb{N}$.

We now enumerate the elements of $\mathbb{V}$ as $v_0,v_1,\ldots$ in such a manner so that $v_{0}=\emptyset$ and the parent of $v_n$ belongs to $v_0,\ldots,v_{n-1}$ for all $n\geq 1$. We write $(i_n,t_n):=(i_{tv_n},t_{v_n})$ for the return pair corresponding to $v_n$ and define the sigma-algebra $\calH_n:=\bigvee_{k\leq n}\calG_{v_k}$, for each $n$. We note that $(\calH_n)_{n\geq 0}$ is a filtration. 


We write $A_n$ and $B_n$ for the event that a child of $X^{i_n}_{t_n}$ touches $\bar B(x_0,r)\setminus U$ (respectively $\bar U$) before a next return time along the corresponding ancestral path. Our goal is to show that for any $\delta>0$ sufficiently small, $U\subset\subset B(0,r_0)$ may be chosen so that 
\[
M_0:=1,\quad M_{n}:=\exp\Big[\sum_{k\leq n-1}(\Ind(A_k)-\Ind(B_k))\Big],\quad n\geq 1,
\]
is a non-negative $(\calH_n)_{n\geq 0}$-supermartingale after time $1$, meaning that
\[
\expE[M_{n+1}\lvert \calH_n]\leq M_n
\]
for $n\geq 1$. It will be convenient to ignore the $n=0$ case because we have no guarantee that $\lvert x\rvert=r_2$, which will be important later on. The non-negativity is trivial. 

Then
\[
\expE\Big[\frac{M_{n+1}}{M_n}\Big\lvert \calH_n\Big]=\expE[\exp[\Ind(A_n)-\Ind(B_n)]\lvert \calH_n],\quad n\geq 1.
\]
It therefore suffices to prove that
\begin{equation}\label{eq:equation to prove supermartingale local survival theorem}
\expE[\exp[\Ind(A_n)-\Ind(B_n)]\lvert \calH_n]\leq 1,\quad n\geq 1.
\end{equation}
We assume henceforth that $n\geq 1$. In this case
\[
\begin{split}
\expE[\exp[\Ind(A_n)-\Ind(B_n)]\lvert \calH_n]-1= \expE[\exp[\Ind(A_n)-\Ind(B_n)]\lvert X^{i_n}_{t_n}]-1\\
=(e-1)\Pm(A_n\setminus B_n\lvert X^{i_n}_{t_n})+(e^{-1}-1)\Pm(B_n\setminus A_n\lvert X^{i_n}_{t_n}).
\end{split}
\]

Our goal now is to show that we can choose $\delta >0$ and $U\subset\subset B(x_0,r_1)$ such that
\[
    \Pm(B_n\setminus A_n\lvert X^{i_n}_{t_n})>\frac{e-1}{(1-e^{-1})}\Pm(A_n\setminus B_n\lvert X^{i_n}_{t_n}),
\]
for all $n\geq 1$. This would imply that $(M_n)_{n\geq 0}$ is a non-negative $(\calH_n)_{n\geq 0}$ supermartingale after time $1$.

Given a branching process started at $\bar X_0=x\in \partial B(0,r_2)$, we write $A$ and $B$ for the event that a descendent of $X^1_0$ touches $\bar B(x_0,r)\setminus U$ (respectively $\bar U$) before a next return time along the corresponding ancestral path. It suffices to show that
\begin{equation}\label{eq:sufficient condition 1 local survival thm}
    \Pm_x(B\setminus A)>\frac{e-1}{1-e^{-1}}\Pm_x(A\setminus B),
\end{equation}
for all $x\in \partial B(0,r_2)\cap E$.

Note that once a particle touches $\partial B(0,r_3)$, neither it nor its descendents can touch $\bar U\cup \bar B(0,r_0)$ prior to a next return time because this would entail firstly returning to $\partial B(0,r_2)$. We may therefore equivalently kill all particles upon exiting $B(0,r_3)$. We now define
\[
\tilde E:=E\cap B(0,r_3)\setminus (\bar U \cup \bar B(0,r_0)).
\]
We therefore equivalently consider particles to be instantaneously killed upon exiting $\tilde{E}$, with $A$ and $B$ the events that at least one particle exits $\tilde{E}$ through $\partial B(0,r_0)$ or $\partial U$ respectively.

In order to prove \eqref{eq:sufficient condition 1 local survival thm}, it suffices to show that
\begin{equation}\label{eq:sufficient condition 2 local survival thm}
\Pm_x(B)\geq [1+\frac{e-1}{1-e^{-1}}]\Pm_x(A)
\end{equation}
for all $x\in \partial B(0,r_2)\cap E$.

We define
\[
\begin{split}
N_A:=\#\{\text{particles that exit through $\partial B(0,r_0)$ prior to time $1$}\}\\+\#\{\text{particles at time $1$ with a descendent that exits through $\partial B(0,r_0)$ after time $1$}\}.
\end{split}
\]
Since $A$ is precisely the event that $N_A\geq 1$, Markov's inequality tells us that
\[
\Pm_x(A)\leq \expE_x[N_A].
\]

We now write $(X^0_t)_{t<\tilde \tau}$ for a copy of the right-process without branching, killed upon exiting $\tilde{ E}$. Then
\[
\Pm_x(B)\geq e^{-\lvert\lvert c\rvert\rvert_{\infty}}\left[\Pm_x(\text{$X^0_t$ exits $\tilde {E}$ through $\partial U$ prior to time $1$})+\expE[\Pm_{X^0_{1}}(B)\Ind(\tilde{\tau}>1)]\right].
\]

Then (by Theorem \ref{theo:parabolic Feynman-Kac}), $\expE_x[N_A]=u(x,1)$ where $u(x,t)$ is the unique classical solution of
\[
\begin{split}
\partial_tu=Lu,\quad (x,t)\in E\times (0,1),\quad u(x,0)=\Pm_x(A),\quad x\in E,\\ u(x,t)=\Ind(x\in \partial B(0,r_0)),\quad (x,t)\in \partial E\times (0,1).
\end{split}
\]
Similarly $v(x,1)=\Pm_x(X^0\text{ exits $\tilde {E}$ through $\partial U$})+\expE[\Pm_{X^0_{1}}(B)\Ind(\tilde{\tau}>1)]$ where $v(x,t)$ is the unique classical solution of
\[
\partial_tv=Lv,\quad (x,t)\in E\times (0,1),\quad v(x,0)=\Pm_x(B),\quad x\in E,\\ u(x,t)=\Ind(x\in \partial U),\quad (x,t)\in \partial E\times (0,1).
\]
Since both $u$ and $v$ vanish continuously along $\partial E\cap (B(0,r_3)\setminus B(0,r_1))$, the parabolic boundary Harnack inequality \cite{Torres-Latorre2024} (note that this is where we use $b\equiv 0$) implies the existence of $C_{\delta}<\infty$ dependent only upon $\delta>0$ (using that for any $\delta>0$ and any compact subset $K$ of $[E\cap (B(0,r_3)\setminus \bar B(0,r_1))]\times (0,1)$ there exists $U_{\min}\subset\subset B(0,r_1)$ such that $v$ is bounded away from $0$ on $K$ whenever $U_0\subset U\subset\subset B(0,r_1)$) such that
\[
\sup_{x\in \partial B(0,r_2)}\frac{u(x,1)}{v(x,1)}\leq C_{\delta}\sup_{(x,t)\in [B(0,r_3)\setminus B(0,r_1)]\times (0,1)}u(x,t).
\]
Now for any fixed $\delta>0$, we can ensure that the right-hand side is as small as we like by taking sufficiently large $U\subset\subset B(0,r_1)$, implying that for all $\delta>0$ there exists $U\subset \subset  B(0,r_1)$ such that
\[
\begin{split}
\frac{\Pm_x(A)}{\Pm_x(B)}\leq e^{\lvert\lvert c\rvert\rvert_{\infty}}\frac{\expE_x[N_A]}{\Pm_x(X^0\text{ exits $\tilde {E}$ through $\partial U$})+\expE[\Pm_{X^0_{1}}(B)\Ind(\tilde{\tau}>1)]}\\
\leq e^{\lvert\lvert c\rvert\rvert_{\infty}}\frac{u(x,t)}{v(x,1)}\leq \Big[1+\frac{e-1}{1-e^{-1}}\Big]^{-1}.
\end{split}
\]
for all $x\in \partial B(0,r_2)\cap E$. We therefore have \eqref{eq:sufficient condition 2 local survival thm} and hence \eqref{eq:sufficient condition 1 local survival thm}.

This implies that for all $\delta >0$ we can choose $U\subset\subset B(0,r_1)$ such that $(M_n)_{n\geq 0}$ is a non-negative $(\calH_n)_{n\geq 0}$-supermartingale after time $1$, hence $\lim_{n\ra\infty}M_n$ exists and is finite almost-surely. 

This then implies that if $A_n$ occurs infinitely often, so too must $B_n$. 

We now suppose that the event $\glSurv\cap\{\liminf_{t\ra\infty}\inf_{1\leq i\leq N_t}\lvert X^i_t\rvert< r_1\}$ occurs. Then one of two things must happen:
\begin{enumerate}
    \item the descendents of one child, considered to be erased upon exiting $\tilde E$, globally survives in $\tilde{E}$;
    \item $A_n$ occurs infinitely often.
\end{enumerate}
By taking $\delta >0$ sufficiently small and $U$ sufficiently large we can ensure that the second possibility cannot occur. Making $U$ larger if necessary to ensure that $(M_n)_{n\geq 0}$ is a non-negative $(\calH_n)_{n\geq 0}$-supermartingale after time $1$, we see that the first possibility possibility must occur on the event  $\glSurv\cap\{\liminf_{t\ra\infty}\inf_{1\leq i\leq N_t}\lvert X^i_t\rvert< r_1\}$. Since $\lim_{n\ra\infty}M_n$ exists and is finite almost-surely, this entails $B_n$ occuring infinitely often, which then implies local survival. This concludes the proof of Theorem \ref{theo:local survival equivalence}.\qed

\subsection{Proof of Theorems \ref{theo:lambda'=lambda'' condition 1} and \ref{theo:lambda'=lambda'' condition}}\label{subsec:lambda'=lambda'' condition}

We will provide a proof of Theorem \ref{theo:lambda'=lambda'' condition}. The proof of Theorem \ref{theo:lambda'=lambda'' condition 1} is identical except for minor changes which we point out along the way. 

We suppose that Assumptions \ref{assum:uniformly elliptic standing assumptions} and \ref{assum:uniformly elliptic symmetric} are satisfied. When we point out the changes to be made in order to prove Theorem \ref{theo:lambda'=lambda'' condition 1}, we will replace Assumption \ref{assum:uniformly elliptic symmetric} with the assumption that $a$ is $C^1$ and $L$ can be put into the divergence form $Lu=\frac{1}{2}\partial_i(a_{ij}(x)\partial_ju)+c(x)u$, which apart from the places we point out will make no difference. 


Our goal is to show that if $\lambda''(L)>0$, then we have local (and hence global) survival with poisitive probability.


Since $\lambda''_1>0$, Theorem \ref{theo:elliptic regularity limsup of mass} implies that
\[
\limsup\frac{1}{t}\ln\expE_{x_0}[N_t]>0.
\]

Given $R<\infty$ we define
\[
N^R_t:=\#\{i:\lvert X^i_t\rvert \leq R\},\quad D^R_t:=\#\{i:\lvert X^i_t\rvert > R\}.
\]

We prove the following lemma.
\begin{lem}\label{lem:number of particles greater than Ct}
There exists $C<\infty$ such that
\begin{equation}
    \limsup\frac{1}{t}\ln\expE_{x_0}[D^{Ct}_t]\leq 0.
\end{equation}
\end{lem}
\begin{proof}
We have the following Aronson estimates from \cite[Theorem 6.1 and the beginning of Section 6]{Daners2000}: there exists $p_t(x,y)$ such that $P_t(x,dy)=p_t(x,y)dy$ and 
\[
p_t(x,y)\leq C_0t^{-\frac{d}{2}}\exp(C_0t-c_0\frac{\lvert x-y\rvert^2}{t}),
\]
for some $C_0<\infty$ and $c_0>0$. 

Lemma \ref{lem:number of particles greater than Ct} then follows immediately by taking $C^2>\frac{C_0}{c_0}$ and integrating over $E\cap (B(x_0,Ct)^c)$.
\end{proof}

It therefore follows from Lemma \ref{lem:number of particles greater than Ct} that
\begin{equation}
\limsup\frac{1}{t}\ln\expE_{x_0}[N^{Ct}_t]>0.
\end{equation}

We abuse notation by writing $\mu(y)$ for the density of $\mu$ with respect to Lebesgue measure. We now define
\[
q_t(x,y)=\frac{p_t(x,y)}{\mu(y)}.
\]
For the proof of Theorem \ref{theo:lambda'=lambda'' condition 1} we simply take $\mu=\Leb$ so that $q_t(x,y)=p_t(x,y)$. We see that:
\begin{enumerate}
    \item $P_t(x,dy)=q_t(x,y)\mu(dy)$ for all $x\in E$;
    \item $\int_{E}q_t(x,y)q_t(y,z)\mu(dy)=q_t(x,z)$ for all $x,z\in E$.
\end{enumerate}
Since $P_t$ is self-adjoint with respect to $L^2(\mu)$ (see for instance \cite[Section 4]{Baudoin2014}), it easily follows that
\[
q_t(x,y)=q_t(y,x)\quad\text{for all}\quad x,y\in E.
\]
For the proof of Theorem \ref{theo:lambda'=lambda'' condition 1}, the fact that $L$ can be written in the divergence form similarly implies that $p_t(x,y)=p_t(y,x)$.


We now observe the following, which represents the crucial step in the proof,
\[
\begin{split}
q_{2t}(x_0,x_0)\geq \int_{B(x_0,Ct)}q_t(x_0,y)q_{t}(y,x_0)\mu(dy)\\=\int_{B(x_0,Ct)}[q_t(x_0,y)]^2\mu(dy)=\lvert\lvert q_t(x_0,\cdot)\rvert\rvert_{L^2((B(x_0,Ct),\mu))}^2\\\geq \mu(B(x_0,Ct))^{-\frac{1}{2}}\lvert\lvert q_t(x_0,\cdot)\rvert\rvert_{L^1((B(x_0,Ct),\mu))}^2=\mu(B(x_0,Ct))^{-\frac{1}{2}}\expE_{x_0}[N^{Ct}_t]^2.
\end{split}
\]
Since $\limsup_{t\ra \infty} \frac{1}{t}\ln\mu(B(x_0,Ct))=0$, we see that
\begin{equation}\label{eq:p2t going to infinity}
\limsup_{t\ra\infty}\frac{1}{t}\ln q_{2t}(x_0,x_0)>0.
\end{equation}

\commentout{By abuse of notation, we write $\mu(x)$ for the density of $\mu$ with respect to Lebesgue measure. Then
\[
q_t(x,y)=\frac{p_t(x,y)}{\mu(y)}
\]
gives the transition density of $P_t(x,dy)$ with respect to Lebesgue measure, i.e.
\[
P_t(x,dy)=q_t(x,y)\Leb(dy).
\]
}

It follows from Harnack's inequality, \eqref{eq:p2t going to infinity} and the fact $\mu$ is uniformly positive on compacts that there exists $\delta>0$ such that $B(x_0,\delta)\subset \subset E$ and
\[
\limsup_{t\ra\infty}\inf_{x,y\in B(x_0,\delta)}q_{2t}(x,y)=\infty.
\]
Therefore for all $C<\infty$ we can choose $T_C<\infty$ such that $q_{2T_C}(x,y)\geq C$ for all $x,y\in B(x_0,\delta)$.

We write $\nu=\frac{\Leb_{\lvert_{B(x_0,\delta)}}}{\Leb(B(x_0,\delta))}$ for the uniform distribution on $B(x_0,\delta)$. It follows from Harnack's inequality that if we start with a single particle $X_0=x\in B(x_0,\delta)$, then 
\[
\Pm_{x}(\text{there is at least one particle in $dy$ at time $1$})\geq c\nu(dy)
\]
for some $c>0$ uniform over all $x\in B(x_0,\delta)$. Therefore we can lower bound the number of particles at times $k(2T_C+1)$ for $k\in \mathbb{N}$ by a Galton-Watson process with offspring distribution having expectation at least
\[
c\int_{B(x_0,\delta)\times B(x_0,\delta)}p_{2T_C}(x,y)\Leb(dy)\nu(dx)\geq c\Leb(B(x_0,\delta))C.
\]
By choosing $C$ such that the right-hand side is strictly greater than $1$, we see that the resulting Galton-Watson process is supercritical, implying that we have local survival (and hence global survival) started from $x_0$ with positive probability. This implies that $\lambda(L)\geq 0$ (by Proposition \ref{prop:nve lambda implies local survival}) and $\lambda'_1(L)\geq 0$ by Theorem \ref{theo:PDE global extinction e-value relationship}.

We have established that $\lambda''(L)>0$ implies $\lambda(L),\lambda'(L)\geq 0$. Since we can always add a constant binary branching or soft killing rate, it follows that $\lambda''(L)\leq \lambda(L),\lambda'(L)$. The converse implication $\lambda''(L)\geq \lambda(L)$ is trivial from their definitions, and the implication $\lambda''(L)\geq \lambda'(L)$ follows from Theorem \ref{theo:elliptic regularity limsup of mass}. We conclude that $\lambda(L)=\lambda'(L)=\lambda''(L)$.
\qed

\commentout{

Therefore
\[
\limsup_{t\ra\infty}\inf_{x,y\in B(x_0,\delta)}p^{Ct}_{2t+2}(x,y)\ra\infty
\]
as $t\ra\infty$. Therefore
\[
\limsup_{t\ra\infty}\inf_{x\in B(x_0,\delta)}P^{Ct}_{2t+2}(x,B(x_0,\delta))\geq \Leb(B(x_0,\delta))\limsup_{t\ra\infty}\inf_{x,y\in B(x_0,\delta)}p^{Ct}_{2t+2}(x,y)\ra \infty.
\]

Therefore we can choose $T<\infty$ and $R<\infty$ such that
\[
p_{T}^R(x,B(x_0,\delta))>10
\]
for all $x\in B(x_0,\delta)$.

This implies that 
\[
\lambda''\leq \limsup_{R\ra\infty}\lambda((P^R_t)_{t\geq 0})\leq \lambda.
\]}

\subsection{Proof of Theorem \ref{theo:PDE semilinear elliptic equation global survival}}

We firstly assume that we have a $(0,1]$-valued solution $u$ of \eqref{eq:semilinear elliptic equation}, and seek to establish global survival. Then
\[
y(x,t):=u(x),\quad x\in E,\quad t\in \Rm_{\geq 0}
\]
solves the parabolic equation
\[
\partial_ty=Ly+f(y,x)
\]
with Dirichlet boundary conditions $y(x,t)=0$ for $x\in \partial E$, $t>0$, and initial condition
\[
y(x,0)=u(x),\quad x\in E.
\]
Therefore
\[
y(x,t)=1-\expE_x\Big[\prod_{i=1}^{N_t}(1-v(X^i_t))\Big],\quad t>0.
\]
Since $y(x,t)$ is constant in $t$, this implies that
\[
\expE_x\Big[\prod_{i=1}^{N_t}(1-u(X^i_t))\Big]=1-u(x)
\]
for any $t>0$. Using the branching and Markov properties, this implies that
\[
\prod_{i=1}^{N_t}(1-u(X^i_t))
\]
is a martingale. Since it is also bounded, Doob's $L^1$-martingale convergence theorem implies that
\[
\lim_{t\ra\infty}\prod_{i=1}^{N_t}(1-u(X^i_t))
\]
exists almost surely, and is equal in expectation to $1-u(x)\in [0,1)$. On the event of global extinction this almost sure limit is equal to $1$, hence we cannot have $\Pm_x(\text{global extinction})=1$. Therefore we have $\Pm_x(\glSurv)>0$. 

We now prove the converse direction, assuming that we have $\Pm_x(\glSurv)>0$ for some (hence all) $x\in E$. We consider the solution of the parabolic equation
\[
\partial_ty=Ly+f(y,x)
\]
with Dirichlet boundary conditions $y(x,t)=0$ for $x\in \partial E$, $t>0$, and initial condition
\[
v(x,0)=\Pm_x(\glSurv),\quad x\in E.
\]
Then
\[
v(x,t)=1-\expE_x\Big[\prod_{i=1}^{N_t}(1-\Pm_{X^i_t}(\glSurv))\Big]=1-\expE_x\Big[\prod_{i=1}^{N_t}\Pm_{X^i_t}(\glSurv^c)\Big].
\]
Since global extinction started from $x$ is equivalent to global extinction for each of the children at time $t$,
\[
v(x,t)=1-\Pm_x(\text{global extinction})=\Pm_x(\glSurv).
\]
Therefore $v(x,t)=v(x,0)=\Pm_x(\glSurv)$ for all $t>0$. This implies that
\begin{equation}\label{eq:u(x)=P global survival semilinear equation}
u(x):=\Pm_x(\glSurv)
\end{equation}
satisfies \eqref{eq:semilinear elliptic equation}.

We have now established the first assertion in Theorem \ref{theo:PDE semilinear elliptic equation global survival}. 

We now prove that if there exists a $(0,1]$-valued solution of \eqref{eq:semilinear elliptic equation} (equivalently if we have global survival), then \eqref{eq:u(x)=P global survival semilinear equation} gives the maximal such solution of \eqref{eq:semilinear elliptic equation} (we already know it gives a solution). We suppose that $\tilde{u}$ is another $(0,1]$-valued solution of \eqref{eq:semilinear elliptic equation}. 

Then from the above proof we know that
\[
\tilde{M}_t:=\prod_{i=1}^{N_t}(1-\tilde{u}(X^i_t))
\]
is a martingale, which converges almost surely and in $L^1$ as $t\ra\infty$ to a limit $\tilde{M}_{\infty}$. Then we know that (1) $\tilde{M}_{\infty}\in [0,1]$ almost surely since $\tilde{u}\in (0,1]$ everywhere (and hence $\tilde{M}_t\in [0,1]$ almost surely), and (2) $\tilde{M}_{\infty}=1$ on the event of global extinction. Therefore we have that
\[
\tilde{M}_{\infty}=\Ind(\text{global extinction})+\tilde{M}_{\infty}\Ind(\glSurv)\geq \Ind(\text{global extinction}).
\]
Therefore
\[
1-\tilde{u}(x)=\expE_x[\tilde{M}_{\infty}]\geq \Pm_x(\text{global extinction})=1-u(x).
\]
Thus $\tilde{u}(x)\leq u(x)$.

\subsection{Proof of Theorem \ref{theo:local survival minimal solution}}

The proof that $w(x):=\Pm_x(\locSurv)$ furnishes a $(0,1]$-valued solution of \eqref{eq:semilinear elliptic equation} is identical to the proof that $v(x):=\Pm_x(\glSurv)$ furnishes a $(0,1]$-valued solution of \eqref{eq:semilinear elliptic equation}, in the proof of Theorem \ref{theo:PDE semilinear elliptic equation global survival}. We now prove the minimality of $w$. 

We suppose that $\tilde{u}(x)$ is another solution of \eqref{eq:semilinear elliptic equation}. Then
\[
\tilde{M}_t:=\prod_{i=1}^{N_t}(1-\tilde{u}(X^i_t))
\]
is a $[0,1]$-valued $\Pm_x$-martingale. This was proven in the proof of Theorem \ref{theo:PDE semilinear elliptic equation global survival}.

We take a ball $B(x_0,\delta)\subset\subset E$, and define $\underline{u}:=\inf_{x\in B(x,\delta)}\tilde{u}(x)>0$. We now take arbitrary $\epsilon>0$, and thereby fix $n=n(\epsilon)\in \mathbb{N}$ such that $(1-\underline{u})^n<\epsilon$. We define the stopping time 
\begin{equation}
    \tau_n:=\inf\{t>0:\lvert \{i:X^i_t\in B(x_0,\delta)\}\rvert\geq n\}
\end{equation}
(which we note may be $+\infty$ with positive probability), and thereby define the $[0,1]$-valued stopped martingale
\[
\tilde{M}^n_t:=\tilde{M}_{t\wedge \tau_n}.
\]
By Doob's $L^1$ martingale convergence theorem this has a limit almost surely and in $L^1$ as $t\ra\infty$, which we denote by $\tilde{M}^n_{\infty}$.

\begin{lem}\label{lem:M infty <= epsilon on local survival}
$\tilde{M}^n_{\infty}\leq \epsilon$ on the event of local survival. 
\end{lem}
\begin{proof}[Proof of Lemma \ref{lem:M infty <= epsilon on local survival}]
Note that $\tilde{M}^n_{\infty}=\tilde{M}_{\tau_n}\leq (1-\underline{u})^n\leq \epsilon$ on the event that $\tau_n<\infty$. It therefore suffices to show that $\tau_n<\infty$ almost surely on the event of local survival, i.e. that $\Pm_x(\text{local survival and $\tau_n=\infty$})=0$. 

We define the sequence of stopping times $\hat{\tau}_0:=0$, 
\[
\hat{\tau}_{k+1}:=\inf\{t>\hat{\tau}_k+1:X^i_t\in B(x_0,\delta)\text{ for some }1\leq i\leq N_t\}. 
\]
We then define the sequence of events 
\[
A_k:=\{\hat{\tau}_k<\infty,\tau_{n}>\hat{\tau}_k+1\}. 
\]
Given a single particle in $\text{cl}(B(x_0,\delta))$, there is a probability $q>0$ such that this particle will have at least $n$ children in $B(x_0,\delta)$ after time $\frac{1}{2}$. From this we see that $\Pm_x(A_k\lvert A_{k-1})\leq (1-q)$, and hence $\Pm_x(A_k)\leq (1-q)^k$. It then follows from the Borel-Cantelli lemma that $\Pm_x(\text{$A_k$ eventually})=0$. Therefore $\Pm_x(\cup_kA_k^c)=1$.

Therefore $\Pm_x(\text{$\hat{\tau}_k=\infty$ or $\tau_n\leq \hat{\tau}_k+1<\infty$ for some $k$})=1$. However local survival is equivalent to $\hat{\tau}_k<\infty$ for all $k<\infty$, hence up to $\Pm_x$-null sets
\[
\{\tau_n\leq \hat{\tau}_k+1<\infty\text{ for some $k<\infty$}\}\supseteq \{\hat{\tau}_k<\infty\}=\locSurv.
\]
Therefore $\Pm_x(\{\tau_n=\infty\}\cap \locSurv)=0$. This implies that $\tilde{M}^n_{\infty}\leq \epsilon$ on the event of local survival.
\end{proof}

We have therefore established that 
\[
\tilde{M}^n_{\infty}\leq \epsilon\Ind(\locSurv)+\Ind(\text{local extinction}). 
\]
Since $\expE_x[\tilde{M}^n_{\infty}]=\expE_x[\tilde{M}^n_0]=1-u(x)$, we see that
\[
1-\tilde{u}(x)\leq \epsilon\Pm_x(\locSurv)+\Pm_x(\text{local extinction}).
\]
Since $\epsilon>0$ was arbitrary, it follows that
\[
1-\tilde{u}(x)\leq \Pm_x(\text{local extinction}),
\]
i.e. $\tilde{u}(x)\geq \Pm_x(\locSurv)$. \qed

\subsection{Proof of Theorem \ref{theo:local survival gives liminf of number of particles}}

We recall that we assume Assumption \ref{assum:uniformly elliptic standing assumptions}, that $\lambda_1'(L)>0$, and $\Pm_x(\locSurv)>0$ for some given $x\in E$. Our goal is to prove that
\[
    \Pm_x(\liminf_{t\ra\infty}\frac{1}{t}\ln N_t\geq \lambda_1'(L) \mid \locSurv)=1.
\]

We fix $0<\lambda<\lambda_1'(L)$ for the time being.

Given $1\leq i\leq N_t$ and $t_0>0$, we let $N^{t_0,i}_t$ for $t\geq t_0$ be the number of children of $X^i_{t_0}$ at time $t$. We observe that for any $t_0>0$, 
\[
\liminf_{t\ra\infty}\frac{1}{t}\ln N_t\geq \lambda\text{ if and only if }\liminf_{t\ra\infty}\frac{1}{t-t_0}\ln N^{t_0,i}_t\geq \lambda_1'(L). 
\]
Therefore by precisely the same argument as in the proof of Theorem \ref{theo:PDE semilinear elliptic equation global survival},
\[
\tilde{u}(x):=\Pm_x(\liminf_{t\ra\infty}\frac{1}{t}\ln N_t\geq \lambda )
\]
solves the stationary FKPP equation \eqref{eq:semilinear elliptic equation}. Moreover $\tilde{u}(x)>0$ everywhere by Theorem \ref{theo:PDE global extinction e-value relationship}, so it is a non-trivial $(0,1]$-valued solution.

As in the proof of Theorem \ref{theo:local survival minimal solution} we define the martingale
\[
\tilde{M}_t:=\prod_{i=1}^{N_t}(1-\tilde{u}(X^i_t)).
\]
As before, we take a ball $B(x_0,\delta)\subset\subset E$, and define $\underline{u}:=\inf_{x\in B(x,\delta)}\tilde{u}(x)>0$. We now take arbitrary $\epsilon>0$, and thereby fix $n=n(\epsilon)\in \mathbb{N}$ such that $(1-\underline{u})^n<\epsilon$. We define the stopping time 
\[
    \tau_n:=\inf\{t>0:\lvert \{i:X^i_t\in B(x_0,\delta)\}\rvert\geq n\}
\]
exactly as before. We again thereby define the $[0,1]$-valued stopped martingale
\[
\tilde{M}^n_t:=\tilde{M}_{t\wedge \tau_n},
\]
which has a limit almost surely and in $L^1$ as $t\ra\infty$, which we denote by $\tilde{M}^n_{\infty}$.

We observe that
\[
\begin{split}
\Pm(\liminf_{t\ra\infty}\frac{1}{t}\ln N_t\geq \lambda\lvert \mathcal{F}_{\tau_n})\Ind(\tau_n<\infty)=\Ind(\tau_n<\infty)\Big[1-\prod_{i=1}^{N_{\tau_n}}(1-\tilde{u}(X^i_{\tau_n})\Big]\\
=[1-\tilde{M}_{\tau_n}]\Ind(\tau_n<\infty)=[1-\tilde{M}^n_{\infty}]\Ind(\tau_n<\infty).
\end{split}
\]
We have that $\tau_n<\infty$ and $\tilde{M}^n_{\infty}\leq \epsilon$ almost surely on the event of local survival, exactly as in the proof of Theorem \ref{theo:local survival minimal solution}. Therefore
\[
\Pm(\liminf_{t\ra\infty}\frac{1}{t}\ln N_t\geq \lambda\lvert \mathcal{F}_{\tau_n})\Ind(\locSurv)\geq [1-\epsilon]\Ind(\locSurv).
\]
Therefore
\[
\Pm_x(\liminf_{t\ra\infty}\frac{1}{t}\ln N_t\geq \lambda\lvert \locSurv)\geq 1-\epsilon.
\]
Since $\epsilon>0$ was arbitrary, we have that
\[
\Pm_x(\liminf_{t\ra\infty}\frac{1}{t}\ln N_t\geq \lambda\lvert \locSurv)=1.
\]
Since $\lambda<\lambda_1'(L)$ was arbitrary we are done. \qed

\section{Examples and counter-examples}\label{sec:examples}
Here we collect various examples and counter-examples. 

\subsection{Discrete time and space example, $\rho_{c,x}' < \liminf_{n\to\infty} (\expE_{x}[N_n])^{1/n} < \rho_{c,x}''$}
    \label{sec:example_mutations}
    
    \begin{figure}[ht]
        \centering
        \includegraphics[width=\textwidth]{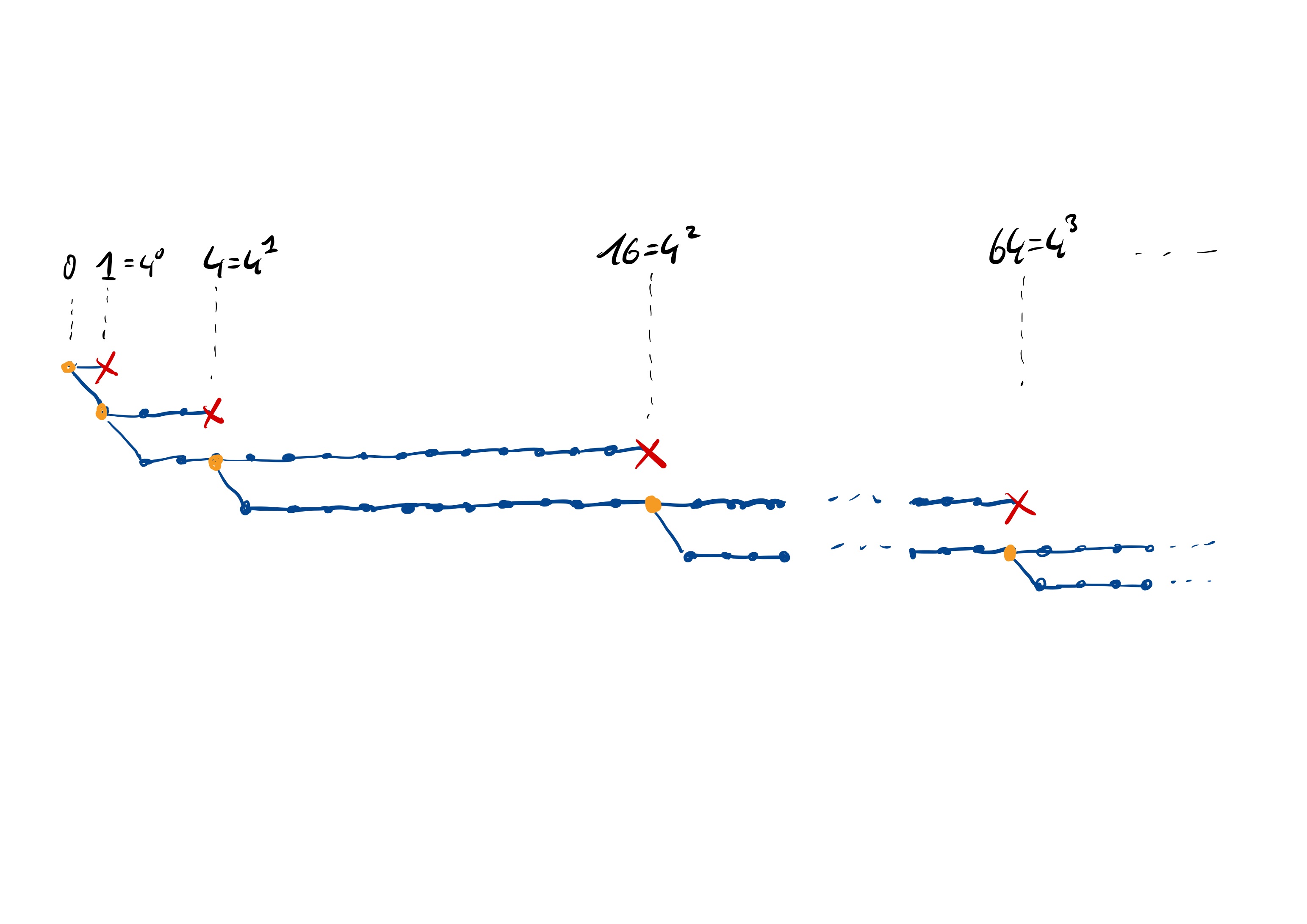}
        \caption{Schematic picture of the underlying graph of the branching Markov chain described in Section~\ref{sec:example_mutations}.}
        \label{fig:example_mutations}
    \end{figure}
    
    Consider the following branching Markov chain. Particles move on the graph sketched in Figure~\ref{fig:example_mutations}.
    At each step, a particle branches into 4 particles, all of which move one step from left to right along a blue edge. Particles arriving at a red cross are killed. At times $n$ of the form $n = 4^k$, $k\in\N_0$, particles arrive at crossroads (yellow dots). In the following step, every particle, when it moves, takes the downwards path with probability $8^{-n}$ and the other path with probability $1-8^{-n}$. 
    
    We claim that at every time $n$ of the form $n=4^k$, $n\in\N_0$, the expected number of particles is at least $2^n$. Indeed, it is true for $k=0$. Assume it is true at $n = 4^k$, for some $k$. Then, the expected number of particles at time $4^{k+1} = 4n$ is at least
    \[
    \frac{2^n}{8^n} 4^{3n}= 4^{2n} = 2^{4n}.
    \]
    The claim then follows by induction.
    
    It follows from the claim, and the fact that the number of particles grows by a factor of 4 at every time other than the times of the form $n=4^k$, that
    \[
   \liminf_{n\to\infty} (\expE[N_n])^{1/n} \ge 2 > 1.
   \]
   
   On the other hand, we trivially have at most $4^n$ particles at every time $n$. It follows that at every time $n$ of the form $n=4^k$, the probability that there exists a particle taking the downwards path is bounded by $4^n/8^n = 2^{-n}$. Hence, there exists almost surely a time of the form $n=4^k$ where no particle takes the downwards path. It follows that the process dies out almost surely. Hence, $\rho_c' \le 1$.

   We can also show that we have in fact $\liminf_{n\to\infty} (\expE[N_n])^{1/n} = 2$, and
   \[
     \rho''_{c,x} = \limsup_{n\to\infty} (\expE[N_n])^{1/n} = \lim_{n\to\infty} \left(2^n 4^{3n}\right)^{1/4n} = 2^{7/4} > 2.
   \]

   To sum up, we have in this example,
   \[
     \rho'_{c,x} \le 1 < 2 = \liminf_{n\to\infty} (\expE[N_n])^{1/n} < 2^{7/4} = \rho''_{c,x}.
   \]
   This shows that both inequalities in Theorem~\ref{th:rho'_rho''} may be strict.
   
   %
   %
   
   \subsection{Continuous time and space counterexample to $\lambda_1'=\lambda_1''$}\label{section:lambda' lambda'' can be different}
We recall that H. Berestycki and Rossi \cite{BerestyckiRossi} defined the following notion of generalised principal eigenvalue for a uniformly elliptic second-order differential operator $L$ on a smooth and \textit{unbounded} domain $E\subseteq \Rm^d$,
\begin{equation}\label{eq:Berestycki-Rossi eigenvalues example section}
\begin{split}
\lambda_1'=\lambda_1'(L):=\sup\{\lambda\in\Rm:\exists u\in W^{2,N}_{\text{loc}}(E)\cap L^{\infty}(E)\text{ such that $(L-\lambda)u\geq 0,$}\\\text{$u(x)>0$ for all $x\in E$, $u(x)\ra 0$ as $x\ra \xi$ for all $\xi\in \partial E$}\},\\
\lambda_1''=\lambda_1''(L):=\inf\{\lambda\in\Rm:\exists u\in W^{2,N}_{\text{loc}}(E)\text{ such that $(L-\lambda)u\leq 0$, $\inf_{x\in E}u(x)>0$}\}.
\end{split}
\end{equation}
Note that the above definition differs from the definition given in \cite{BerestyckiRossi} by a factor of two. They established general conditions under which $\lambda_1'\leq \lambda_1''$, and conjectured that they are equal. We will provide here a counterexample demonstrating that in fact this may not be.

We inductively define
\[
\begin{split}
S_0:=0,\quad a_1:=1,\quad S_1:=2a_1,\\ a_{n+1}:=4S_n,\quad S_{n+1}:=S_n+2a_{n+1},\quad n\geq 1.
\end{split}
\]
This gives
\[
S_n=2*9^{n-1},\quad S_n+a_{n+1}=10*9^{n-1},\quad n\geq 1.
\]
This then defines the sequence of intervals $([S_n,S_n+a_{n+1}])_{n=0}^{\infty}$ and $([S_n+a_{n+1},S_{n+1}])_{n=0}^{\infty}$. We then define
\[
A:=\cup_{n\geq 0}[S_n,S_n+a_{n+1}),\quad B:=\cup_{n\geq 0}[S_n+a_{n+1},S_{n+1}).
\]
For $\sigma>0$ we define
\[
L_0^{\sigma}:=\begin{cases}
\frac{\sigma}{2}\Delta+\partial_x+1,\quad x\in A,\\
\frac{\sigma}{2}\Delta+\partial_x-1,\quad x\in B,\\
\frac{\sigma}{2}\Delta+\partial_x,\quad x<1.
\end{cases}
\]
Note that for any $\sigma>0$ this is a uniformly elliptic second-order differential operator with bounded coefficients and smooth (in fact constant) second-order coefficients, exactly the setting considered by Berestycki-Rossi.

We prove the following.
\begin{prop}\label{prop:lambda' neq lambda''}
We have the inequality
\begin{equation}\label{eq:upper bound lambda'}
    \limsup_{\sigma\ra 0}\lambda'(L_0^{\sigma})\leq 0
\end{equation}
Conversely
\begin{equation}\label{eq:lower bound lambda''}
    \lambda''(L_0^\sigma)\geq \frac{4}{5}
\end{equation}
for all $\sigma>0$. Therefore $\lambda'(L_0^{\sigma})< \lambda''(L_0^\sigma)$ for $\sigma>0$ sufficiently small.
\end{prop}

\subsubsection*{Proof of Proposition \ref{prop:lambda' neq lambda''}}
We will firstly provide a Feynman-Kac formula, which provides the key tool for estimating the generalised eigenvalues. 

\subsubsection*{Feynman-Kac formula}

For arbitrary $\kappa\in \Rm$ we define $L^{\sigma,\kappa}:=L^{\sigma}_0+\kappa$. Given some choice of $\sigma>0$ and $\kappa\in \Rm$, we write as a shorthand
\[
\begin{split}
    \calU_{\sub}^{\sigma,\kappa}:=\{u\in W^{2,N}_{\text{loc}}(E)\cap L^{\infty}(E)\text{ such that $L^{\sigma,\kappa} u\geq 0,\;u(x)>0$ for all $x\in E$,}\},\\
\calU_{\super}^{\sigma,\kappa}:=\{ u\in W^{2,N}_{\text{loc}}(E)\text{ such that $L^{\sigma,\kappa}u\leq 0$, $\inf_{x\in E}u(x)>0$}\}.
\end{split}
\]

We define $X_t$ to be the diffusion
\[
X_t=\sigma B_t+t,
\]
where $B_t$ is standard Brownian motion started from $0$. 

Given a non-negative Borel-measurable function $u>0$ we define
\[
M_t[u]:=e^{\kappa t+\int_0^t\Ind(X_s\in A)-\Ind(X_s\in B)ds}u(X_t).
\]
We see that:
\begin{enumerate}
    \item[(i)]if $u\in \calU^{\sigma,\kappa}_\sub$ then $M_t[u]$ is a submartingale;
    \item[(ii)]if $u\in \calU^{\sigma,\kappa}_\super$ then $M_t[u]$ is a supermartingale.
\end{enumerate}

From this we see that:
\begin{enumerate}
    \item[(i)]if $u\in \calU^{\sigma,\kappa}_\sub$ and $t\in \Rm_{>0}$ then
    \begin{equation}\label{eq:counterexample upper bound on u when subharmonic 1}
    u(0)=M_0[u]\leq \expE_0[e^{\kappa t+\int_0^\Ind(X_s\in A)-\Ind(X_s\in B)ds}u(X_t)]\leq \expE_0[e^{\kappa t+\int_0^t\Ind(X_s\in A)-\Ind(X_s\in B)ds}]\lvert\lvert u\rvert\rvert_{\infty}.
    \end{equation}
    \item[(ii)]if $u\in \calU^{\sigma,\kappa}_\super$ and $t\in \Rm_{>0}$ then 
    \begin{equation}\label{eq:counterexample lower bound on u when superharmonic 1}
    u(0)=M_0[u]\geq \expE_0[e^{\kappa t+\int_0^t\Ind(X_s\in A)-\Ind(X_s\in B)ds}u(X_t)]\geq \expE_0[e^{\kappa t+\int_0^t\Ind(X_s\in A)-\Ind(X_s\in B)ds}]\inf_{x\in \Rm}u(x).
    \end{equation}
\end{enumerate}

\subsubsection*{Proof of \eqref{eq:lower bound lambda''}}

If $u\in \calU^{\sigma,\kappa}_{\super}$ then it follows from \eqref{eq:counterexample upper bound on u when subharmonic 1} and Jensen's inequality that
    \begin{equation}
    u(0)\geq \exp[\expE_0[\kappa t+\int_0^t\Ind(X_s\in A)-\Ind(X_s\in B)ds]]\inf_{x\in \Rm}u(x).
    \end{equation}
This then implies that
\begin{equation}\label{eq:upper bound 2 u(0) subharmonic}
u(0)\geq \limsup_{n\ra\infty}\exp\Big[(S_n+a_{n+1})\Big(\kappa +\expE_0\Big[\frac{1}{S_n+a_{n+1}}\int_0^{{S_n+a_{n+1}}}\Ind(X_s\in A)-\Ind(X_s\in B)ds\Big]\Big)\Big]\inf_{x\in \Rm}u(x).
\end{equation}
Our goal is to show that the limit supremum on the right-hand side is infinite for $\kappa>-\frac{4}{5}$, from which we conclude a contradiction, implying that $\calU^{\sigma,\kappa}_{\super}$ is empty for $\kappa>-\frac{4}{5}$.
    
When we rescale time and space by $\frac{1}{S_n+a_{n+1}}$, we see that the Brownian motion term vanishes as $n\ra\infty$, from which we see that
\[
\begin{split}
\frac{1}{S_n+a_{n+1}}\int_0^{S_n+a_{n+1}}\Ind(X_s\in A)-\Ind(X_s\in B)ds-\frac{1}{S_n+a_{n+1}}\int_0^{S_n+a_{n+1}}\Ind(s\in A)-\Ind(X_s\in B)ds\ra 0
\end{split}
\]
in probability as $n\ra\infty$. Now we have
\[
\int_0^{S_n}\Ind(s\in A)-\Ind(X_s\in B)ds=0
\]
for all $n\geq 1$. Therefore
\[
\frac{1}{S_n+a_{n+1}}\int_0^{S_n+a_{n+1}}\Ind(s\in A)-\Ind(X_s\in B)ds=\frac{a_{n+1}}{S_n+a_{n+1}}=\frac{4}{5}.
\]

It follows that
\[
\begin{split}
\frac{1}{S_n+a_{n+1}}\int_0^{S_n+a_{n+1}}\Ind(X_s\in A)-\Ind(X_s\in B)ds\ra \frac{4}{5}
\end{split}
\]
in probability as $n\ra\infty$. Therefore if $\kappa>-\frac{4}{5}$ then
\[
\limsup_{n\ra\infty}\Big(\kappa +\expE_0\Big[\frac{1}{S_n+a_{n+1}}\int_0^{{S_n+a_{n+1}}}\Ind(X_s\in A)-\Ind(X_s\in B)ds\Big]\Big)>0.
\]
This implies that the limit-supremum on the right-hand side of \eqref{eq:upper bound 2 u(0) subharmonic} is infinite for $\kappa>-\frac{4}{5}$. We therefore conclude that $\calU^{\sigma,\kappa}_{{\super}}$ is empty for $\kappa>-\frac{4}{5}$. Therefore
\[
\lambda''(L^\sigma_0)+\kappa=\lambda''(L^\sigma_0+\kappa)\geq 0
\]
for $\kappa>-\frac{4}{5}$. Therefore
\[
\lambda''(L^\sigma_0)\geq \frac{4}{5},
\]
whence we have established \eqref{eq:lower bound lambda''}.

\subsubsection*{Proof of \eqref{eq:upper bound lambda'}}

We now turn to the upper bound on $\lambda_1'(L^{\sigma}_0)$. We rescale $Y_t$ to obtain
\[
Y^n_t:=\frac{X_{S_nt}}{S_n}=t+\frac{\sigma}{S_n}\tilde{B}_t.
\]
where $\tilde{B}_t$ is the Brownian motion
\[
\tilde{B}_t:=\frac{1}{\sqrt{S_n}}B_{S_n t}.
\]
Note that whilst $\tilde{B}_t$ does depend on $n$, it doesn't distributionally, so we will not denote the dependence of $\tilde{B}_t$ upon $n$ as we will only ever need that it's distributed as a standard Brownian motion.

We define the set
\[
\calA:=\cup_{n\geq 1}[\frac{1}{9^n},\frac{5}{9^n}). 
\]
We further define for $c,\epsilon>0$ the sets
\[
\calA_c:=[0,c)\cup \calA,\quad  \calA_{c,\epsilon}:=[0,c)\cup \{x\in \Rm:d(x,\calA)\leq \epsilon\}.
\]
We observe that for all $c>0$ there exists $n_c$ dependent only upon $c$ such that for all $\epsilon>0$ and all $n\geq n_c$, if 
\[
\{\frac{t}{S_n}:t\leq S_n,t\in A\}+B(0,\epsilon)\subseteq \calA_{c,\epsilon}. 
\]
Therefore on the event $\{\sup_{s\leq 1}\lvert 
s-Y^n_s\rvert\leq \epsilon\}$, it follows that
\[
\frac{1}{S_n}\Leb(\{t\leq S_n:X_t\in A\})\leq \Leb(\calA_{c,\epsilon}).
\]
We also observe that for all $c>0$ there exists $C_c<\infty$ dependent only upon $c$ such that
\[
\Leb(\calA_{c,\epsilon})\leq \Leb(\calA_c)+C_c\epsilon.
\]

We now suppose that we have $u$ such that$Lu\geq 0$, $u>0$ and $u$ is bounded. We recall that by \eqref{eq:counterexample upper bound on u when subharmonic 1} we have
    \begin{equation}\label{eq:counterexample upper bound on u when subharmonic 2}
    u(1)\leq \expE_0\left[\exp\left(S_n\left(\kappa +\frac{1}{S_n}\int_0^t\Ind(X_s\in A)-\Ind(X_s\in B)ds\right)\right)\right]\lvert\lvert u\rvert\rvert_{\infty}.
    \end{equation}
We would therefore like to upper bound this quantity. From the above discussion we have that
\[
\left(\kappa +\frac{1}{S_n}\int_0^t\Ind(X_s\in A)-\Ind(X_s\in B)ds\right)\leq \kappa+\Ind(\sup_{s\leq 1}\lvert 
s-Y^n_s\rvert> \epsilon)+2\Leb(\calA_{c,\epsilon})-1
\]
for all $n\geq n_c$. It follows that
\[
\begin{split}
\expE_0\left[\exp\left(S_n\left(\kappa +\frac{1}{S_n}\int_0^t\Ind(X_s\in A)-\Ind(X_s\in B)ds\right)\right)\right]\\\leq
e^{S_n(\kappa+2\Leb(\calA_c)+C_c\epsilon-1)}[1+e^{S_n}\Pm(\sup_{s\leq 1}\lvert s-Y^n_s\rvert\geq \epsilon)].
\end{split}
\]
Since $Y^n_s-s=\frac{\sigma}{\sqrt{S_n}}\tilde{B}_s$, we have that
\[
\Pm(\sup_{s\leq 1}\lvert s-Y^n_s\rvert>\epsilon)=\Pm(\sup_{s\leq 1}\lvert\tilde{B}_s\rvert>\frac{\sqrt{S_n}\epsilon}{\sigma})\leq 2\Pm(\sup_{s\leq 1}\tilde{B}_s>\frac{\sqrt{S_n}\epsilon}{\sigma}).
\]
The running maximum of Brownian motion is very well-known. In particular, writing $\Phi(x):=\Pm(N(0,1)\geq x)$ for the complementary CDF of a normal distribution, we have
\[
\Pm(\sup_{s\leq 1}\tilde{B}_s>\frac{\sqrt{S_n}\epsilon}{\sigma})=2\Phi(\frac{\sqrt{S_n}\epsilon}{\sigma}).
\]
Therefore
\[
\Pm(\sup_{s\leq 1}\lvert s-Y^n_s\rvert>\epsilon)\leq 4\Phi(\frac{\sqrt{S_n}\epsilon}{\sigma}).
\]
Putting this all together we obtain
\[
\begin{split}
\expE_0\left[\exp\left(S_n\left(\kappa +\frac{1}{S_n}\int_0^t\Ind(X_s\in A)-\Ind(X_s\in B)ds\right)\right)\right]\\\leq e^{S_n(\kappa+2\Leb(\calA_c)+2C_c\epsilon-1)}[1+4e^{S_n}\Phi(\frac{\sqrt{S_n}\epsilon}{\sigma})].
\end{split}
\]

We now use the upper bound
\[
\Phi(t)\leq \frac{1}{\sqrt{2\pi}t}e^{-\frac{t^2}{2}},\quad t>0,
\]
to obtain
\[
\begin{split}
\expE_0\left[\exp\left(S_n\left(\kappa +\frac{1}{S_n}\int_0^t\Ind(X_s\in A)-\Ind(X_s\in B)ds\right)\right)\right]\\
\leq e^{S_n(\kappa+2\Leb(\calA_c)+2C_c\epsilon-1)}\left[1+\frac{4\sqrt{\sigma}}{\sqrt{2\pi S_n}\epsilon}e^{S_n(1-\frac{\epsilon^2}{2\sigma^2})}\right].
\end{split}
\]

Therefore 
\[
\limsup_{n\ra\infty}\expE_0\left[\exp\left(S_n\left(\kappa +\frac{1}{S_n}\int_0^t\Ind(X_s\in A)-\Ind(X_s\in B)ds\right)\right)\right]=0.
\]
whenever $\sigma<\frac{\epsilon}{2}$ and $\kappa+2\Leb(\calA_c)+2C_c\epsilon-1<0$. By the same argument as in the proof of \eqref{eq:lower bound lambda''} (using \eqref{eq:counterexample upper bound on u when subharmonic 1} in place of \eqref{eq:counterexample lower bound on u when superharmonic 1}),
\[
\lambda_1'(L^{\sigma})=\kappa+\lambda_1'(L^{\sigma}_0)\leq 0
\]
whenever $\sigma<\frac{\epsilon}{2}$ and $\kappa+2\Leb(\calA_c)+2C_c\epsilon-1<0$.

By allowing $\kappa$ to vary, we conclude that
\[
\lambda_1'(L^{\sigma}_0)\leq 2\Leb(\calA_c)+2C_c\epsilon-1
\]
whenever $\sigma<\frac{\epsilon}{2}$. 

We finally observe that
\[
\begin{split}
\limsup_{\sigma\ra 0}\lambda_1'(L^{\sigma}_0)=\limsup_{c\ra 0}\limsup_{\epsilon\ra 0}\limsup_{\sigma\ra 0}\lambda_1'(L^{\sigma}_0)\leq \limsup_{c\ra 0}\limsup_{\epsilon\ra 0}\limsup_{\sigma\ra 0}[\Leb(\calA_c)+C_c\epsilon]\\
=\limsup_{c\ra 0}\limsup_{\epsilon\ra 0}[2\Leb(\calA_c)+2C_c\epsilon-1]=\limsup_{c\ra 0}2\Leb(\calA_c)-1=2\Leb(\calA)-1.
\end{split}
\]

We finally compute that $\Leb(\calA)=4\sum_{n\geq 1}9^{-n}=-\frac{1}{2}$, implying that
\[
\limsup_{\sigma\ra 0}\lambda_1'(L^{\sigma}_0)\leq 0.
\]
\qed

\subsection{Counterexample: if $\lambda_1' =0$ then both global extinction and global survival are possible}\label{section: example both global extinction and survival are possible at criticality}

\begin{prop}\label{prop:both possible at lambda' crit}
If $\lambda'(L)=0$ then both global extinction and global survival are possible.
\end{prop}

\begin{proof}[Proof of Proposition \ref{prop:both possible at lambda' crit}]
We define the function
\[
g(x):=\begin{cases}
    \frac{1}{\sqrt{x}},\quad x\geq 1,\\
    0,\quad x<1.
    \end{cases}
\]

We take spatial movement $dX_t=2dt+ dW_t$, and consider two possible killing/branching rates:
\begin{enumerate}
    \item killing rate $g$, and no branching;
    \item branching rate $g$, and no killing.
\end{enumerate}

The corresponding infinitesimal generators are given by
\[
L^{\pm}=L_0\pm g,
\]
whereby
\[
L_0=\frac{1}{2}\Delta+2\partial_x.
\]

We now prove that
\begin{equation}\label{eq:evalues all 0 counterexample both possible}
 \lambda_1'(L^+)=\lambda_1'(L^-) =0.
\end{equation}

We observe that:
\begin{enumerate}
    \item We clearly have global survival in the case with branching rate $g$ and no killing.
    \item We have global extinction with killing rate $g$ and no branching. This follows since we then have a single particle and
    \[
    \begin{split}
    \Pm_x(\glSurv)=\lim_{t\ra\infty}\Pm_x[\text{survival to time $t$}]\\
    =\expE[\Pm(\text{survival to time $t$}\lvert (W_s)_{0\leq s\leq t})]=\expE[e^{-\int_0^tg(X_s)ds}],
    \end{split}
    \]
    and we can straightforwardly see that $\int_0^tg(X_s)ds\ra \infty$ almost surely as $t\ra\infty$.
\end{enumerate}

We now note that if we add any small amount $\epsilon>0$ of branching to the second (pure killing) process, then we have global survival. Therefore $\lambda_1'(L^{-}+\epsilon)>0$ for any $\epsilon>0$ by Theorem \ref{theo:PDE global extinction e-value relationship}, hence $\lambda_1'(L^+)\geq \lambda_1'(L^-)\geq 0$. 

Conversely we suppose that we add a small amount $\epsilon>0$ of killing to the first branching process. Since $g\leq 1$ everywhere, the branching process is dominated by a branching Brownian motion (BBM) with constant drift $2$ and branch rate $1$. 

Since the asymptotic spreading speed of driftless BBM with unit branch rate is $\sqrt{2}$, BBM with drift $2$ and branch rate $1$ must undergo local extinction, i.e. 
\[
\Pm(\text{there is a particle in $[-R,R]$ at arbitrarily large times $t$})=0.
\]

Therefore our first BBM with branch rate $g(x)$ and constant killing rate $\epsilon$ must eventually only have particles within the set $\{x:g(x)<\frac{\epsilon}{2}\}$. Since a branching process with branch rate $\frac{\epsilon}{2}$ and killing rate $\epsilon$ must undergo extinction, we must have global extinction of the BBM with branch rate $g(x)$ and constant killing rate $\epsilon$. Therefore $\lambda(L^+-\epsilon)\leq 0$ for all $\epsilon>0$ by Theorem \ref{theo:PDE global extinction e-value relationship}, hence $\lambda_1'(L^-)\leq \lambda_1'(L^+)\leq 0$.

This concludes the proof of \eqref{eq:evalues all 0 counterexample both possible}, and hence of Proposition \ref{prop:both possible at lambda' crit}.
\end{proof}

\appendix

\section{Rigorous derivation of the Feynman-Kac formula and McKean's representation of the generalised FKPP equation}

\subsection{The Feynman-Kac formula}
Whereas there are many references for the statement that if $u$ solves Kolmogorov's backward equation in an appropriate sense, then it is given by the Feynman-Kac formula, it seems to be much harder to find references for the statement that solutions actually exist. One might try checking directly that the Feynman-Kac formula satisfies Kolmogorov's backward equation, but in order to do this we need to know it belongs to a suitable Sobolev space in order to apply Ito's lemma ($W^{2,1}_{d,loc}$ is needed to apply Ito's lemma, see \cite[Krylov, Controlled diffusions, p.122 and 47]{Krylov1980}). It is straightforward enough to check that the Feynamn-Kac formula belongs to the domain of the infinitesimal generator defined via Hille-Yosida (see \cite[Theorem 8.1.1]{Ksendal2013}), and hence the Feynman-Kac formula solves Kolmogorov's backward equation in an abstract sense. However since we want a solution in a PDE sense (with solutions belonging at least to the Sobolev space $W^{2,1}_{d,loc}$), we will need to appeal to PDE theory.

Given $u(x,t)$ defined on a subset of $\Rm^d\times \Rm$ and $\alpha\in (0,1)$, we say that
\[
u\in C^{2,\alpha}_{\loc}
\]
if $u(x,t)$ is locally twice differentiable in space, once in time, and $u$, $D_xu$, $D^2_xu$ and $D_tu$ are all locally H\"older continuous (exponent $\alpha$). 

\begin{assum}\label{assum:Feynman-Kac coefficient assumptions}
We will assume that
\[
a_{ij}(x),\quad b_i(x),\quad c(x,t),\quad \phi(x,t)
\]
are all bounded and locally H\"older continuous (with exponent $\alpha\in (0,1)$), with $a$ and $b$ not depending upon time $t$. We further assume that $a$ is symmetric and uniformly positive definite. We thereby define the differential operator
\[
L_0=a_{ij}\partial_{ij}+b_i\partial_i,\quad L=L_0+c.
\]
\end{assum}
We consider the diffusion $X_t$ associated to $L_0$, defined as follows. We take $\sigma$ such that $\sigma(x)\sigma^{\dagger}(x)=a(x)$. The diffusion $X_t$ evolves as a weak solution of the SDE
\begin{equation}\label{eq:solutions of SDE}
dX_t=b(x_t)dt+\sigma(X_t)dW_t,
\end{equation}
solutions of which are unique in law (by \cite[Chapter 7]{Stroock1997}). We write $\Pm_{x,t}$ for the probability measure under which $X_s$ is a solution of the above solution issued from time $t$ with $X_t=x$. 

We suppose that $D=B(x_{\ast},\epsilon)\times (t_0,t_1)$ for some $x_{\ast}\in \Rm^d$, $0\leq t_0<t_1<\infty$ and $\epsilon>0$. We define $\partial D_+:=\partial B(x_{\ast},\epsilon)\times (t_0,t_1)\cup B(x_{\ast},\epsilon)\times \{t_1\}$ (note that the boundary at time $t_0$ is not included. 

The key theorem we will need to prove is the following.
\begin{theo}\label{theo:parabolic Feynman-Kac}
We take an arbitrary bounded, measurable function $\psi$ on $\partial D$ and let $\tau:=\inf\{s>t:(X_s,s)\in \partial D\}$. Then
\[
u(x,t):=\expE_{x,t}\Big[\int_0^{\tau}e^{\int_0^{s}c(X_{u},u)du} \phi(X_s,s)ds+e^{\int_0^{\tau}c(X_{s},s)ds}\psi(X_{\tau},\tau)\Big]
\]
is a $C^{2,\alpha}_{\loc}$ classical solution of Kolmogorov's backward equation
\[
\partial_tu=-Lu-\phi
\]
on $D$.
\end{theo}

Classically, the Feynman-Kac formula is defined for a diffusion on the whole space, where we ``score'' the function at the terminal time. One can also stop the diffusion upon hitting a boundary, and score it as a function of the hitting location. For instance we will kill the diffusion upon contact with the boundary of $E$, in which case there is no ``score'' in the Feynman-Kac formula. The above theorem however, is local. Namely, however one defines the Feynman-Kac formula $u(x,t)$, one can place a parabolic cylinder around any point on the interior of the domain, define $\psi$ to be the value of $u$ on $\partial D$, and thereby (using also the strong Markov property) see that $u(x,t)$ is a solution of Kolmogorov's backward equation on that cylinder (and hence on the whole interior of the domain). 

We obtain the following elliptic version of Theorem \ref{theo:parabolic Feynman-Kac} as a corollary of it.
\begin{theo}\label{theo:elliptic Feynman-Kac}
We take an arbitrary bounded, measurable function $\psi$ on $\partial D$ and let $\tilde \tau:=\inf\{t>0:X_t\in \partial B(x_{\ast},\epsilon)\}$. We suppose that $c$ depends only upon $x\in D$. Then if
\[
v(x):=\expE_{x}\Big[\int_0^{\tilde\tau}e^{\int_0^{s}c(X_{u})du} \phi(X_s,s)ds+e^{\int_0^{\tilde\tau}c(X_{s})ds}\psi(X_{\tilde\tau},\tilde\tau)\Big]
\]
is bounded, it is a $C^{2,\alpha}_{\loc}$ classical solution of 
\[
Lv=-\phi
\]
on $B(x_{\ast},\epsilon)$.
\end{theo}

\subsubsection*{Proof of Theorem \ref{theo:parabolic Feynman-Kac}}
The first order of business is to check that $u$ is continuous on $D$. We will do this using the strong Feller property. Once the continuity of $u$ is established, we will be able to apply PDE theory on arbitrary sub-cylinders whose closure belongs to $D$. 

We let $P_t$ be the expectation semigroup for $X_t$ killed upon the boundary $\partial B(x_{\ast},\epsilon)$, i.e.
\[
P_tf(x):=\expE_{x,0}[f(X_t)\Ind(\tau_B>t)],
\]
where $\tau:=\inf\{t>0:X_s\notin B(x_{\ast},\epsilon)\}$. Since $P_t$ is a strongly continuous strong Feller semigroup, 
\begin{equation}\label{eq:continuity killed diffusion semigroup}
    D\times \Rm_{>0}\ni (x,t)\mapsto P_tf(x)
\end{equation}
is continuous for all bounded, Borel measurable $f$. 

We fix $(x_0,t_0)\in D$, and seek to show that $\limsup_{(x,t)\ra (x_0,t_0)}\lvert u(x,t)-u(x_0,t_0)\rvert=0$. We take $0<r<d(x_0,\partial B(x_{\ast},\epsilon))$ and $0<\delta<d(t_0,\{t_0,t_1\})$. We write 
\[
u_{\delta}(x):=u(x,t_0+\delta)\quad\text{and}\quad \delta t:=t_0+\delta-t
\]
for $t\in B(t_0,\delta)$. Then for $(x,t)\in B(x_0,r)\times B(t_0,\delta)$ we have by the strong Markov property
\[
\begin{split}
u(x,t)=\expE_{x,0}\Big[\int_0^{\tau\wedge \delta t}e^{\int_0^{s}c(X_{t+u},t+u)du} f(X_{t+s},t+s)ds+\Ind(\tau>\delta t)e^{\int_0^{\delta t}c(X_{s+u},s+u)ds}u_{\delta}(X_{\delta t})\\+\Ind(\tau\leq \delta t)e^{\int_0^{\tau}c(X_{s+u},s+u)ds}\psi(X_{t+\tau},t+\tau)\Big].
\end{split}
\]
On the other hand
\[
(P_{\delta t}u_{\delta})(x)=\expE_{x,0}\Big[\Ind(\tau>\delta t)u_{\delta}(X_{\delta t})\Big].
\]
Therefore
\[
\begin{split}
\lvert u(x,t)-(P_{\delta t}u_{\delta})(x)\rvert=\Big\lvert\expE_{x,0}\Big[\int_0^{\tau\wedge \delta t}e^{\int_0^{s}c(X_{t+u},t+u)du} f(X_{t+s},t+s)ds\\
+\Ind(\tau>\delta t)[e^{\int_0^{\delta t}c(X_{s+u},s+u)ds}-1]u_{\delta}(X_{\delta t})+\Ind(\tau\leq \delta t)e^{\int_0^{\tau}c(X_{s+u},s+u)ds}\psi(X_{t+\tau},t+\tau)\Big]\Big\rvert.
\end{split}
\]
Since $\sup_{x\in B(x_0,r)}\Pm_{x,0}(\tau\leq \delta)\ra 0$ as $\delta \ra 0$ and all the coefficients are bounded, we have
\[
\sup_{\substack{x\in B(x_0,r),\\ t\in B(t_0,\delta)}}\lvert u(x,t)-P_{\delta t}u_{\delta}(x)\rvert\ra 0
\]
as $\delta \ra 0$. Therefore by the triangle inequality
\[
\limsup_{r\ra 0}\limsup_{\delta\ra 0}\sup_{\substack{x\in B(x_0,r),\\ t\in B(t_0,\delta)}}\lvert u(x,t)-u(x_0,t_0)\rvert\leq \limsup_{r\ra 0}\limsup_{\delta\ra 0}\sup_{\substack{x\in B(x_0,r),\\ t\in B(t_0,\delta)}}\lvert P_{\delta t}u_{\delta}(x)-P_{\delta}u_{\delta}(x_0)\rvert.
\]
The latter is equal to $0$ by \eqref{eq:continuity killed diffusion semigroup}, hence we have established the continuity of $u$. 

We now consider an arbitrary subcylinder compactly contained in $D$, $D'=B(x_{\ast}',\epsilon')\subset\subset D$, and define $\partial D'$ similarly to before. Then by \cite[Theorem 9, p.69]{Friedman1964} there exists a $C^{2+\alpha}$ classical solution of $Lv=f$ on $D'$ with boundary conditions $v=u$ on $\partial D'$. We fix $(x_0,t_0)\in D'$, start a diffusion from $t=t_0$ and $X_{t}=x_0$, and define $\tau'$ to be the first hitting time of $\partial D'$. We observe that
\[
\int_{t_0}^{t\wedge \tau'}e^{\int_{t_0}^{t\wedge \tau'}c(X_{s},s)ds} f(X_{s\wedge \tau'},s\wedge \tau')ds+e^{\int_{t_0}^{t\wedge \tau'}c(X_{s},s)ds}v(X_{t\wedge \tau'},t\wedge \tau')
\]
is a $\Pm_{x_0,t_0}$-martingale by Ito's lemma and the fact that $v$ satisfies Kolmogorov's backward equation
\[
(\partial_t+L)v=-f.
\]
Therefore
\[
v(x_0,t_0)=\expE_{x_0,t_0}\Big[\int_{t_0}^{\tau'}e^{\int_{t_0}^{ \tau'}c(X_{s},s)ds} f(X_{s\wedge \tau'},s\wedge \tau')ds+e^{\int_{t_0}^{ \tau'}c(X_{s},s)ds}v(X_{ \tau'},  \tau')\Big]=u(x_0,t_0).
\]
We have established that $u\equiv v$ on $D'$, hence $u$ satisfies Kolmogorov's backward equation on $D'$, hence on all of $D$.
\qed

\subsubsection*{Proof of Theorem \ref{theo:elliptic Feynman-Kac}}
We take $t_0=0$ and $t_1=1$. Then $v(x)=u(x,t)$ for $t\in (0,1)$, where we impose the boundary condition $\psi(x,t)=v(x)$ on $\partial D$. It follows from Theorem \ref{theo:parabolic Feynman-Kac} that $u(x,t)$ is a $C^{2,1,\alpha}_{\loc}$ classical solution of Kolmogorov's backward equation on $D$. Since $u(x,t)$ doesn't depend upon $t$ for fixed $x\in B(x_{\ast},\epsilon)$, we obtain Theorem \ref{theo:elliptic Feynman-Kac}.
\qed

\subsection{McKean's representation of the generalised FKPP equation}

It is well-known since the work of Skorokhod \cite{Skorokhod1964} (see also the work of McKean \cite{McKean1975}) that branching Brownian motion is dual to the classical FKPP equation \cite{McKean1975}
\[
\partial_tu=\frac{1}{2}\Delta u+u(1-u).
\]
This connection is known to be valid in much greater generality (see e.g. \cite{Henry-Labordre2019}), however we couldn't find a proof in the literature under quite the assumptions on the coefficients we wish to impose here. Ikeda, Nagasawa and Watanabe established this in very great generality in \cite{Ikeda1966}, but there the FKPP equation is formulated in an abstract sense, rather than in the PDE sense we desire. We therefore provide a proof here.

We take non-negative, bounded and Borel measurable
\[
r(x),\; p_n(x), \quad n\geq 0,
\]
such that 
\[
\sum_np_n(x)\equiv 1\quad\text{and}\quad r(x)\sum_{n\geq 0}np_n(x)\quad\text{is bounded.}
\]

The branching process $\bar X_t=(X^1_t,\ldots,X^{N_t}_t)$ is defined as follows. Each particle $X^i_t$ evolves in between branching and killing events independently as a weak solution of \eqref{eq:solutions of SDE}. Each particle in our branching process branches locally (i.e. the children are born at the parents' location) at rate $r(x)$ with offspring distribution $(p_n(x))_{n=0}^{\infty}$, so that $n$ offspring produced independently with probability $p_n(x)$ and the parent particle is killed. If $n=0$ offspring are produced, this corresponds to \textit{soft killing}. The process is defined on $E$, assumed to be a non-empty, open subset $E$ of $\Rm^d$ with locally Lipschitz boundary $\partial E$. 

We will associate this branching process with $C^{2,1,\beta}_{\loc}$ (for some $\beta\in (0,1)$) classical solutions of the PDE
\begin{equation}\label{eq:FKPP equation BBM representation theorem}
\partial_t y=L_0y+f(y,x),
\end{equation}
whereby we define
\begin{equation}
f(y,x):=r(x)\sum_np_n(x)[y^n-y].
\end{equation}

The function is said to satisfy Dirichlet boundary conditions if $y(x,t)$ vanishes continuously along $\partial E\times (0,\infty)$. Given a continuous function $g(x)$, it is said to have initial condition $g(x)$ if $y(x,t)-g(x)$ vanishes continuously along $E\times \{0\}$.

\begin{theo}\label{theo:McKean rep}
\begin{enumerate}
    \item Given any $[0,1]$-valued continuous function $g$, there is a unique solution of \eqref{eq:FKPP equation BBM representation theorem} with Dirichlet boundary conditions and initial condition $g$.
    \item Given a $[0,1]$-valued Borel-measurable function $g$,
    \[
    u(x,t):=\expE\Big[\prod_{i=1}^{N_t}g(X^i_t)\Big]
    \]
    satisfies \eqref{eq:FKPP equation BBM representation theorem} with Dirichlet boundary conditions on $E$. If $g$ is continuous then this solution of \eqref{eq:FKPP equation BBM representation theorem} has initial condition $g$.
\end{enumerate}

\end{theo}

\subsubsection*{Proof of Theorem \ref{theo:McKean rep}}
As with the proof of Theorem \ref{theo:parabolic Feynman-Kac}, the main difficulty is to do with the existence of classical solutions of the PDE. To that end we will prove the following proposition.
\begin{prop}\label{prop:classical solution semilinear PDE}
Suppose that $g:E\ra \Rm$ is continuous and $\tilde f:\Rm\times E\ra \Rm_{\geq 0}$ is bounded, non-negative, locally H\"older continuous and Lipschitz in the first variable uniformly over the spatial variable, i.e. 
\[
\sup_{x\in E}\sup_{u_1\neq u_2}\frac{\lvert f(u_2,x)-f(u_1,x)\rvert}{\lvert u_2-u_1\rvert}<\infty.
\]
We further assume that $\tilde{f}(x,0)\equiv 0$. Then for some $\beta\in (0,1)$ there exists a $C^{2,\beta}_{\loc}$ classical solution of
\begin{equation}\label{eq:general FKPP PDE}
\begin{split}
\partial_ty=L_0y+\tilde f(y,x),\quad x\in E,\; t>0,\\
y(x,0)=g(x),\quad x\in E,\quad y(x,t)=0,\quad t>0.
\end{split}
\end{equation}
\end{prop}
\begin{proof}[Proof of Proposition \ref{prop:classical solution semilinear PDE}]
We observe that solutions of \eqref{eq:general FKPP PDE} are equivalent to those of
\begin{equation}\label{eq:reformulation of general FKPP PDE}
    \begin{split}
\partial_tz=L_0z+e^{Ct}\tilde f(e^{-Ct}z,x)+Cz,\quad x\in E,\; t>0,\\
z(x,0)=g(x),\quad x\in E,\quad z(x,t)=0,\quad t>0,
\end{split}
\end{equation}
via the transform $z=e^{Ct}y$, for any $C\in \Rm_{\geq 0}$. We choose $0<C<\infty$ to be larger than the uniform Lipschitz constant of $\tilde{f}$ in its first variable, so that
\[
z\mapsto e^{Ct}f(e^{-Ct}z,x)+Cz
\]
is positive and non-decreasing in $z$ for all $t\geq 0$ and $x\in E$. 

Our goal now is to establish the existence of classical solutions of \eqref{eq:reformulation of general FKPP PDE}.

We inductively define $z^{(n)}$ as follows. We take $z^{(0)}\equiv 0$. We inductively claim we have $z^{(n)}$ for $n\geq 1$ satisfying the following:
\begin{enumerate}
    \item $z^{(n)}\geq z^{(n-1)}$ everywhere if $n\geq 1$;
    \item for $n\geq 1$, $z^{(n)}$ is a $C^{2,\gamma}_{\loc}$ (for some $\gamma\in (0,1)$) classical solution of
    \begin{equation}\label{eq:inductive PDE solution in n} 
  \begin{split}
\partial_tz^{(n)}=L_0z^{(n)}+e^{Ct}\tilde f(e^{-Ct}z^{(n-1)},x)+Cz^{(n-1)},\quad x\in E,\; t>0,\\
z^{(n)}(x,0)=g(x),\quad x\in E,\quad z(x,t)=0,\quad t>0,
\end{split}
    \end{equation}  
    for all $n\geq 1$.
\end{enumerate}
The second condition is clearly satisfied for $n=0$. We now suppose we have established the induction hypothesis for some $n\geq 0$, and seek to verify it for $n+1$.

Since $z^{(n)}$ is $C^{2}_{\loc}$, $e^{Ct}\tilde f(e^{-Ct}z^{(n)},x)+Cz$ is locally $\gamma$-Holder continuous for some $\gamma$. Therefore Theorem \ref{theo:parabolic Feynman-Kac} provides for the existence of a $C^{2,\gamma}_{\loc}$ classical solution $z^{(n+1)}$ of \eqref{eq:inductive PDE solution in n}, which is uniformly bounded over bounded time intervals (but at this point the bound may depend upon $n<\infty$). Since $z^{(n)}\geq z^{(n-1)}$, 
\[
e^Ct\tilde{f}(e^{-Ct}z^{n}(x,t),x)+Cz^{(n)}\geq e^Ct\tilde{f}(e^{-Ct}z^{n-1}(x,t),x)+Cz^{(n-1)},
\]
so the parabolic maximum principle implies that $z^{(n+1)}\geq z^{(n)}$. 

We now show that $z^{(n)}$ is bounded uniformly over $n<\infty$ over any given time horizon. For any $n$ we can solve the linear equation
\begin{equation}\label{eq:linearised equation fixed n}
\begin{split}
\partial_tz=L_0z+e^{Ct}\tilde f(e^{-Ct}z^{(n-1)},x)+Cz,\quad x\in E,\; t>0,\\
z(x,0)=g(x),\quad x\in E,\quad z(x,t)=0,\quad t>0,
\end{split}
\end{equation}
This is bounded uniformly in $n$ over any time horizon (say by $C_T<\infty$) by the boundedness of $\tilde f$ and Theorem \ref{theo:parabolic Feynman-Kac}. On the other hand we see that
\[
\begin{split}
\partial_t(z-z^{(n)})=L_0(z-z^{(n)})+C(z-z^{(n-1)}),\quad x\in E,\; t>0,\\
(z-z^{(n)})(x,0)=0,\quad x\in E,\quad (z-z^{(n)})(x,t)=0,\quad t>0,
\end{split}
\]
Since $z^{(n)}$ is non-decreasing in $n$, 
\[
\partial_t(z-z^{(n)})=L_0(z-z^{(n)})+C(z-z^{(n-1)})\geq L_0(z-z^{(n)})+C(z-z^{(n)}).
\]
Therefore the parabolic maximum principle implies that $z\geq z^{(n)}$ everywhere. This implies that $z^{(n)}(x,t)\leq C_T$ for all $n<\infty$, $x\in E$ and $t\leq T$. Therefore $z^{(n)}$ converges pointwise to some limit $z^{(\infty)}$, which is bounded by $C_T$ over any time horizon $T<\infty$. 

We observe that $z^{(n)}$ satisfies
\[
\partial_tz^{(n)}=L_0z^{(n)}+\frac{e^{Ct}\tilde f(e^{-Ct}z^{(n-1)},x)+Cz^{(n-1)}}{z^{(n)}}z^{(n)}.
\]
On the other hand
\[
0\leq \frac{e^{Ct}\tilde f(e^{-Ct}z^{(n-1)},x)+Cz^{(n-1)}}{z^{(n)}}\leq \frac{e^{Ct}\tilde f(e^{-Ct}z^{(n)},x)+Cz^{(n)}}{z^{(n)}},
\]
which is bounded uniformly in $n$ since $\tilde{f}(u,x)\equiv 0$ and $\tilde{f}$ is uniformly Lipschitz in the second variable. We can therefore apply the Krylov-Safonov theorem \cite{Krylov1981} to see that $z^{(n)}$ is $\beta$-Holder continuous for some $\beta\in (0,1)$, uniformly in $n$ (over any fixed compact subset of $E\times \Rm_{>0}$). Therefore the limit $z^{(\infty)}$ is locally $\beta$-Holder continuous on $E\times \Rm_{>0}$.

Theorem \ref{theo:parabolic Feynman-Kac} tells us that
\[
z^{(n)}(x,t)=\expE_x\Big[\int_0^{t\wedge \tau}e^{Ct}\tilde{f}(e^{-Ct}z^{(n-1)}(X_s,t-s),X_s)+Cz^{(n-1)}(X_s,t-s)ds+g(X_t)\Ind(\tau>t)\Big],
\]
where $\tau_E$ is the exit time of the diffusion $X_t$ from $E$. Then by taking the $n\ra\infty$ limit and applying the monotone convergence theorem, we see that
\[
z^{(\infty)}(x,t)=\expE_x\Big[\int_0^{t\wedge \tau}e^{Ct}\tilde{f}(e^{-Ct}z^{(\infty)}(X_s,t-s),X_s)+C\tilde{z}^{(\infty)}(X_s,t-s)ds+g(X_t)\Ind(\tau>t)\Big].
\]
It follows from Theorem \ref{theo:parabolic Feynman-Kac} that $z^{(\infty)}$ is a $C^{2,\beta}_{\loc}$ classical solution of \eqref{eq:reformulation of general FKPP PDE}, and therefore there exists a $C^{2,\beta}_{\loc}$ classical solution of \eqref{eq:general FKPP PDE}.\end{proof}

Having established Proposition \ref{prop:classical solution semilinear PDE}, we now use it to establish Theorem \ref{theo:McKean rep}. We suppose that $u$ is a classical solution of \eqref{eq:FKPP equation BBM representation theorem} and $\bar X_t$ is the branching process corresponding to $f$. Then we claim that for any $T<\infty$,
\[
M_s:=\prod_{i=1}^{N_s}u(X_s,T-s)
\]
is a martingale. By Ito's lemma (which we can apply by the smoothness of $u$),
\[
\begin{split}
dM_s=\sum_{i=1}^{N_s}\prod_{j\neq i}u(X^j_s,T-s)\Big[Lu(X_s,T-s)-\partial_tu(X_s,T-s)\\+r(X_s)\sum_np_n(X_s)[u^n(X_s,T-s)-u(X_s,T-s)]\Big]=0.
\end{split}
\]
Here $\partial_tu$ denotes the partial derivative of $u$ in its second (time) variable.

Since $M_t$ is bounded, it is a martingale. Therefore
\[
u(x,t)=\expE_x\Big[\prod_{i=1}^{N_t}u(X^i_t,0)\Big],
\]
so long as $u$ is continuous at the initial condition. This provides for Part 1 of Theorem \ref{theo:parabolic Feynman-Kac}, and Part 2 for continuous $g$. To extend Part 2 to discontinuous $g$, we take $0<s<t$, and use the branching property to see that
\[
u(x,t):=\expE_x\Big[\prod_{i=1}^{N_t}g(X^i_t)\Big]=\expE_x\Big[\expE\Big[\prod_{i=1}^{N_t}g(X^i_t)\Big\lvert \mathcal{F}_s\Big]\Big]=\expE_x\Big[\prod_{i=1}^{N_s}u(X^i_s,t-s)\Big].
\]
We then apply the strong Feller property to see that $u$ is necessarily continuous at positive times, and apply Part 2 for continuous initial conditions.

\qed

\section*{Acknowledgements}

We are grateful to Henri Berestycki for interesting discussions at the CIRM conference `Non-local branching processes' in September 2024.

Pascal Maillard acknowledges partial support from Institut Universitaire de France, the MITI interdisciplinary program 80PRIME GEx-MBB and the ANR MBAP-P (ANR-24-CE40-1833) project. The work of Oliver Tough was partially supported by the EPSRC MathRad programme grant EP/W026899/.

Part of this work of conducted while one of the authors (PM) was in
residence at the Simons Laufer Mathematical Sciences Institute in
Berkeley, California, during the Spring 2025 semester, partially funded by the National Science
Foundation under Grant No. DMS-1928930.

\bibliographystyle{alpha}
\bibliography{global_survival}

\end{document}